\documentclass[review]{elsarticle}
\usepackage{lineno,hyperref}
\modulolinenumbers[2]
\usepackage{amsmath}
\usepackage{amssymb}
\usepackage{amsfonts}
\usepackage{color}
\usepackage{soul}

\soulregister\cite7
\soulregister\ref7
\soulregister\eqref7

\newtheorem{theorem}{Theorem}[section]
\renewcommand{\theequation}{\thesection.\arabic{equation}}

\newtheorem{assumption}[theorem]{Assumption}

\newtheorem{case}[theorem]{Case}

\newtheorem{lemma}[theorem]{Lemma}

\newtheorem{proposition}[theorem]{Proposition}
\newtheorem{remark}[theorem]{Remark}

\numberwithin{equation}{section}
\newenvironment{proof}[1][Proof]{\noindent\textit{#1.}}{\hfill \rule{0.5em}{0.5em}}
\usepackage{paralist,graphics,epsfig,graphicx,epstopdf,mathrsfs}
\usepackage{geometry}
\begin{document}

\begin{frontmatter}

\title{The spatially homogeneous Hopf bifurcation induced jointly by memory and general delays in a diffusive system}
\author[1]{Yehu Lv\corref{mycorrespondingauthor}}
\ead{mathlyh@163.com}
\address[1]{School of Mathematical Sciences, Beijing Normal University, Beijing 100875, China}
\cortext[mycorrespondingauthor]{Corresponding author.}

\begin{abstract}
In this paper, by incorporating the general delay to the reaction term in the memory-based diffusive system, we propose a diffusive system with memory delay and general delay (e.g., digestion, gestation, hunting, migration and maturation delays, etc.). We first derive an algorithm for calculating the normal form of Hopf bifurcation in the proposed system. The developed algorithm for calculating the normal form of Hopf bifurcation can be used to investigate the direction and stability of Hopf bifurcation. As a real application, we consider a diffusive predator-prey model with ratio-dependent Holling type-\uppercase\expandafter{\romannumeral3} functional response, which includes with memory and gestation delays. The Hopf bifurcation analysis without gestation delay is first studied, then the Hopf bifurcation analysis with memory and gestation delays is studied. By using the developed algorithm for calculating the normal form of Hopf bifurcation, the supercritical and stable spatially homogeneous periodic solutions induced jointly by memory and general delays are found. The stable spatially homogeneous periodic solutions are also found by the numerical simulations which confirms our analytic result.
\end{abstract}


\end{frontmatter}

\linenumbers

\section{Introduction}
\label{sec:1}

In many mathematical modeling of specific disciplines, such as physics, chemistry and biology \cite{lv1,lv2,lv3}, the reaction-diffusion equations have been widely used. In general, the reaction-diffusion equations are based on the Fick's law, that is the movement flux is in the direction of negative gradient of the density distribution function \cite{lv4}. The diffusion term based on the Fick's law is usually called as the random diffusion driven by inherent mechanism.

Satulovsky et al. \cite{lv5} have proposed a stochastic lattice gas model to describe the dynamics of a predator-prey system. More precisely, the authors proposed a model which can be seen as a system consisting of two interacting particles residing in the site of a lattice. One type of particle represents a prey and the other a predator. Each site can be either empty, occupied by one prey, or occupied by one predator. Tsyganov et al. \cite{lv6} have considered a predator-prey system with cross-diffusion, and they found a new type of propagating wave in this system. The authors called it as "taxis" wave, which is entirely different from wave in a predator-prey system with self-diffusion. More precisely, they found that unlike the typical reaction-diffusion wave, which annihilate on collision, the "taxis" wave can often penetrate through each other and reflect from impermeable boundaries. McKane et al. \cite{lv7} have described the predator-prey system using an individual level model, and they focused on modeling the phenomenon of cycles. They think that the phenomenon of cycles involves concepts such as resonance. Carlos et al. \cite{lv8} have pointed out that a standard paradigm of condensed matter physics involves the interaction of discrete entities positioned on the sites of a regular lattice which can be described by a differential equation after coarse-graining when observed at a macroscopic scale, and they used a simple diffusive predator-prey model to predict that predator and prey numbers oscillate in time and space. Moreover, the diffusion-advection systems have been studied by many scholars, such as the chemotaxis model \cite{lv9,lv10,lv11,lv12,lv13}, the predator-prey model with prey-taxis \cite{lv14,lv15,lv16,lv17,lv18}, the predator-prey model with indirect prey-taxis \cite{lv19,lv20,lv21}, the competition-diffusion-advection model in the river environment \cite{lv22,lv23,lv24} and the reaction-diffusion-advection population model with delay in reaction term \cite{lv25}. However, the animal movements are different from the chemical movements, especially for highly developed animals, because they can even remember the historic distribution or clusters of the species in space. Therefore, in order to include the episodic-like spatial memory of animals, Shi et al. \cite{lv4} proposed a modified Fick's law that in addition to the negative gradient of the density distribution function at the present time, there is a directed movement toward the negative or positive gradient of the density distribution function at past time, and they proposed the following diffusive model with spatial memory
\begin{eqnarray}\left\{\begin{aligned}
&\frac{\partial u(x,t)}{\partial t}=d_{1}\Delta u(x,t)+d_{2}\left(u(x,t)u_{x}(x,t-\tau)\right)_{x}+f\left(u(x,t)\right), & x \in \Omega,~t>0, \\
&\frac{\partial u}{\partial \mathbf{n}}(x,t)=0, & x \in \partial \Omega,~t>0, \\
&u(x,t)=u_{0}(x,t), & x \in \Omega,~-\tau \leq t\leq 0,
\end{aligned}\right.\end{eqnarray}
where $u(x,t)$ is the population density at the spatial location $x$ and at time $t$, $d_{1}$ and $d_{2}$ are the Fickian diffusion coefficient and the memory-based diffusion coefficient, respectively, $\Omega \subset \mathbb{R}$ is a smooth and bounded domain, $u_{0}(x,t)$ is the initial function, $\Delta u(x,t)=\partial^{2}u(x,t)/\partial x^{2}$, $u_{x}(x,t)=\partial u(x,t)/\partial x$, $u_{x}(x,t-\tau)=\partial u(x,t-\tau)/\partial x$, $u_{xx}(x,t-\tau)=\partial^{2} u(x,t-\tau)/\partial x^{2}$, and $\mathbf{n}$ is the outward unit normal vector at the smooth boundary $\partial \Omega$. Here, the time delay $\tau>0$ represents the averaged memory period, which is usually called as the memory delay, and $f\left(u(x,t)\right)$ describes the chemical reaction or biological birth and death. Notice that such movement is based on the memory (or history) of a particular past time density distribution. However, by the stability analysis, they found that the stability of the positive constant steady state fully depends on the relationship between the diffusion coefficients $d_{1}$ and $d_{2}$, but is independent of the memory delay. In order to further investigate the influence of memory delay on the stability of the positive constant steady state, Shi et al. \cite{lv26} studied the spatial memory diffusion model with memory and maturation delays
\begin{eqnarray*}\left\{\begin{aligned}
&\frac{\partial u(x,t)}{\partial t}=d_{1}\Delta u(x,t)+d_{2}\left(u(x,t)u_{x}(x,t-\tau)\right)_{x}+f\left(u(x,t),u(x,t-\sigma\right), & x \in \Omega,~t>0, \\
&\frac{\partial u}{\partial \mathbf{n}}(x,t)=0, & x \in \partial \Omega,~t>0,
\end{aligned}\right.\end{eqnarray*}
where $\sigma>0$ is the maturation delay. They found that memory-based diffusion with memory and maturation delays can induce more complicated spatiotemporal dynamics, such as spatially homogeneous and inhomogeneous periodic solutions.

By introducing the non-local effect to the memory-based diffusive system (1.1), Song et al. \cite{lv27} proposed the single population model with memory-based diffusion and non-local interaction
\begin{eqnarray*}\left\{\begin{aligned}
&\frac{\partial u(x,t)}{\partial t}=d_{1}\Delta u(x,t)+d_{2}\left(u(x,t)u_{x}(x,t-\tau)\right)_{x}+f(u(x,t),\widehat{u}), & x \in \Omega,~t>0, \\
&\frac{\partial u}{\partial \mathbf{n}}(x,t)=0, & x \in \partial \Omega,~t>0,
\end{aligned}\right.\end{eqnarray*}
where $\Omega=(0,\ell\pi)$ with $\ell\in\mathbb{R}^{+}$, $\widehat{u}=(1/\ell\pi)\int_{0}^{\ell \pi}u(y,t)dy$. Many complicated spatiotemporal dynamics are found, such as the stable spatially homogeneous or inhomogeneous periodic solutions, homogeneous or inhomogeneous steady states, the transition from one of these solutions to another, and the coexistence of two stable spatially inhomogeneous steady states or two spatially inhomogeneous periodic solutions near the Turing-Hopf bifurcation point. Recently, for the single-species model with spatial memory, Song et al. \cite{lv28} studied the memory-based movement with spatiotemporal distributed delays in diffusion and reaction terms.

In addition, Song et al. \cite{lv29} considered the following resource-consumer model with random and memory-based diffusions
\begin{eqnarray}\left\{\begin{aligned}
&\frac{\partial u(x,t)}{\partial t}=d_{11}\Delta u(x,t)+f\left(u(x,t),v(x,t)\right), & x \in \Omega,~t>0, \\
&\frac{\partial v(x,t)}{\partial t}=d_{22}\Delta v(x,t)-d_{21}\left(v(x,t)u_{x}(x,t-\tau)\right)_{x}+g\left(u(x,t),v(x,t)\right), & x \in \Omega,~t>0, \\
&u_{x}(0,t)=u_{x}(\ell\pi,t)=v_{x}(0,t)=v_{x}(\ell\pi,t)=0, & t \geq 0, \\
&u(x,t)=u_{0}(x,t),~v(x,t)=v_{0}(x,t), & x \in \Omega,~-\tau \leq t \leq 0,
\end{aligned}\right.\end{eqnarray}
where $u(x,t)$ and $v(x,t)$ are the densities of resource and consumer, respectively, $d_{11} \geq 0$ and $d_{22} \geq 0$ are the random diffusion coefficients, $d_{21} \geq 0$ is the memory-based diffusion coefficient, $v_{0}(x,t)$ is also the initial function, and $f\left(u(x,t),v(x,t)\right)$ and $g\left(u(x,t),v(x,t)\right)$ are the reaction terms. The well-posedness of solutions is studied, and the rich dynamics of the system (1.2) with Holling type-\uppercase\expandafter{\romannumeral1} or type-\uppercase\expandafter{\romannumeral2} functional responses are found. Notice that by comparing with the classical reaction-diffusion systems with delay, the system (1.2) has the two main differences, one is that the memory delay appears in the diffusion term, another is that the diffusion term is nonlinear. Thus, the normal form for Hopf bifurcation in the classical reaction-diffusion systems is not suitable for the system (1.2). Recently, Song et al. \cite{lv30} developed an algorithm for calculating the normal form of Hopf bifurcation in the system (1.2), and they studied the direction and stability of Hopf bifurcation by using their newly developed algorithm for calculating the normal form. The existences of stable spatially inhomogeneous periodic solutions and the transition from one unstable spatially inhomogeneous periodic solution to another stable spatially inhomogeneous periodic solution are found.

Ghosh et al. \cite{lv31} have researched the reaction-cattaneo equation with fluctuating relaxation time of the diffusive flux, and they pointed out that the delay is closely related to correlated or persistent random walk. The persistence in time implies that a particle continues in its initial direction with a definite probability. Furthermore, the rich spatiotemporal patterns induced by Hopf and double Hopf bifurcations are researched. Ghosh \cite{lv32} has pointed out that the time-delayed feedback is a practical method of controlling bifurcations in reaction-diffusion systems. Furthermore, delayed feedback and its modifications are widely used to control chaos and to stabilize unstable oscillations. For the chemical reaction models, it is practical to consider the influence of the time delay caused by gene expression. The Brusselator model with gene expression delay has been studied in \cite{lv33}. For the artificial neural networks, it is practical to consider the influence of the time delay caused by leakage delay. The delay-dependent stability of neutral neural networks with leakage term delays has been studied in \cite{lv34}. For the biology model, especially for the predator-prey model, the digestion, gestation, hunting, migration and maturation delays are usually considered \cite{lv35,lv36,lv37}, and in this paper, we call these delays as the general delays. By considering that "clever" animals in a polar region usually judge footprints to decide its spatial movement, and footprints record a history of species distribution and movements, thus it is more realistic to consider the memory delay in the diffusive predator-prey model. The general delays, such as the gestation and maturation delays, are common to some animals or plants, and from this point of view, they are different from the memory delay. Furthermore, the digestion, gestation, hunting, migration and maturation periods maybe different from the average memory period, thus it is worth studying the case where the memory and the general delays are different.

By incorporating the general delay to the reaction term in the memory-based diffusive system, we propose the following diffusive system with memory and general delays
\begin{eqnarray}\left\{\begin{aligned}
&\frac{\partial u(x,t)}{\partial t}=d_{11}\Delta u(x,t)+f\left(u(x,t),v(x,t),u(x,t-\tau),v(x,t-\tau)\right),~~~~~~~~~~~~~x \in (0,\ell \pi),~t>0, \\
&\frac{\partial v(x,t)}{\partial t}=d_{22}\Delta v(x,t)-d_{21}\left(v(x,t)u_{x}(x,t-\tau)\right)_{x}+g\left(u(x,t),v(x,t),u(x,t-\tau),v(x,t-\tau)\right), \\
&~~~~~~~~~~~~~~~~~~~~~~~~~~~~~~~~~~~~~~~~~~~~~~~~~~~~~~~~~~~~~~~~~~~~~~~~~~~~~~~~~~~~~~~~~~~~~~~~~~~~x \in(0,\ell\pi),~t>0, \\
&u_{x}(0,t)=u_{x}(\ell\pi,t)=v_{x}(0,t)=v_{x}(\ell\pi,t)=0,~~~~~~~~~~~~~~~~~~~~~~~~~~~~~~~~~~~~~~~~~t \geq 0, \\
&u(x,t)=u_{0}(x,t),~v(x,t)=v_{0}(x,t),~~~~~~~~~~~~~~~~~~~~~~~~~~~~~~~~~~~~~~~~~~~~~~~~~~~~~~x \in (0,\ell\pi),~-\tau \leq t \leq 0.
\end{aligned}\right.\end{eqnarray}
At the beginning, we pointed out that Satulovsky et al. \cite{lv5} used a stochastic lattice gas model to describe the dynamics of a predator-prey system without diffusion and delay. Tsyganov et al. \cite{lv6} researched a predator-prey system with cross-diffusion without delay, the "taxis" wave which is generated by this system can often penetrate through each other and reflect from impermeable boundaries. Therefore, from the physical insight, a stochastic lattice gas model can also be used to describe the proposed model (1.3). Especially, the memory and general delays of model (1.3) can be understand as the time delay to arrive a particular location in the lattice due to the influences of external perturbations. Furthermore, from the subsequent numerical simulation, we can see that a limit cycle occurs, and our derived algorithm for calculating the normal form of Hopf bifurcation in model (1.3) can be used to determine the direction and stability of the Hopf bifurcation period solution. Therefore, the connection between the limit cycle occurs in model (1.3) and the solitary propagating wave maybe a worthwhile research area which needs to be investigated in terms of the physical subject. Once we make the connection between them, our derived algorithm for calculating the normal form of Hopf bifurcation can be used to determine the direction and stability of the solitary propagating wave.

The paper is divided into five sections. In Section 2, we derive an algorithm for calculating the normal form of Hopf bifurcation induced jointly by memory and general delays. In Section 3, we obtain the normal form of Hopf bifurcation truncated to the third-order term by using the algorithm developed in Sec.2, and we give the detail calculation process of its corresponding coefficients. In Section 4, we consider a diffusive predator-prey model with ratio-dependent Holling type-\uppercase\expandafter{\romannumeral3} functional response, which includes with memory and gestation delays. Then we give the detail Hopf bifurcation analysis for two cases, i.e., with memory delay and without gestation delay, and with memory and gestation delays. Furthermore, we study the direction and stability of Hopf bifurcation corresponding to the above two cases. Finally, we give a brief conclusion and discussion in Section 5.

\section{Algorithm for calculating the normal form of Hopf bifurcation induced jointly by memory and general delays}
\label{sec:2}

\subsection{Characteristic equation at the positive constant steady state}

Define the real-valued Sobolev space
\begin{eqnarray*}
X:=\left\{(u,v)^{T} \in \left(W^{2,2}(0,\ell\pi)\right)^{2}:\frac{\partial u}{\partial x}=\frac{\partial v}{\partial x}=0 \text { at } x=0,\ell \pi\right\}
\end{eqnarray*}
with the inner product defined by
\begin{eqnarray*}
\left[U_{1},U_{2}\right]=\int_{0}^{\ell\pi} U_{1}^{T}U_{2}~dx \text { for } U_{1}=\left(u_{1},v_{1}\right)^{T} \in X \text{ and } U_{2}=\left(u_{2},v_{2}\right)^{T} \in X,
\end{eqnarray*}
where the symbol $T$ represents the transpose of vector, and let $\mathcal{C}:=C([-1,0];X)$ be the Banach space of continuous mappings from $[-1,0]$ to $X$ with the sup norm. It is well known that the eigenvalue problem
\begin{eqnarray*}\left\{\begin{aligned}
&\widetilde{\varphi}^{\prime \prime}(x)=\widetilde{\lambda}\widetilde{\varphi}(x),~x \in(0,\ell\pi), \\
&\widetilde{\varphi}^{\prime}(0)=\widetilde{\varphi}^{\prime}(\ell\pi)=0
\end{aligned}\right.\end{eqnarray*}
has eigenvalues $\widetilde{\lambda}_{n}=-n^{2}/\ell^{2}$ with corresponding normalized eigenfunctions
\begin{eqnarray}
\beta_{n}^{(j)}=\gamma_{n}(x)e_{j},~\gamma_{n}(x)=\frac{\cos(nx/\ell)}{\left\|\cos(nx/\ell)\right\|_{L^{2}}}=\left\{\begin{aligned}
&\frac{1}{\sqrt{\ell\pi}}, & n=0, \\
&\sqrt{\frac{2}{\ell\pi}}\cos\left(\frac{nx}{\ell}\right), & n \geq 1,
\end{aligned}\right.
\end{eqnarray}
where $e_{j},~j=1,2$ is the unit coordinate vector of $\mathbb{R}^{2}$, and $n \in \mathbb{N}_{0}=\mathbb{N}\cup \left\{0\right\}$ is often called wave number, $\mathbb{N}_{0}$ is the set of all non-negative integers, $\mathbb{N}=\left\{1,2,...\right\}$ represents the set of all positive integers.

Without loss of generality, we assume that $E_{*}\left(u_{*},v_{*}\right)$ is the positive constant steady state of system (1.3). The linearized equation of (1.3) at $E_{*}\left(u_{*},v_{*}\right)$ is
\begin{eqnarray}
\left(\begin{aligned}
&\frac{\partial u(x,t)}{\partial t} \\
&\frac{\partial v(x,t)}{\partial t}
\end{aligned}\right)=D_{1}\left(\begin{aligned}
&\Delta u(x,t) \\
&\Delta v(x,t)
\end{aligned}\right)+D_{2}\left(\begin{aligned}
&\Delta u(x,t-\tau) \\
&\Delta v(x,t-\tau)
\end{aligned}\right)+A_{1}\left(\begin{aligned}
&u(x,t) \\
&v(x,t)
\end{aligned}\right)+A_{2}\left(\begin{aligned}
&u(x,t-\tau) \\
&v(x,t-\tau)
\end{aligned}\right),
\end{eqnarray}
where
\begin{eqnarray}
D_{1}=\left(\begin{array}{cc}
d_{11} & 0 \\
0 & d_{22}
\end{array}\right),~D_{2}=\left(\begin{array}{cc}
0 & 0 \\
-d_{21} v_{*} & 0
\end{array}\right),~A_{1}=\left(\begin{array}{cc}
a_{11} & a_{12} \\
a_{21} & a_{22}
\end{array}\right),~A_{2}=\left(\begin{array}{cc}
b_{11} & b_{12} \\
b_{21} & b_{22}
\end{array}\right)
\end{eqnarray}
and
\begin{eqnarray}\begin{aligned}
a_{11}&=\frac{\partial f\left(u_{*},v_{*}\right)}{\partial u(x,t)},~a_{12}=\frac{\partial f\left(u_{*},v_{*}\right)}{\partial v(x,t)},~a_{21}=\frac{\partial g\left(u_{*},v_{*}\right)}{\partial u(x,t)},~a_{22}=\frac{\partial g\left(u_{*},v_{*}\right)}{\partial v(x,t)}, \\
b_{11}&=\frac{\partial f\left(u_{*},v_{*}\right)}{\partial u(x,t-\tau)},~b_{12}=\frac{\partial f\left(u_{*},v_{*}\right)}{\partial v(x,t-\tau)},~b_{21}=\frac{\partial g\left(u_{*},v_{*}\right)}{\partial u(x,t-\tau)},~b_{22}=\frac{\partial g\left(u_{*},v_{*}\right)}{\partial v(x,t-\tau)}.
\end{aligned}\end{eqnarray}

Therefore, the characteristic equation of system (2.2) is
\begin{eqnarray*}
\prod_{n \in \mathbb{N}_{0}}\Gamma_{n}(\lambda)=0,
\end{eqnarray*}
where $\Gamma_{n}(\lambda)=\det\left(M_{n}(\lambda)\right)$ with
\begin{eqnarray}
M_{n}(\lambda)=\lambda I_{2}+\frac{n^{2}}{\ell^{2}} D_{1}+\frac{n^{2}}{\ell^{2}}e^{-\lambda \tau} D_{2}-A_{1}-A_{2}e^{-\lambda \tau}.
\end{eqnarray}
Here, $\det(.)$ represents the determinant of a matrix, $I_{2}$ is the identity matrix of $2 \times 2$, and $D_{1}, D_{2}, A_{1}, A_{2}$ are defined by (2.3). Then we obtain
\begin{eqnarray}
\Gamma_{n}(\lambda)=\det\left(M_{n}(\lambda)\right)=\lambda^{2}-T_{n}\lambda+\widetilde{J}_{n}(\tau)=0,
\end{eqnarray}
where
\begin{eqnarray}\begin{aligned}
T_{n}&=(a_{11}+a_{22})-(d_{11}+d_{22})\frac{n^{2}}{\ell^{2}}, \\
\widetilde{J}_{n}(\tau)&=d_{11}d_{22}\frac{n^{4}}{\ell^{4}}-\left(d_{11}a_{22}+d_{22}a_{11}+(d_{11}b_{22}+d_{22}b_{11}+d_{21}a_{12}v_{*})e^{-\lambda\tau}+d_{21}b_{12}v_{*}e^{-2\lambda\tau}\right)\frac{n^{2}}{\ell^{2}} \\
&+(a_{11}b_{22}+a_{22}b_{11}-a_{12}b_{21}-a_{21}b_{12})e^{-\lambda\tau}-(b_{11}+b_{22})\lambda e^{-\lambda\tau}+(b_{11}b_{22}-b_{12}b_{21})e^{-2\lambda\tau} \\
&+a_{11}a_{22}-a_{12}a_{21}.
\end{aligned}\end{eqnarray}

\subsection{Basic assumption and equation transformation}

\begin{assumption}
Assume that at $\tau=\tau_{c}$, (2.6) has a pair of purely imaginary roots $\pm i \omega_{n_{c}}$ with $\omega_{n_{c}}>0$ for $n=n_{c} \in \mathbb{N}_{0}$ and all other eigenvalues have negative real part. Let $\lambda(\tau)=\alpha_{1}(\tau) \pm i \alpha_{2}(\tau)$ be a pair of roots of (2.6) near $\tau=\tau_{c}$ satisfying $\alpha_{1}(\tau_{c})=0$ and $\alpha_{2}(\tau_{c})=\omega_{n_{c}}$. In addition, the corresponding transversality condition holds.
\end{assumption}

Let $\tau=\tau_{c}+\mu,~|\mu| \ll 1$ such that $\mu=0$ corresponds to the Hopf bifurcation value for system (1.3). Moreover, we shift $E_{*}(u_{*},v_{*})$ to the origin by setting
\begin{eqnarray*}
U(x,t)=\left(U_{1}(x,t),U_{2}(x,t)\right)^{T}=\left(u(x,t),v(x,t)\right)^{T}-\left(u_{*},v_{*}\right)^{T},
\end{eqnarray*}
and normalize the delay by rescaling the time variable $t \rightarrow t/\tau$. Furthermore, we rewrite $U(t)$ for $U(x,t)$, and $U_{t} \in \mathcal{C}$ for $U_{t}(\theta)=U(x,t+\theta),~-1 \leq \theta \leq 0$. Then the system (1.3) becomes the compact form
\begin{eqnarray}
\frac{dU(t)}{dt}=d(\mu)\Delta(U_{t})+L(\mu)(U_{t})+F(U_{t},\mu),
\end{eqnarray}
where for $\varphi=\left(\varphi^{(1)},\varphi^{(2)}\right)^{T} \in \mathcal{C}$, $d(\mu)\Delta$ is given by
\begin{eqnarray*}
d(\mu)\Delta(\varphi)=d_{0}\Delta(\varphi)+F^{d}(\varphi,\mu)
\end{eqnarray*}
with
\begin{eqnarray}\begin{aligned}
d_{0}\Delta(\varphi)&=\tau_{c}D_{1}\Delta\varphi(0)+\tau_{c}D_{2}\Delta\varphi(-1), \\
F^{d}(\varphi,\mu)&=-d_{21}(\tau_{c}+\mu)\left(\begin{array}{c}
0 \\
\varphi^{(1)}_{x}(-1)\varphi^{(2)}_{x}(0)+\varphi^{(1)}_{xx}(-1)\varphi^{(2)}(0)
\end{array}\right) \\
&+\mu\left(\begin{array}{c}
d_{11}\varphi^{(1)}_{xx}(0) \\
-d_{21}v_{*}\varphi^{(1)}_{xx}(-1)+d_{22}\varphi^{(2)}_{xx}(0)
\end{array}\right).
\end{aligned}\end{eqnarray}
Furthermore, $L(\mu): \mathcal{C} \rightarrow X$ is given by
\begin{eqnarray}
L(\mu)(\varphi)=\left(\tau_{c}+\mu\right)\left(A_{1}\varphi(0)+A_{2}\varphi(-1)\right),
\end{eqnarray}
and $F: \mathcal{C} \times \mathbb{R}^{2} \rightarrow X$ is given by
\begin{eqnarray}
F(\varphi,\mu)=\left(\tau_{c}+\mu\right)\left(\begin{aligned}
f\left(\varphi^{(1)}(0)+u_{*},\varphi^{(2)}(0)+v_{*},\varphi^{(1)}(-1)+u_{*},\varphi^{(2)}(-1)+v_{*}\right) \\
g\left(\varphi^{(1)}(0)+u_{*},\varphi^{(2)}(0)+v_{*},\varphi^{(1)}(-1)+u_{*},\varphi^{(2)}(-1)+v_{*}\right)
\end{aligned}\right)-L(\mu)(\varphi).
\end{eqnarray}

In what follows, we assume that $F(\varphi,\mu)$ is $C^{k}(k \geq 3)$ function, which is smooth with respect to $\varphi$ and $\mu$. Notice that $\mu$ is the perturbation parameter and is treated as a variable in the calculation of normal form. Moreover, from (2.10), if we denote $L_{0}(\varphi)=\tau_{c} \left(A_{1} \varphi(0)+A_{2}\varphi(-1)\right)$, then (2.8) can be rewritten as
\begin{eqnarray}
\frac{dU(t)}{dt}=d_{0}\Delta(U_{t})+L_{0}(U_{t})+\widetilde{F}\left(U_{t},\mu\right),
\end{eqnarray}
where the linear and nonlinear terms are separated, and
\begin{eqnarray}
\widetilde{F}(\varphi,\mu)=\mu\left(A_{1}\varphi(0)+A_{2}\varphi(-1)\right)+F(\varphi,\mu)+F^{d}(\varphi,\mu).
\end{eqnarray}

Thus, the linearized equation of (2.12) can be written as
\begin{eqnarray}
\frac{dU(t)}{dt}=d_{0}\Delta(U_{t})+L_{0}(U_{t}).
\end{eqnarray}
Moreover, the characteristic equation for the linearized equation (2.14) is
\begin{eqnarray}
\prod_{n \in \mathbb{N}_{0}}\widetilde{\Gamma}_{n}(\lambda)=0,
\end{eqnarray}
where $\widetilde{\Gamma}_{n}(\lambda)=\operatorname{det}\left(\widetilde{M}_{n}(\lambda)\right)$ with
\begin{eqnarray}
\widetilde{M}_{n}(\lambda)=\lambda I_{2}+\tau_{c}\frac{n^{2}}{\ell^{2}}D_{1}+\tau_{c}\frac{n^{2}}{\ell^{2}}e^{-\lambda}D_{2}-\tau_{c} A_{1}-\tau_{c}A_{2}e^{-\lambda}.
\end{eqnarray}

By comparing (2.16) with (2.5), we know that (2.15) has a pair of purely imaginary roots $\pm i \omega_{c}$ for $n=n_{c} \in \mathbb{N}_{0}$, and all other eigenvalues have negative real parts, where $\omega_{c}=\tau_{c}\omega_{n_{c}}$. In order to write (2.12) as an abstract ordinary differential equation in a Banach space, follows by \cite{lv38}, we can take the enlarged space
\begin{eqnarray*}
\mathcal{BC}:=\left\{\widetilde{\psi}:[-1,0]\rightarrow X:\widetilde{\psi} \text{ is continuous on } [-1,0), \exists \lim _{\theta \rightarrow 0^{-}} \widetilde{\psi}(\theta)\in X\right\},
\end{eqnarray*}
then the equation (2.12) is equivalent to an abstract ordinary differential equation on $\mathcal{BC}$
\begin{eqnarray*}
\frac{dU_{t}}{dt}=\widetilde{A}U_{t}+X_{0}\widetilde{F}\left(U_{t},\mu\right).
\end{eqnarray*}
Here, $\widetilde{A}$ is a operator from $\mathcal{C}_{0}^{1}=\{\varphi \in \mathcal{C}:\dot{\varphi} \in \mathcal{C},\varphi(0) \in \operatorname{dom}(\Delta)\}$ to $\mathcal{BC}$, which is defined by
\begin{eqnarray*}
\widetilde{A}\varphi=\dot{\varphi}+X_{0}\left(\tau_{c}D_{1}\Delta\varphi(0)+\tau_{c}D_{2}\Delta\varphi(-1)+L_{0}(\varphi)-\dot{\varphi}(0)\right),
\end{eqnarray*}
and $X_{0}=X_{0}(\theta)$ is given by
\begin{eqnarray*}
X_{0}(\theta)=\left\{\begin{aligned}
& 0, & -1 \leq \theta<0, \\
& 1, & \theta=0.
\end{aligned}\right.
\end{eqnarray*}

In the following, the method given in \cite{lv38} is used to complete the decomposition of $\mathcal{BC}$. Let $C:=C\left([-1,0],\mathbb{R}^{2}\right),~C^{*}:=C\left([0,1],\mathbb{R}^{2*}\right)$, where $\mathbb{R}^{2*}$ is the two-dimensional space of row vectors, and define the adjoint bilinear form on $C^{*}\times C$ as follows
\begin{eqnarray*}
\langle\Psi(s),\Phi(\theta)\rangle=\Psi(0)\Phi(0)-\int_{-1}^{0}\int_{0}^{\theta}\Psi(\xi-\theta)dM_{n}(\theta)\Phi(\xi)d\xi
\end{eqnarray*}
for $\Psi \in C^{*},~\Phi \in C$ and $\xi \in [-1,0]$, where $M_{n}(\theta)$ is a bounded variation function from $[-1,0]$ to $\mathbb{R}^{2 \times 2}$, i.e., $M_{n}(\theta)\in BV\left([-1,0];\mathbb{R}^{2 \times 2}\right)$, such that for $\Phi(\theta)\in\mathcal{C}$, one has
\begin{eqnarray*}
-\tau_{c}\frac{n^{2}}{\ell^{2}}D_{1}\Phi(0)-\tau_{c}\frac{n^{2}}{\ell^{2}}D_{2}\Phi(-1)+L_{0}(\Phi(\theta))=\int_{-1}^{0}dM_{n}(\theta)\Phi(\theta).
\end{eqnarray*}

By choosing
\begin{eqnarray*}
\Phi(\theta)=(\phi(\theta),\overline{\phi}(\theta)),~\Psi(s)=\operatorname{col}\left(\psi^{T}(s),\overline{\psi}^{T}(s)\right),
\end{eqnarray*}
where the $\operatorname{col}(.)$ represents the column vector, $\phi(\theta)=\operatorname{col}\left(\phi_{1}(\theta),\phi_{2}(\theta)\right)=\phi e^{i \omega_{c} \theta}\in \mathbb{C}^{2}$ with $\phi=\operatorname{col}\left(\phi_{1},\phi_{2}\right)$ is the eigenvector of (2.14) associated with the eigenvalue $i \omega_{c}$, and $\psi(s)=\operatorname{col}\left(\psi_{1}(s),\psi_{2}(s)\right)=\psi e^{-i\omega_{c}s}\in \mathbb{C}^{2}$ with $\psi=\operatorname{col}\left(\psi_{1},\psi_{2}\right)$ is the corresponding adjoint eigenvector such that
\begin{eqnarray*}
\left\langle\Psi(s), \Phi(\theta)\right\rangle=I_{2},
\end{eqnarray*}
where
\begin{eqnarray*}
\phi=\left(\begin{array}{c}
1 \\
\frac{i\omega_{n_{c}}+d_{11}\left(n_{c}/\ell\right)^{2}-a_{11}-b_{11}e^{-i \omega_{c}}}{a_{12}+b_{12}e^{-i \omega_{c}}}
\end{array}\right),~\psi=\eta\left(\begin{array}{c}
1 \\
\frac{a_{12}+b_{12}e^{-i \omega_{c}}}{i\omega_{n_{c}}+d_{22}\left(n_{c}/\ell\right)^{2}-a_{22}-b_{22}e^{-i \omega_{c}}}
\end{array}\right).
\end{eqnarray*}
Here,
\begin{eqnarray*}
\eta=\frac{1}{1+k_{1}k_{2}+e^{-i\omega_{c}}\tau_{c}b_{11}+e^{-i\omega_{c}}k_{2}\left(\tau_{c}b_{21}+\tau_{c}d_{21}v_{*}\left(n_{c}/\ell\right)^{2}\right)}
\end{eqnarray*}
with
\begin{eqnarray*}
k_{1}=\frac{i\omega_{n_{c}}+d_{11}\left(n_{c}/\ell\right)^{2}-a_{11}-b_{11}e^{-i\omega_{c}}}{a_{12}+b_{12}e^{-i\omega_{c}}},~k_{2}=\frac{a_{12}+b_{12}e^{-i\omega_{c}}}{i\omega_{n_{c}}+d_{22}\left(n_{c}/\ell\right)^{2}-a_{22}-b_{22}e^{-i\omega_{c}}}.
\end{eqnarray*}

According to \cite{lv38}, the phase space $\mathcal{C}$ can be decomposed as
\begin{eqnarray*}
\mathcal{C}=\mathcal{P} \oplus \mathcal{Q},~\mathcal{P}=\operatorname{Im} \pi,~\mathcal{Q}=\operatorname{Ker} \pi,
\end{eqnarray*}
where for $\widetilde{\phi} \in \mathcal{C}$, the projection $\pi: \mathcal{C} \rightarrow \mathcal{P}$ is defined by
\begin{eqnarray}
\pi(\widetilde{\phi})=\left(\Phi\left\langle\Psi,\left(\begin{aligned}
&\left[\widetilde{\phi}(\cdot),\beta_{n_{c}}^{(1)}\right] \\
&\left[\widetilde{\phi}(\cdot),\beta_{n_{c}}^{(2)}\right]
\end{aligned}\right)\right\rangle\right)^{T} \beta_{n_{c}}.
\end{eqnarray}

Therefore, according to the method given in \cite{lv38}, $\mathcal{BC}$ can be divided into a direct sum of center subspace and its complementary space, that is
\begin{eqnarray}
\mathcal{BC}=\mathcal{P} \oplus \operatorname{ker} \pi,
\end{eqnarray}
where $\operatorname{dim}\mathcal{P}=2$. It is easy to see that the projection $\pi$ which is defined by (2.17), is extended to a continuous projection (which is still denoted by $\pi$), that is, $\pi:\mathcal{BC}\mapsto \mathcal{P}$. In particular, for $\alpha \in \mathcal{C}$, we have
\begin{eqnarray}
\pi\left(X_{0}(\theta)\alpha\right)=\left(\Phi(\theta)\Psi(0)\left(\begin{aligned}
&\left[\alpha,\beta_{n_{c}}^{(1)}\right] \\
&\left[\alpha,\beta_{n_{c}}^{(2)}\right]
\end{aligned}\right)\right)^{T} \beta_{n_{c}}.
\end{eqnarray}

By combining with (2.18) and (2.19), $U_{t}(\theta)$ can be decomposed as
\begin{eqnarray}\begin{aligned}
U_{t}(\theta)&=\left(\Phi(\theta)\left(\begin{array}{c}
z_{1} \\
z_{2}
\end{array}\right)\right)^{T}\left(\begin{array}{c}
\beta_{n_{c}}^{(1)} \\
\beta_{n_{c}}^{(2)}
\end{array}\right)+w \\
&=\left(z_{1}\phi e^{i\omega_{c}\theta}+z_{2}\overline{\phi}e^{-i\omega_{c}\theta}\right)\gamma_{n_{c}}(x)+w \\
&=\left(\phi(\theta)~~\overline{\phi}(\theta)\right)\left(\begin{array}{c}
z_{1}\gamma_{n_{c}}(x) \\
z_{2}\gamma_{n_{c}}(x)
\end{array}\right)+\left(\begin{array}{c}
w_{1} \\
w_{2}
\end{array}\right),
\end{aligned}\end{eqnarray}
where $w=\operatorname{col}\left(w_{1},w_{2}\right)$ and
\begin{eqnarray*}\begin{aligned}
&\left(\begin{array}{c}
z_{1} \\
z_{2}
\end{array}\right)=\left\langle\Psi(0),\left(\begin{aligned}
&\left[U_{t}(\theta),\beta_{n_{c}}^{(1)}\right] \\
&\left[U_{t}(\theta),\beta_{n_{c}}^{(2)}\right]
\end{aligned}\right)\right\rangle.
\end{aligned}\end{eqnarray*}

If we assume that
\begin{eqnarray*}
\Phi(\theta)=\left(\phi(\theta),\overline{\phi}(\theta)\right),~z_{x}=\left(z_{1}\gamma_{n_{c}}(x),z_{2}\gamma_{n_{c}}(x)\right)^{T},
\end{eqnarray*}
then (2.20) can be rewritten as
\begin{eqnarray}
U_{t}(\theta)=\Phi(\theta)z_{x}+w \text{ with } w \in \mathcal{C}_{0}^{1} \cap \operatorname{Ker} \pi:=\mathcal{Q}^{1}.
\end{eqnarray}

Then by combining with (2.21), the system (2.12) is decomposed as a system of abstract ordinary differential equations (ODEs) on $\mathbb{R}^{2} \times \text{Ker}~\pi$, with finite and infinite dimensional variables are separated in the linear term. That is
\begin{eqnarray}
\left\{\begin{aligned}
&\dot{z}=Bz+\Psi(0)\left(\begin{aligned}
&\left[\widetilde{F}\left(\Phi(\theta)z_{x}+w,\mu\right),\beta_{n_{c}}^{(1)}\right] \\
&\left[\widetilde{F}\left(\Phi(\theta)z_{x}+w,\mu\right),\beta_{n_{c}}^{(2)}\right]
\end{aligned}\right), \\
&\dot{w}=A_{\mathcal{Q}^{1}}w+(I-\pi)X_{0}(\theta)\widetilde{F}\left(\Phi(\theta)z_{x}+w,\mu\right),
\end{aligned}\right.
\end{eqnarray}
where $I$ is the identity matrix, $z=\left(z_{1},z_{2}\right)^{T}$, $B=\text{diag}\left\{i\omega_{c},-i\omega_{c}\right\}$ is the diagonal matrix, and $A_{\mathcal{Q}^{1}}:\mathcal{Q}^{1} \rightarrow \operatorname{Ker}\pi$ is defined by
\begin{eqnarray*}
A_{\mathcal{Q}^{1}}w=\dot{w}+X_{0}(\theta)\left(\tau_{c}D_{1}\Delta w(0)+\tau_{c}D_{2}\Delta w(-1)+L_{0}(w)-\dot{w}(0)\right).
\end{eqnarray*}

Consider the formal Taylor expansion
\begin{eqnarray*}
\widetilde{F}(\varphi,\mu)=\sum_{j \geq 2} \frac{1}{j!}\widetilde{F}_{j}(\varphi,\mu),~F(\varphi,\mu)=\sum_{j \geq 2}\frac{1}{j!} F_{j}(\varphi,\mu),~F^{d}(\varphi,\mu)=\sum_{j \geq 2}\frac{1}{j!}F_{j}^{d}(\varphi,\mu).
\end{eqnarray*}
From (2.13), we have
\begin{eqnarray}
\widetilde{F}_{2}(\varphi,\mu)=2\mu\left(A_{1}\varphi(0)+A_{2}\varphi(-1)\right)+F_{2}(\varphi,\mu)+F_{2}^{d}(\varphi,\mu)
\end{eqnarray}
and
\begin{eqnarray}
\widetilde{F}_{j}(\varphi,\mu)=F_{j}(\varphi,\mu)+F_{j}^{d}(\varphi,\mu),~j=3,4,\cdots.
\end{eqnarray}

By combining with (2.19), the system (2.22) can be rewritten as
\begin{eqnarray*}
\left\{\begin{aligned}
&\dot{z}=Bz+\sum_{j \geq 2}\frac{1}{j!}f_{j}^{1}(z,w,\mu), \\
&\dot{w}=A_{\mathcal{Q}^{1}}w+\sum_{j \geq 2}\frac{1}{j!}f_{j}^{2}(z,w,\mu),
\end{aligned}\right.
\end{eqnarray*}
where
\begin{eqnarray}\begin{aligned}
&f_{j}^{1}(z,w,\mu)=\Psi(0)\left(\begin{aligned}
&\left[\widetilde{F}_{j}\left(\Phi(\theta)z_{x}+w,\mu\right),\beta_{n_{c}}^{(1)}\right] \\
&\left[\widetilde{F}_{j}\left(\Phi(\theta)z_{x}+w,\mu\right),\beta_{n_{c}}^{(2)}\right]
\end{aligned}\right), \\
&f_{j}^{2}(z,w,\mu)=(I-\pi)X_{0}(\theta)\widetilde{F}_{j}\left(\Phi(\theta)z_{x}+w,\mu\right).
\end{aligned}\end{eqnarray}

In terms of the normal form theory of partial functional differential equations \cite{lv38}, after a recursive transformation of variables of the form
\begin{eqnarray}
(z,w)=(\widetilde{z},\widetilde{w})+\frac{1}{j!}\left(U_{j}^{1}(\widetilde{z},\mu),U_{j}^{2}(\widetilde{z},\mu)\right),~j \geq 2,
\end{eqnarray}
where $z, \widetilde{z} \in \mathbb{R}^{2}$, $w, \widetilde{w} \in \mathcal{Q}^{1}$ and $U_{j}^{1}:\mathbb{R}^{3} \rightarrow \mathbb{R}^{2}$, $U_{j}^{2}: \mathbb{R}^{3} \rightarrow \mathcal{Q}^{1}$ are homogeneous polynomials of degree $j$ in $\widetilde{z}$ and $\mu$, a locally center manifold for (2.12) satisfies $w=0$ and the flow on it is given by the two-dimensional ODEs
\begin{eqnarray*}
\dot{z}=B z+\sum_{j \geq 2} \frac{1}{j!} g_{j}^{1}(z,0,\mu),
\end{eqnarray*}
which is the normal form as in the usual sense for ODEs.

By following \cite{lv38} and \cite{lv39}, we have
\begin{eqnarray}
g_{2}^{1}(z,0,\mu)=\operatorname{Proj}_{\operatorname{Ker}\left(M_{2}^{1}\right)}f_{2}^{1}(z,0,\mu)
\end{eqnarray}
and
\begin{eqnarray}
g_{3}^{1}(z,0,\mu)=\operatorname{Proj}_{\operatorname{Ker}\left(M_{3}^{1}\right)}\widetilde{f}_{3}^{1}(z,0,\mu)=\operatorname{Proj}_{S} \widetilde{f}_{3}^{1}(z,0,0)+O\left(\mu^{2}|z|\right),
\end{eqnarray}
where $\operatorname{Proj}_{p}(q)$ represents the projection of $q$ on $p$, and $\widetilde{f}_{3}^{1}(z,0,\mu)$ is vector and its element is the cubic polynomial of $(z, \mu)$ after the variable transformation of (2.26), and it can be determined by (2.38),
\begin{eqnarray}\begin{aligned}
&\operatorname{Ker}\left(M_{2}^{1}\right)=\operatorname{Span}\left\{\left(\begin{array}{c}
\mu z_{1} \\
0
\end{array}\right),\left(\begin{array}{c}
0 \\
\mu z_{2}
\end{array}\right)\right\}, \\
&\operatorname{Ker}\left(M_{3}^{1}\right)=\operatorname{Span}\left\{\left(\begin{array}{c}
z_{1}^{2}z_{2} \\
0
\end{array}\right),\left(\begin{array}{c}
\mu^{2}z_{1} \\
0
\end{array}\right),\left(\begin{array}{c}
0 \\
z_{1}z_{2}^{2}
\end{array}\right),\left(\begin{array}{c}
0 \\
\mu^{2}z_{2}
\end{array}\right)\right\},
\end{aligned}\end{eqnarray}
and
\begin{eqnarray}
S=\operatorname{Span}\left\{\left(\begin{array}{c}
z_{1}^{2}z_{2} \\
0
\end{array}\right),\left(\begin{array}{c}
0 \\
z_{1}z_{2}^{2}
\end{array}\right)\right\}.
\end{eqnarray}

In the following, for notational convenience, we let
\begin{eqnarray*}
\mathcal{H}\left(\alpha z_{1}^{q_{1}}z_{2}^{q_{2}}\mu\right)=\left(\begin{array}{c}
\alpha z_{1}^{q_{1}}z_{2}^{q_{2}}\mu \\
\overline{\alpha}z_{1}^{q_{2}}z_{2}^{q_{1}}\mu
\end{array}\right),~\alpha \in \mathbb{C}.
\end{eqnarray*}

We then calculate $g_{j}^{1}(z,0,\mu),~j=2,3$ step by step.

\subsection{Algorithm for calculating the normal form of Hopf bifurcation}

\subsubsection{Calculation of $g_{2}^{1}(z,0,\mu)$}

From the second mathematical expression in (2.9), we have
\begin{eqnarray}
F_{2}^{d}(\varphi,\mu)=F_{20}^{d}(\varphi)+\mu F_{21}^{d}(\varphi)
\end{eqnarray}
and
\begin{eqnarray}
F_{3}^{d}(\varphi,\mu)=\mu F_{31}^{d}(\varphi),~F_{j}^{d}(\varphi,\mu)=(0,0)^{T},~j=4,5,\cdots,
\end{eqnarray}
where
\begin{eqnarray}
\left\{\begin{aligned}
F_{20}^{d}(\varphi)&=-2d_{21}\tau_{c}\left(\begin{array}{c}
0 \\
\varphi^{(1)}_{x}(-1)\varphi^{(2)}_{x}(0)+\varphi^{(1)}_{xx}(-1)\varphi^{(2)}(0)
\end{array}\right), \\
F_{21}^{d}(\varphi)&=2D_{1}\Delta\varphi(0)+2D_{2}\Delta\varphi(-1), \\
F_{31}^{d}(\varphi)&=-6d_{21}\left(\begin{array}{c}
0 \\
\varphi^{(1)}_{x}(-1)\varphi^{(2)}_{x}(0)+\varphi^{(1)}_{xx}(-1)\varphi^{(2)}(0)
\end{array}\right).
\end{aligned}\right.
\end{eqnarray}

Furthermore, it is easy to verify that
\begin{eqnarray}\begin{aligned}
&\left(\begin{aligned}
&\left[2 \mu \left(A_{1}(\Phi(0)z_{x})+A_{2}(\Phi(-1)z_{x})\right),\beta_{n_{c}}^{(1)}\right] \\
&\left[2 \mu \left(A_{1}(\Phi(0)z_{x})+A_{2}(\Phi(-1)z_{x})\right),\beta_{n_{c}}^{(2)}\right]
\end{aligned}\right)=2\mu A_{1}\left(\Phi(0)\left(\begin{array}{c}
z_{1} \\
z_{2}
\end{array}\right)\right)+2\mu A_{2}\left(\Phi(-1)\left(\begin{array}{c}
z_{1} \\
z_{2}
\end{array}\right)\right), \\
&\left(\begin{aligned}
&\left[\mu F_{21}^{d}\left(\Phi(\theta)z_{x}\right),\beta_{n_{c}}^{(1)}\right] \\
&\left[\mu F_{21}^{d}\left(\Phi(\theta)z_{x}\right),\beta_{n_{c}}^{(2)}\right]
\end{aligned}\right)=-2\frac{n_{c}^{2}}{\ell^{2}}\mu\left(D_{1}\left(\Phi(0)\left(\begin{array}{c}
z_{1} \\
z_{2}
\end{array}\right)\right)+D_{2}\left(\Phi(-1)\left(\begin{array}{c}
z_{1} \\
z_{2}
\end{array}\right)\right)\right).
\end{aligned}\end{eqnarray}

From (2.11), we have for all $\mu \in \mathbb{R}$, $F_{2}\left(\Phi(\theta)z_{x},\mu\right)=F_{2}\left(\Phi(\theta)z_{x},0\right)$. It follows from the first mathematical expression in (2.25) that
\begin{eqnarray*}
f_{2}^{1}(z, 0, \mu)=\Psi(0)\left(\begin{aligned}
&\left[\widetilde{F}_{2}\left(\Phi(\theta)z_{x},\mu\right),\beta_{n_{c}}^{(1)}\right] \\
&\left[\widetilde{F}_{2}\left(\Phi(\theta)z_{x},\mu\right),\beta_{n_{c}}^{(2)}\right]
\end{aligned}\right).
\end{eqnarray*}

This, together with (2.23), (2.27), (2.29), (2.31), (2.32), (2.33) and (2.34), yields to
\begin{eqnarray}
g_{2}^{1}(z,0,\mu)=\operatorname{Proj}_{\operatorname{Ker}\left(M_{2}^{1}\right)}f_{2}^{1}(z,0,\mu)=\mathcal{H}\left(B_{1}\mu z_{1}\right),
\end{eqnarray}
where
\begin{eqnarray}
B_{1}=\left\{\begin{aligned}
&2\psi^{T}(0)\left(A_{1}\phi(0)+A_{2}\phi(-1)-\frac{n_{c}^{2}}{\ell^{2}}\left(D_{1}\phi(0)+D_{2}\phi(-1)\right)\right), &n_{c}\in\mathbb{N}, \\
&2\psi^{T}(0)\left(A_{1}\phi(0)+A_{2}\phi(-1)\right), &n_{c}=0.
\end{aligned}\right.
\end{eqnarray}

\subsubsection{Calculation of $g_{3}^{1}(z,0,\mu)$}

Notice that the calculation of $g_{3}^{1}(z,0,\mu)$ is very similar to that in \cite{lv30}. Here, we simply give the results. In this subsection, we calculate the third term $g_{3}^{1}(z,0,0)$ in terms of (2.28).

Denote
\begin{eqnarray}\begin{aligned}
&f_{2}^{(1,1)}(z,w,0)=\Psi(0)\left(\begin{aligned}
&\left[F_{2}\left(\Phi(\theta)z_{x}+w,0\right),\beta_{n_{c}}^{(1)}\right] \\
&\left[F_{2}\left(\Phi(\theta)z_{x}+w,0\right),\beta_{n_{c}}^{(2)}\right]
\end{aligned}\right), \\
&f_{2}^{(1,2)}(z,w,0)=\Psi(0)\left(\begin{aligned}
&\left[F_{2}^{d}\left(\Phi(\theta)z_{x}+w,0\right),\beta_{n_{c}}^{(1)}\right] \\
&\left[F_{2}^{d}\left(\Phi(\theta)z_{x}+w,0\right),\beta_{n_{c}}^{(2)}\right]
\end{aligned}\right).
\end{aligned}\end{eqnarray}

It follows from (2.35) that $g_{2}^{1}(z,0,0)=(0,0)^{T}$. Then $\widetilde{f}_{3}^{1}(z,0,0)$ is determined by
\begin{eqnarray}\begin{aligned}
\widetilde{f}_{3}^{1}(z,0,0)&=f_{3}^{1}(z,0,0)+\frac{3}{2}\left((D_{z}f_{2}^{1}(z,0,0))U_{2}^{1}(z,0)+(D_{w}f_{2}^{(1,1)}(z,0,0))U_{2}^{2}(z,0)(\theta)\right. \\
&~~~~~~~~~~~~~~~~~\left.+(D_{w,w_{x},w_{xx}}f_{2}^{(1,2)}(z,0,0))U_{2}^{(2,d)}(z,0)(\theta)\right),
\end{aligned}\end{eqnarray}
where $f_{2}^{1}(z,0,0)=f_{2}^{(1,1)}(z,0,0)+f_{2}^{(1,2)}(z,0,0)$,
\begin{eqnarray}\begin{aligned}
D_{w,w_{x},w_{xx}}f_{2}^{(1,2)}(z,0,0)&=\left(D_{w}f_{2}^{(1,2)}(z,0,0),D_{w_{x}}f_{2}^{(1,2)}(z,0,0),D_{w_{xx}}f_{2}^{(1,2)}(z,0,0)\right), \\
U_{2}^{1}(z,0)&=\left(M_{2}^{1}\right)^{-1}\operatorname{Proj}_{\operatorname{Im}\left(M_{2}^{1}\right)}f_{2}^{1}(z,0,0),~ U_{2}^{2}(z,0)(\theta)=\left(M_{2}^{2}\right)^{-1}f_{2}^{2}(z,0,0),
\end{aligned}\end{eqnarray}
and
\begin{eqnarray}
U_{2}^{(2,d)}(z,0)(\theta)=\operatorname{col}\left(U_{2}^{2}(z,0)(\theta),U_{2 x}^{2}(z,0)(\theta),U_{2xx}^{2}(z,0)(\theta)\right).
\end{eqnarray}

We calculate $\operatorname{Proj}_{S}\widetilde{f}_{3}^{1}(z,0,0)$ in the following four steps.
\par~~~

\noindent {\bf{Step 1: Calculation of $\operatorname{Proj}_{S} f_{3}^{1}(z,0,0)$}}

Writing $F_{3}\left(\Phi(\theta)z_{x},0\right)$ as follows
\begin{eqnarray}\begin{aligned}
&F_{3}\left(\Phi(\theta)z_{x},0\right)=\sum_{q_{1}+q_{2}=3}A_{q_{1}q_{2}}z_{1}^{q_{1}}z_{2}^{q_{2}}\gamma_{n_{c}}^{3}(x),
\end{aligned}\end{eqnarray}
where $A_{q_{1}q_{2}}=\overline{A}_{q_{2}q_{1}}$ with $q_{1}, q_{2} \in \mathbb{N}_{0}$. From (2.24) and (2.32), we have $\widetilde{F}_{3}\left(\Phi(\theta) z_{x}, 0\right)=F_{3}\left(\Phi(\theta) z_{x}, 0\right)$, and by noticing that
\begin{eqnarray*}
\int_{0}^{\ell\pi}\gamma_{n_{c}}^{4}(x)dx=\left\{\begin{aligned}
&\frac{3}{2\ell\pi}, & n_{c}\in\mathbb{N}, \\
&\frac{1}{\ell\pi}, & n_{c}=0,
\end{aligned}\right.
\end{eqnarray*}
we have
\begin{eqnarray*}
\operatorname{Proj}_{S}f_{3}^{1}(z,0,0)=\mathcal{H}\left(B_{21}z_{1}^{2}z_{2}\right),
\end{eqnarray*}
where
\begin{eqnarray}
B_{21}=\left\{\begin{aligned}
&\frac{3}{2\ell\pi}\psi^{T}A_{21}, & n_{c}\in\mathbb{N}, \\
&\frac{1}{\ell\pi}\psi^{T}A_{21}, & n_{c}=0.
\end{aligned}\right.\end{eqnarray}
\par~~~

\noindent {\bf{Step 2: Calculation of $\operatorname{Proj}_{S}\left((D_{z}f_{2}^{1}(z,0,0))U_{2}^{1}(z,0)\right)$}}

Form (2.23) and (2.31), we have
\begin{eqnarray}
\widetilde{F}_{2}\left(\Phi(\theta)z_{x},0\right)=F_{2}\left(\Phi(\theta)z_{x},0\right)+F_{20}^{d}\left(\Phi(\theta) z_{x}\right).
\end{eqnarray}

By (2.11), we write
\begin{eqnarray}\begin{aligned}
&F_{2}\left(\Phi(\theta)z_{x}+w,\mu\right)=F_{2}\left(\Phi(\theta)z_{x}+w,0\right) \\
&=\sum_{q_{1}+q_{2}=2}A_{q_{1}q_{2}}z_{1}^{q_{1}}z_{2}^{q_{2}}\gamma_{n_{c}}^{2}(x)+S_{2}\left(\Phi(\theta)z_{x},w\right)+O\left(|w|^{2}\right),
\end{aligned}\end{eqnarray}
where $S_{2}\left(\Phi(\theta)z_{x},w\right)$ is the product term of $\Phi(\theta)z_{x}$ and $w$.

By (2.31) and (2.33), we write
\begin{eqnarray}
F_{2}^{d}\left(\Phi(\theta)z_{x},0\right)=F_{20}^{d}\left(\Phi(\theta)z_{x}\right)=\frac{n_{c}^{2}}{\ell^{2}}\sum_{q_{1}+q_{2}=2}A_{q_{1}q_{2}}^{d} z_{1}^{q_{1}}z_{2}^{q_{2}}\left(\xi_{n_{c}}^{2}(x)-\gamma_{n_{c}}^{2}(x)\right),
\end{eqnarray}
where $\xi_{n_{c}}(x)=(\sqrt{2}/\sqrt{\ell \pi})\sin \left((n_{c}/\ell)x\right)$, and
\begin{eqnarray}\begin{aligned}
&\left\{\begin{array}{l}
A_{20}^{d}=-2d_{21}\tau_{c}\left(\begin{array}{c}
0 \\
\phi_{1}(-1)\phi_{2}(0)
\end{array}\right)=\overline{A_{02}^{d}}, \\
A_{11}^{d}=-4d_{21}\tau_{c}\left(\begin{array}{c}
0 \\
\operatorname{Re}\left\{\phi_{1}(-1)\overline{\phi_{2}}(0)\right\}
\end{array}\right).
\end{array}\right.
\end{aligned}\end{eqnarray}

From (2.1), it is easy to verify that
\begin{eqnarray*}
\int_{0}^{\ell\pi}\gamma_{n_{c}}^{3}(x)dx=\left\{\begin{aligned}
&0, & n_{c}\in\mathbb{N}, \\
&\frac{1}{\sqrt{\ell\pi}}, & n_{c}=0,
\end{aligned}\right.,~\int_{0}^{\ell\pi}\xi_{n_{c}}^{2}(x)\gamma_{n_{c}}(x)dx=0.
\end{eqnarray*}

Then from (2.43), (2.44) and (2.45), we have
\begin{eqnarray}\begin{aligned}
&f_{2}^{1}(z,0,0)=\Psi(0)\left(\begin{aligned}
&\left[\widetilde{F}_{2}\left(\Phi(\theta)z_{x},0\right),\beta_{n_{c}}^{(1)}\right] \\
&\left[\widetilde{F}_{2}\left(\Phi(\theta)z_{x},0\right),\beta_{n_{c}}^{(2)}\right]
\end{aligned}\right) \\
&=\left\{\begin{aligned}
&(0,0)^{T}, & n_{c}\in\mathbb{N}, \\
&\frac{1}{\sqrt{\ell\pi}}\mathcal{H}(\psi^{T}(A_{20}z_{1}^{2}+A_{02}z_{2}^{2}+A_{11}z_{1}z_{2})), & n_{c}=0.
\end{aligned}\right.
\end{aligned}\end{eqnarray}

Hence, by combining with (2.30) and (2.47), we have
\begin{eqnarray*}\begin{aligned}
\operatorname{Proj}_{S}\left((D_{z}f_{2}^{1}(z,0,0))U_{2}^{1}(z,0)\right)=\mathcal{H}(B_{22}z_{1}^{2}z_{2}),
\end{aligned}\end{eqnarray*}
where
\begin{eqnarray}
B_{22}=\left\{\begin{aligned}
&0, & n_{c}\in\mathbb{N}, \\
&\frac{1}{i\omega_{c}\ell\pi}(-(\psi^{T}A_{20})(\psi^{T}A_{11})+|\psi^{T}A_{11}|^{2}+\frac{2}{3}|\psi^{T}A_{02}|^{2}), &n_{c}=0.
\end{aligned}\right.
\end{eqnarray}
\par~~~

\noindent {\bf{Step 3: Calculation of $\operatorname{Proj}_{S}\left((D_{w}f_{2}^{(1,1)}(z,0,0))U_{2}^{2}(z,0)(\theta)\right)$}}

Let
\begin{eqnarray*}
U_{2}^{2}(z,0)(\theta)\triangleq h(\theta,z)=\sum_{n \in \mathbb{N}_{0}}h_{n}(\theta,z)\gamma_{n}(x),
\end{eqnarray*}
where $h_{n}(\theta,z)=\sum_{q_{1}+q_{2}=2}h_{n,q_{1}q_{2}}(\theta)z_{1}^{q_{1}}z_{2}^{q_{2}}$. Then we have
\begin{eqnarray*}\begin{aligned}
&\left(\begin{aligned}
\left[S_{2}\left(\Phi(\theta)z_{x},\sum_{n \in \mathbb{N}_{0}} h_{n}(\theta,z)\gamma_{n}(x)\right),\beta_{n_{c}}^{(1)}\right] \\
\left[S_{2}\left(\Phi(\theta)z_{x},\sum_{n \in \mathbb{N}_{0}} h_{n}(\theta,z)\gamma_{n}(x)\right),\beta_{n_{c}}^{(2)}\right]
\end{aligned}\right) \\
&=\sum_{n \in \mathbb{N}_{0}} b_{n}\left(S_{2}\left(\phi(\theta)z_{1},h_{n}(\theta,z)\right)+S_{2}\left(\overline{\phi}(\theta)z_{2},h_{n}(\theta,z)\right)\right),
\end{aligned}\end{eqnarray*}
where
\begin{eqnarray*}
b_{n}=\left\{\begin{aligned}
&\int_{0}^{\ell\pi}\gamma_{n_{c}}^{2}(x)\gamma_{n}(x) dx=\left\{\begin{array}{cc}
\frac{1}{\sqrt{\ell\pi}}, & n=0, \\
\frac{1}{\sqrt{2\ell\pi}}, & n=2n_{c}, \\
0, & \text { otherwise },
\end{array}\right.~n_{c}\in\mathbb{N}, \\
&\int_{0}^{\ell\pi}\gamma_{n_{c}}^{2}(x)\gamma_{n}(x) dx=\left\{\begin{array}{cc}
\frac{1}{\sqrt{\ell\pi}}, & n=0, \\
0, & n\neq 0,
\end{array}\right.~n_{c}=0.
\end{aligned}\right.\end{eqnarray*}

Hence, we have
\begin{eqnarray*}\begin{aligned}
&(D_{w}f_{2}^{(1,1)}(z,0,0))U_{2}^{2}(z,0)(\theta) \\
&=\left\{\begin{aligned}
&\Psi(0)\left(\sum_{n=0,2n_{c}}b_{n}\left(S_{2}\left(\phi(\theta)z_{1},h_{n}(\theta,z)\right)+S_{2}\left(\overline{\phi}(\theta)z_{2},h_{n}(\theta,z)\right)\right)\right), &n_{c}\in\mathbb{N} \\
&\Psi(0)b_{0}\left(S_{2}\left(\phi(\theta)z_{1},h_{0}(\theta,z)\right)+S_{2}\left(\overline{\phi}(\theta)z_{2},h_{0}(\theta,z)\right)\right), &n_{c}=0,
\end{aligned}\right.
\end{aligned}\end{eqnarray*}
and
\begin{eqnarray*}
\operatorname{Proj}_{S}\left((D_{w}f_{2}^{(1,1)}(z,0,0))U_{2}^{2}(z,0)(\theta)\right)=\mathcal{H}\left(B_{22}z_{1}^{2}z_{2}\right),
\end{eqnarray*}
where
\begin{eqnarray}
B_{23}=\left\{\begin{aligned}
&\frac{1}{\sqrt{\ell\pi}}\psi^{T}\left(S_{2}\left(\phi(\theta),h_{0,11}(\theta)\right)+S_{2}\left(\overline{\phi}(\theta), h_{0,20}(\theta)\right)\right) \\
&+\frac{1}{\sqrt{2\ell\pi}}\psi^{T}\left(S_{2}\left(\phi(\theta),h_{2n_{c},11}(\theta)\right)+S_{2}\left(\overline{\phi}(\theta),h_{2n_{c},20}(\theta)\right)\right), ~n_{c}\in\mathbb{N}, \\
&\frac{1}{\sqrt{\ell\pi}}\psi^{T}\left(S_{2}\left(\phi(\theta),h_{0,11}(\theta)\right)+S_{2}\left(\overline{\phi}(\theta), h_{0,20}(\theta)\right)\right),~n_{c}=0.
\end{aligned}\right.
\end{eqnarray}
\par~~~

\noindent {\bf{Step 4: Calculation of $\operatorname{Proj}_{S}\left((D_{w,w_{x},w_{xx}}f_{2}^{(1,2)}(z,0,0))U_{2}^{(2,d)}(z,0)(\theta)\right)$}}

Denote $\varphi(\theta)=\left(\varphi_{1}(\theta),\varphi_{2}(\theta)\right)^{T}=\Phi(\theta)z_{x}$ and
\begin{eqnarray*}\begin{aligned}
&F_{2}^{d}\left(\varphi(\theta),w,w_{x},w_{xx}\right)=F_{2}^{d}(\varphi(\theta)+w,0)=F_{20}^{d}(\varphi(\theta)+w) \\
&=-2d_{21}\tau_{c}\left(\begin{array}{c}
0 \\
\left(\varphi^{(1)}_{xx}(-1)+(w_{1})_{xx}(-1)\right)\left(\varphi^{(2)}(0)+w_{2}(0)\right)
\end{array}\right) \\
&-2d_{21}\tau_{c}\left(\begin{array}{c}
0 \\
\left(\varphi^{(1)}_{x}(-1)+(w_{1})_{x}(-1)\right)\left(\varphi^{(2)}_{x}(0)+(w_{2})_{x}(0)\right)
\end{array}\right).
\end{aligned}\end{eqnarray*}

Furthermore, from (2.37), (2.39) and (2.40), we have
\begin{eqnarray*}\begin{aligned}
&(D_{w,w_{x},w_{xx}}f_{2}^{(1,2)}(z,0,0))U_{2}^{(2,d)}(z,0)(\theta) \\
&=\Psi(0)\left(\begin{aligned}
&\left[D_{w,w_{x},w_{xx}}F_{2}^{d}\left(\varphi(\theta),w,w_{x},w_{xx}\right)U_{2}^{(2,d)}(z,0)(\theta),\beta_{n_{c}}^{(1)}\right] \\
&\left[D_{w,w_{x},w_{xx}}F_{2}^{d}\left(\varphi(\theta),w,w_{x},w_{xx}\right)U_{2}^{(2,d)}(z,0)(\theta),\beta_{n_{c}}^{(2)}\right]
\end{aligned}\right),
\end{aligned}\end{eqnarray*}
and then we obtain
\begin{eqnarray*}
\operatorname{Proj}_{S}\left((D_{w,w_{x},w_{xx}}f_{2}^{(1,2)}(z,0,0))U_{2}^{(2,d)}(z,0)(\theta)\right)=\mathcal{H}\left(B_{23}z_{1}^{2}z_{2}\right),
\end{eqnarray*}
where
\begin{eqnarray}
B_{24}=\left\{\begin{aligned}
&-\frac{1}{\sqrt{\ell \pi}}(n_{c}/\ell)^{2}\psi^{T}\left(S_{2}^{(d,1)}\left(\phi(\theta),h_{0,11}(\theta)\right)+S_{2}^{(d,1)}\left(\overline{\phi}(\theta), h_{0,20}(\theta)\right)\right) \\
&+\frac{1}{\sqrt{2\ell\pi}}\psi^{T}\sum_{j=1,2,3}b_{2n_{c}}^{(j)}\left(S_{2}^{(d,j)}\left(\phi(\theta),h_{2n_{c},11}(\theta)\right)+S_{2}^{(d,j)}\left(\overline{\phi}(\theta), h_{2n_{c},20}(\theta)\right)\right), & n_{c}\in\mathbb{N}, \\
& 0, & n_{c}=0
\end{aligned}\right.
\end{eqnarray}
with
\begin{eqnarray*}
b_{2 n_{c}}^{(1)}=-\frac{n_{c}^{2}}{\ell^{2}},~b_{2n_{c}}^{(2)}=2\frac{n_{c}^{2}}{\ell^{2}},~b_{2n_{c}}^{(3)}=-\frac{(2n_{c})^{2}}{\ell^{2}}.
\end{eqnarray*}

Furthermore, for $\phi(\theta)=\left(\phi_{1}(\theta),\phi_{2}(\theta)\right)^{T},~y(\theta)=\left(y_{1}(\theta),y_{2}(\theta)\right)^{T} \in C\left([-1,0], \mathbb{R}^{2}\right)$, we have
\begin{eqnarray*}
\left\{\begin{aligned}
&S_{2}^{(d,1)}(\phi(\theta),y(\theta))=-2d_{21}\tau_{c}\left(\begin{array}{c}
0 \\
\phi_{1}(-1)y_{2}(0)
\end{array}\right), \\
&S_{2}^{(d,2)}(\phi(\theta),y(\theta))=-2d_{21}\tau_{c}\left(\begin{array}{c}
0 \\
\phi_{2}(0)y_{1}(-1)+\phi_{1}(-1)y_{2}(0)
\end{array}\right), \\
&S_{2}^{(d,3)}(\phi(\theta),y(\theta))=-2d_{21}\tau_{c}\left(\begin{array}{c}
0 \\
\phi_{2}(0)y_{1}(-1)
\end{array}\right).
\end{aligned}\right.
\end{eqnarray*}

\section{Normal form of the Hopf bifurcation and the corresponding coefficients}
\label{sec:3}

According to the algorithm developed in Section 2, we obtain the normal form of the Hopf bifurcation truncated to the third-order term
\begin{eqnarray}
\dot{z}=Bz+\frac{1}{2}\left(\begin{array}{c}
B_{1}z_{1} \mu \\
\overline{B}_{1}z_{2}\mu
\end{array}\right)+\frac{1}{3!}\left(\begin{array}{c}
B_{2}z_{1}^{2}z_{2} \\
\overline{B}_{2}z_{1}z_{2}^{2}
\end{array}\right)+O\left(|z| \mu^{2}+\left|z\right|^{4}\right),
\end{eqnarray}
where
\begin{eqnarray*}\begin{aligned}
B_{1}&=\left\{\begin{aligned}
&2\psi^{T}(0)\left(A_{1}\phi(0)+A_{2}\phi(-1)-\frac{n_{c}^{2}}{\ell^{2}}\left(D_{1}\phi(0)+D_{2}\phi(-1)\right)\right), &n_{c}\in\mathbb{N}, \\
&2\psi^{T}(0)\left(A_{1}\phi(0)+A_{2}\phi(-1)\right), &n_{c}=0,
\end{aligned}\right. \\
B_{2}&=B_{21}+\frac{3}{2}\left(B_{22}+B_{23}+B_{24}\right).
\end{aligned}\end{eqnarray*}
Here, $B_{1}$ is determined by (2.36), $B_{21}$, $B_{22}$ and $B_{23}$ are determined by (2.42), (2.48), (2.49), (2.50), respectively, and they can be calculated by using the MATLAB software. The normal form (3.1) can be written in real coordinates through the change of variables $z_{1}=v_{1}-iv_{2},~z_{2}=v_{1}+iv_{2}$, and then changing to polar coordinates by $v_{1}=\rho \cos \Theta,~v_{2}=\rho \sin \Theta$, where $\Theta$ is the azimuthal angle. Therefore, by the above transformation and removing the azimuthal term $\Theta$, (3.1) can be rewritten as
\begin{eqnarray*}
\dot{\rho}=K_{1}\mu\rho+K_{2}\rho^{3}+O\left(\mu^{2}\rho+\lvert(\rho,\mu)\rvert^{4}\right),
\end{eqnarray*}
where
\begin{eqnarray*}
K_{1}=\frac{1}{2}\operatorname{Re}\left(B_{1}\right),~K_{2}=\frac{1}{3!}\operatorname{Re}\left(B_{2}\right).
\end{eqnarray*}

According to \cite{lv40}, the sign of $K_{1}K_{2}$ determines the direction of the Hopf bifurcation, and the sign of $K_{2}$ determines the stability of the Hopf bifurcation periodic solution. More precisely, we have the following results

(i) when $K_{1}K_{2}<0$, the Hopf bifurcation is supercritical, and the Hopf bifurcation periodic solution is stable for $K_{2}<0$ and unstable for $K_{2}>0$;

(ii) when $K_{1}K_{2}>0$, the Hopf bifurcation is subcritical, and the Hopf bifurcation periodic solution is stable for $K_{2}<0$ and unstable for $K_{2}>0$.

From (2.42), (2.48), (2.49) and (2.50), it is obvious that in order to obtain the value of $K_{2}$, we still need to calculate $h_{0,20}(\theta), h_{0,11}(\theta), h_{2 n_{c},20}(\theta), h_{2 n_{c},11}(\theta)$ and $A_{ij}$.

\subsection{Calculations of $h_{0,20}(\theta), h_{0,11}(\theta), h_{2n_{c},20}(\theta)$ and $h_{2n_{c},11}(\theta)$}

From \cite{lv38}, we have
\begin{eqnarray*}
M_{2}^{2}\left(h_{n}(\theta,z) \gamma_{n}(x)\right)=D_{z}\left(h_{n}(\theta,z) \gamma_{n}(x)\right)Bz-A_{\mathcal{Q}^{1}}\left(h_{n}(\theta,z)\gamma_{n}(x)\right),
\end{eqnarray*}
which leads to
\begin{eqnarray}\begin{aligned}
&\left(\begin{aligned}
\left[M_{2}^{2}\left(h_{n}(\theta,z)\gamma_{n}(x)\right),\beta_{n}^{(1)}\right] \\
\left[M_{2}^{2}\left(h_{n}(\theta,z)\gamma_{n}(x)\right),\beta_{n}^{(2)}\right]
\end{aligned}\right) \\
&=2i\omega_{c}\left(h_{n,20}(\theta)z_{1}^{2}-h_{n,02}(\theta)z_{2}^{2}\right)-\left(\dot{h}_{n}(\theta,z)+X_{0}(\theta)\left(\mathscr{L}_{0}\left(h_{n}(\theta,z)\right)-\dot{h}_{n}(0,z)\right)\right),
\end{aligned}\end{eqnarray}
where
\begin{eqnarray*}
\mathscr{L}_{0}\left(h_{n}(\theta,z)\right)=-\tau_{c}(n/\ell)^{2}\left(D_{1}h_{n}(0,z)+D_{2}h_{n}(-1,z)\right)+\tau_{c}(A_{1}h_{n}(0,z)+A_{2}h_{n}(-1,z)).
\end{eqnarray*}

By (2.19) and the second mathematical expression in (2.25), we have
\begin{eqnarray}\begin{aligned}
f_{2}^{2}(z,0,0)&=X_{0}(\theta)\widetilde{F}_{2}\left(\Phi(\theta)z_{x},0\right)-\pi\left(X_{0}(\theta)\widetilde{F}_{2}\left(\Phi(\theta)z_{x},0\right)\right) \\
&=X_{0}(\theta)\widetilde{F}_{2}\left(\Phi(\theta)z_{x},0\right)-\Phi(\theta)\Psi(0)\left(\begin{aligned}
\left[\widetilde{F}_{2}\left(\Phi(\theta)z_{x},0\right),\beta_{n_{c}}^{(1)}\right] \\
\left[\widetilde{F}_{2}\left(\Phi(\theta)z_{x},0\right),\beta_{n_{c}}^{(2)}\right]
\end{aligned}\right)\gamma_{n_{c}}(x).
\end{aligned}\end{eqnarray}

Furthermore, by (2.43), (2.44) and (2.45), when $n_{c}\in\mathbb{N}$, we have
\begin{eqnarray}
\left(\begin{aligned}
\left[f_{2}^{2}(z,0,0),\beta_{n}^{(1)}\right] \\
\left[f_{2}^{2}(z,0,0),\beta_{n}^{(2)}\right]
\end{aligned}\right)=\left\{\begin{aligned}
& \frac{1}{\sqrt{\ell\pi}}X_{0}(\theta)\left(A_{20}z_{1}^{2}+A_{02}z_{2}^{2}+A_{11}z_{1}z_{2}\right), & n=0, \\
& \frac{1}{\sqrt{2\ell\pi}}X_{0}(\theta)\left(\widetilde{A}_{20}z_{1}^{2}+\widetilde{A}_{02}z_{2}^{2}+\widetilde{A}_{11}z_{1}z_{2}\right), & n=2n_{c},
\end{aligned}\right.
\end{eqnarray}
where $\widetilde{A}_{j_{1}j_{2}}$ is defined as follows
\begin{eqnarray}
\left\{\begin{aligned}
&\widetilde{A}_{j_{1}j_{2}}=A_{j_{1}j_{2}}-2\left(n_{c}/\ell\right)^{2}A_{j_{1}j_{2}}^{d}, \\
&j_{1},j_{2}=0,1,2,~j_{1}+j_{2}=2,
\end{aligned}\right.
\end{eqnarray}
where $A_{j_{1}j_{2}}^{d}$ is determined by (2.46), and $A_{j_{1}j_{2}}$ will be calculated in the following section. When $n_{c}=0$, we have
\begin{eqnarray}
\left(\begin{aligned}
\left[f_{2}^{2}(z,0,0),\beta_{n}^{(1)}\right] \\
\left[f_{2}^{2}(z,0,0),\beta_{n}^{(2)}\right]
\end{aligned}\right)=\frac{1}{\sqrt{\ell\pi}}\left(X_{0}(\theta)-\Phi(\theta)\Psi(0)\right)\left(A_{20}z_{1}^{2}+A_{02}z_{2}^{2}+A_{11}z_{1}z_{2}\right),~n=0.
\end{eqnarray}

Therefore, from (3.2), (3.3), (3.4), (3.6), and by matching the coefficients of $z_{1}^{2}$ and $z_{1}z_{2}$, when $n_{c}\in\mathbb{N}$, we have
\begin{eqnarray}
n=0,~\left\{\begin{array}{l}
z_{1}^{2}: \left\{\begin{aligned}
&\dot{h}_{0,20}(\theta)-2i\omega_{c}h_{0,20}(\theta)=(0,0)^{T}, \\
&\dot{h}_{0,20}(0)-L_{0}\left(h_{0,20}(\theta)\right)=\frac{1}{\sqrt{\ell\pi}}A_{20},
\end{aligned}\right. \\
z_{1} z_{2}: \left\{\begin{aligned}
&\dot{h}_{0,11}(\theta)=(0,0)^{T}, \\
&\dot{h}_{0,11}(0)-L_{0}\left(h_{0,11}(\theta)\right)=\frac{1}{\sqrt{\ell\pi}}A_{11}
\end{aligned}\right.
\end{array}\right.
\end{eqnarray}
and
\begin{eqnarray}
n=2n_{c},~\left\{\begin{array}{l}
z_{1}^{2}: \left\{\begin{aligned}
&\dot{h}_{2n_{c},20}(\theta)-2i\omega_{c}h_{2n_{c},20}(\theta)=(0,0)^{T}, \\
&\dot{h}_{2n_{c},20}(0)-\mathscr{L}_{0}\left(h_{2n_{c},20}(\theta)\right)=\frac{1}{\sqrt{2\ell\pi}}\widetilde{A}_{20},
\end{aligned}\right. \\
z_{1} z_{2}: \left\{\begin{aligned}
&\dot{h}_{2n_{c},11}(\theta)=(0,0)^{T}, \\
&\dot{h}_{2n_{c},11}(0)-\mathscr{L}_{0}\left(h_{2n_{c},11}(\theta)\right)=\frac{1}{\sqrt{2\ell\pi}}\widetilde{A}_{11}.
\end{aligned}\right. \\
\end{array}\right.
\end{eqnarray}

When $n_{c}=0$, we have
\begin{eqnarray*}
n=0,~\left\{\begin{array}{l}
z_{1}^{2}: \left\{\begin{aligned}
&\dot{h}_{0,20}(\theta)-2i\omega_{c}h_{0,20}(\theta)=\frac{1}{\sqrt{\ell\pi}}\Phi(\theta)\Psi(0)A_{20}, \\
&\dot{h}_{0,20}(0)-L_{0}\left(h_{0,20}(\theta)\right)=\frac{1}{\sqrt{\ell\pi}}A_{20},
\end{aligned}\right. \\
z_{1} z_{2}: \left\{\begin{aligned}
&\dot{h}_{0,11}(\theta)=\frac{1}{\sqrt{\ell\pi}}\Phi(\theta)\Psi(0)A_{11}, \\
&\dot{h}_{0,11}(0)-L_{0}\left(h_{0,11}(\theta)\right)=\frac{1}{\sqrt{\ell\pi}}A_{11}.
\end{aligned}\right.
\end{array}\right.
\end{eqnarray*}

Next, by combining with (3.7) and (3.8), we will give the mathematical expressions of $h_{2n_{c},20}(\theta)$ and $h_{2n_{c},11}(\theta)$ for $n_{c}\in\mathbb{N}$, and the mathematical expressions of $h_{0,20}(\theta)$ and $h_{0,11}(\theta)$ for $n_{c}\in\mathbb{N}$ and $n_{c}=0$, respectively.

\par~~~
\par\noindent (1) Calculations of $h_{0,20}(\theta)$ and $ h_{0,11}(\theta)$ for $n_{c}\in\mathbb{N}$

(i) Notice that
\begin{eqnarray}
\left\{\begin{aligned}
\dot{h}_{0,20}(\theta)-2i\omega_{c}h_{0,20}(\theta)&=(0,0)^{T}, \\
\dot{h}_{0,20}(0)-L_{0}\left(h_{0,20}(\theta)\right)&=\frac{1}{\sqrt{\ell\pi}}A_{20},
\end{aligned}\right.
\end{eqnarray}
then from (3.9), we have $h_{0,20}(\theta)=e^{2i\omega_{c}\theta}h_{0,20}(0)$, and hence $h_{0,20}(-1)=e^{-2i\omega_{c}}h_{0,20}(0)$. Furthermore, from (3.9) and $L_{0}\left(h_{0,20}(\theta)\right)=\tau_{c}(A_{1} h_{0,20}(0)+A_{2}h_{0,20}(-1))$, we have
\begin{eqnarray}
2i\omega_{c}h_{0,20}(0)=\frac{1}{\sqrt{\ell\pi}}A_{20}+\tau_{c}(A_{1}h_{0,20}(0)+A_{2}h_{0,20}(-1)).
\end{eqnarray}

Therefore, by combining with $h_{0,20}(-1)=e^{-2i\omega_{c}}h_{0,20}(0)$ and (3.10), we can obtain
\begin{eqnarray*}
(2i\omega_{c}I_{2}-\tau_{c} A_{1}-\tau_{c}A_{2}e^{-2i\omega_{c}})h_{0,20}(0)=\frac{1}{\sqrt{\ell \pi}}A_{20},
\end{eqnarray*}
and hence $h_{0,20}(\theta)=e^{2i\omega_{c}\theta}C_{1}$ with
\begin{eqnarray*}\begin{aligned}
C_{1}&=(2i\omega_{c}I_{2}-\tau_{c} A_{1}-\tau_{c}A_{2}e^{-2i\omega_{c}})^{-1}\frac{1}{\sqrt{\ell\pi}}A_{20}.
\end{aligned}\end{eqnarray*}

(ii) Notice that
\begin{eqnarray}
\left\{\begin{aligned}
&\dot{h}_{0,11}(\theta)=(0,0)^{T}, \\
&\dot{h}_{0,11}(0)-L_{0}\left(h_{0,11}(\theta)\right)=\frac{1}{\sqrt{\ell\pi}}A_{11},
\end{aligned}\right.
\end{eqnarray}
then from (3.11), we have $h_{0,11}(\theta)=h_{0,11}(0)$, and hence $h_{0,11}(-1)=h_{0,11}(0)$. Furthermore, from (3.11) and $L_{0}\left(h_{0,11}(\theta)\right)=\tau_{c}(A_{1}h_{0,11}(0)+A_{2}h_{0,11}(-1))$, we have
\begin{eqnarray}
(0,0)^{T}=\tau_{c}(A_{1}h_{0,11}(0)+A_{2}h_{0,11}(-1))+\frac{1}{\sqrt{\ell\pi}}A_{11}.
\end{eqnarray}

Therefore, by combining with $h_{0,11}(-1)=h_{0,11}(0)$ and (3.12), we can obtain
\begin{eqnarray*}
(-\tau_{c} A_{1}-\tau_{c}A_{2})h_{0,11}(0)=\frac{1}{\sqrt{\ell \pi}}A_{11},
\end{eqnarray*}
and hence $h_{0,11}(\theta)=C_{2}$ with
\begin{eqnarray*}
C_{2}=(-\tau_{c}A_{1}-\tau_{c}A_{2})^{-1}\frac{1}{\sqrt{\ell \pi}}A_{11}.
\end{eqnarray*}
\par~~~
\par\noindent (2) Calculations of $h_{2n_{c},20}(\theta)$ and $h_{2n_{c},11}(\theta)$ for $n_{c}\in\mathbb{N}$

(i) Notice that
\begin{eqnarray}
\left\{\begin{aligned}
\dot{h}_{2n_{c},20}(\theta)-2i\omega_{c}h_{2n_{c},20}(\theta)&=(0,0)^{T}, \\
\dot{h}_{2n_{c},20}(0)-\mathscr{L}_{0}\left(h_{2n_{c},20}(\theta)\right)&=\frac{1}{\sqrt{2\ell\pi}}\widetilde{A}_{20},
\end{aligned}\right.
\end{eqnarray}
then from (3.13), we have $h_{2n_{c},20}(\theta)=e^{2i\omega_{c}\theta}h_{2n_{c},20}(0)$, and hence $h_{2n_{c},20}(-1)=e^{-2i\omega_{c}}h_{2n_{c},20}(0)$. Furthermore, from (3.13) and
\begin{eqnarray*}
\mathscr{L}_{0}\left(h_{2n_{c},20}(\theta)\right)=-\tau_{c}\frac{4n_{c}^{2}}{\ell^{2}}\left(D_{1}h_{2n_{c},20}(0)+D_{2}h_{2n_{c},20}(-1)\right)+\tau_{c}A_{1}h_{2n_{c},20}(0)+\tau_{c}A_{2} h_{2n_{c},20}(-1),
\end{eqnarray*}
we have
\begin{eqnarray}
2i\omega_{c}h_{2n_{c},20}(0)=\frac{1}{\sqrt{2\ell\pi}} \widetilde{A}_{20}-\tau_{c}\frac{4n_{c}^{2}}{\ell^{2}}\left(D_{1}h_{2n_{c},20}(0)+D_{2}h_{2n_{c},20}(-1)\right)+\tau_{c}A_{1}h_{2n_{c},20}(0)+\tau_{c}A_{2} h_{2n_{c},20}(-1).
\end{eqnarray}

Therefore, by combining with $h_{2n_{c},20}(-1)=e^{-2i\omega_{c}}h_{2n_{c},20}(0)$ and (3.14), we can obtain
\begin{eqnarray*}
(2i\omega_{c}I_{2}+\tau_{c}\frac{4n_{c}^{2}}{\ell^{2}}D_{1}+\tau_{c}\frac{4n_{c}^{2}}{\ell^{2}}D_{2}e^{-2i\omega_{c}}-\tau_{c} A_{1}-\tau_{c}A_{2}e^{-2i\omega_{c}})h_{0200}(0)=\frac{1}{\sqrt{2\ell\pi}}\widetilde{A}_{20},
\end{eqnarray*}
and hence $h_{2n_{c},20}(\theta)=e^{2i\omega_{c}\theta}C_{3}$ with
\begin{eqnarray*}
C_{3}=(2i\omega_{c}I_{2}+\tau_{c}\frac{4n_{c}^{2}}{\ell^{2}}D_{1}+\tau_{c}\frac{4n_{c}^{2}}{\ell^{2}}D_{2}e^{-2i\omega_{c}}-\tau_{c} A_{1}-\tau_{c}A_{2}e^{-2i\omega_{c}})^{-1}\frac{1}{\sqrt{2\ell\pi}}\widetilde{A}_{20}.
\end{eqnarray*}

Here, $A_{20}^{d}$ and $\widetilde{A}_{20}$ are defined by (2.46) and (3.5), respectively.

(ii) Notice that
\begin{eqnarray}
\left\{\begin{aligned}
&\dot{h}_{2n_{c},11}(\theta)=(0,0)^{T}, \\
&\dot{h}_{2n_{c},11}(0)-\mathscr{L}_{0}\left(h_{2n_{c},11}(\theta)\right)=\frac{1}{\sqrt{2\ell\pi}}\widetilde{A}_{11},
\end{aligned}\right.
\end{eqnarray}
then from (3.15), we have $h_{2n_{c},11}(\theta)=h_{2n_{c},11}(0)$, and hence $h_{2n_{c},11}(-1)=h_{2n_{c},11}(0)$. Furthermore, from (3.15) and
\begin{eqnarray*}
\mathscr{L}_{0}\left(h_{2n_{c},11}(\theta)\right)=-\tau_{c}\frac{4n_{c}^{2}}{\ell^{2}}\left(D_{1}h_{2n_{c},11}(0)+D_{2}h_{2n_{c},11}(-1)\right)+\tau_{c}A_{1} h_{2n_{c},11}(0)+\tau_{c}A_{2}h_{2n_{c},11}(-1),
\end{eqnarray*}
we have
\begin{eqnarray}
(0,0)^{T}=-\tau_{c}\frac{4n_{c}^{2}}{\ell^{2}}\left(D_{1}h_{2n_{c},11}(0)+D_{2}h_{2n_{c},11}(-1)\right)+\tau_{c}A_{1}h_{2n_{c},11}(0)+\tau_{c}A_{2}h_{2n_{c},11}(-1)+\frac{1}{\sqrt{2\ell\pi}} \widetilde{A}_{11}.
\end{eqnarray}

Therefore, by combining with $h_{2n_{c},11}(-1)=h_{2n_{c},11}(0)$ and (3.16), we can obtain
\begin{eqnarray*}
\left(\tau_{c}\frac{4n_{c}^{2}}{\ell^{2}}D_{1}+\tau_{c}\frac{4n_{c}^{2}}{\ell^{2}}D_{2}-\tau_{c}A_{1}-\tau_{2}A_{2}\right)h_{2n_{c},11}(0)=\frac{1}{\sqrt{2\ell\pi}}\widetilde{A}_{11},
\end{eqnarray*}
and hence $h_{2n_{c},11}(\theta)=C_{4}$ with
\begin{eqnarray*}
C_{4}=\left(\tau_{c}\frac{4n_{c}^{2}}{\ell^{2}}D_{1}+\tau_{c}\frac{4n_{c}^{2}}{\ell^{2}}D_{2}-\tau_{c}A_{1}-\tau_{c}A_{2}\right)^{-1}\frac{1}{\sqrt{2\ell\pi}}\widetilde{A}_{11}.
\end{eqnarray*}

Here, $A_{11}^{d}$ and $\widetilde{A}_{11}$ are defined by (2.46) and (3.5), respectively.
\par~~~
\par\noindent (3) Calculations of $h_{0,20}(\theta)$ and $h_{0,11}(\theta)$ for $n_{c}=0$

(i) Notice that
\begin{eqnarray}
\left\{\begin{aligned}
\dot{h}_{0,20}(\theta)-2i\omega_{c}h_{0,20}(\theta)&=\frac{1}{\sqrt{\ell\pi}}\Phi(\theta)\Psi(0)A_{20}, \\
\dot{h}_{0,20}(0)-L_{0}\left(h_{0,20}(\theta)\right)&=\frac{1}{\sqrt{\ell\pi}}A_{20},
\end{aligned}\right.
\end{eqnarray}
then from (3.17), we have
\begin{eqnarray*}
h_{0,20}(\theta)=e^{2i\omega_{c}\theta}h_{0,20}(0)+\frac{1}{\sqrt{\ell\pi}}e^{2i\omega_{c}\theta}\int_{0}^{\theta}\Phi(t)\Psi(0)A_{20}e^{-2i\omega_{c}t}dt,
\end{eqnarray*}
and hence
\begin{eqnarray}
h_{0,20}(-1)=e^{-2i\omega_{c}}h_{0,20}(0)+\frac{1}{\sqrt{\ell\pi}}e^{-2i\omega_{c}}\int_{0}^{-1}\Phi(t)\Psi(0)A_{20}e^{-2i\omega_{c}t}dt.
\end{eqnarray}
Furthermore, from (3.17), we have
\begin{eqnarray}
2i\omega_{c}h_{0,20}(0)+\frac{1}{\sqrt{\ell\pi}}\Phi(0)\Psi(0)A_{20}=\frac{1}{\sqrt{\ell\pi}}A_{20}+L_{0}\left(h_{0,20}(\theta)\right).
\end{eqnarray}

Therefore, by combining with $L_{0}\left(h_{0,20}(\theta)\right)=\tau_{c}\left(A_{1}h_{0,20}(0)+A_{2}h_{0,20}(-1)\right)$, (3.18) and (3.19), we can obtain
\begin{eqnarray*}\begin{aligned}
&\left(2i\omega_{c}I_{2}-\tau_{c}A_{1}-\tau_{c}A_{2}e^{-2i\omega_{c}}\right)h_{0,20}(0)= \\
&\left(\frac{1}{\sqrt{\ell\pi}}-\frac{1}{\sqrt{\ell\pi}}\Phi(0)\Psi(0)\right)A_{20}-\tau_{c}A_{2}\frac{1}{\sqrt{\ell\pi}}e^{-2i\omega_{c}}\int_{-1}^{0}e^{-2i\omega_{c}t}\Phi(t)\Psi(0)A_{20}dt,
\end{aligned}\end{eqnarray*}
and hence
\begin{eqnarray*}
h_{0,20}(\theta)=\frac{1}{\sqrt{\ell\pi}}e^{2i\omega_{c}\theta}\int_{0}^{\theta}\Phi(t)\Psi(0)A_{20}e^{-2i\omega_{c}t}dt+C_{5}e^{2i\omega_{c}\theta}
\end{eqnarray*}
with
\begin{eqnarray*}\begin{aligned}
C_{5}&=\left(2i\omega_{c}I_{2}-\tau_{c}A_{1}-\tau_{c}A_{2}e^{-2i\omega_{c}}\right)^{-1} \\
&\left((\frac{1}{\sqrt{\ell\pi}}-\frac{1}{\sqrt{\ell\pi}}\Phi(0)\Psi(0))A_{20}-\tau_{c}A_{2}\frac{1}{\sqrt{\ell\pi}}e^{-2i\omega_{c}}\int_{-1}^{0}e^{-2i\omega_{c}t}\Phi(t)\Psi(0)A_{20}dt\right).
\end{aligned}\end{eqnarray*}

(ii) Notice that
\begin{eqnarray}
\left\{\begin{aligned}
&\dot{h}_{0,11}(\theta)=\frac{1}{\sqrt{\ell\pi}}\Phi(\theta)\Psi(0)A_{11}, \\
&\dot{h}_{0,11}(0)-L_{0}(h_{0,11}(\theta))=\frac{1}{\sqrt{\ell\pi}}A_{11},
\end{aligned}\right.
\end{eqnarray}
then from (3.20), we have
\begin{eqnarray*}
h_{0,11}(\theta)=h_{0,11}(0)+\frac{1}{\sqrt{\ell\pi}}\int_{0}^{\theta}\Phi(t)\Psi(0)A_{11}dt,
\end{eqnarray*}
and hence
\begin{eqnarray}
h_{0,11}(-1)=h_{0,11}(0)+\frac{1}{\sqrt{\ell\pi}}\int_{0}^{-1}\Phi(t)\Psi(0)A_{11}dt.
\end{eqnarray}
Furthermore, from (3.20), we have
\begin{eqnarray}
\frac{1}{\sqrt{\ell\pi}}\Phi(0)\Psi(0)A_{11}=L_{0}(h_{0,11}(\theta))+\frac{1}{\sqrt{\ell\pi}}A_{11}.
\end{eqnarray}

Therefore, by combining with $L_{0}(h_{0,11}(\theta))=\tau_{c}\left(A_{1}h_{0,11}(0)+A_{2}h_{0,11}(-1)\right)$, (3.21) and (3.22), we can obtain
\begin{eqnarray*}
\left(-\tau_{c}A_{1}-\tau_{c}A_{2}\right)h_{0,11}(0)=\frac{1}{\sqrt{\ell\pi}}A_{11}-\frac{1}{\sqrt{\ell\pi}}\Phi(0)\Psi(0)A_{11}-\tau_{c}A_{2}\frac{1}{\sqrt{\ell\pi}}\int_{-1}^{0}\Phi(t)\Psi(0)A_{11}dt,
\end{eqnarray*}
and hence
\begin{eqnarray*}
h_{0,11}(\theta)=\frac{1}{\sqrt{\ell\pi}}\int_{0}^{\theta}\Phi(t)\Psi(0)A_{11}dt+C_{6}
\end{eqnarray*}
with
\begin{eqnarray*}
C_{6}=\left(-\tau_{c}A_{1}-\tau_{c}A_{2}\right)^{-1}\left(\frac{1}{\sqrt{\ell\pi}}A_{11}-\frac{1}{\sqrt{\ell\pi}}\Phi(0)\Psi(0)A_{11}-\tau_{c}A_{2}\frac{1}{\sqrt{\ell\pi}}\int_{-1}^{0}\Phi(t)\Psi(0)A_{11}dt\right).
\end{eqnarray*}

\subsection{Calculations of $A_{i,j}$ and $S_{2}(\Phi(\theta)z_{x},w)$}

In this subsection, let $F(\varphi,\mu)=\left(F^{(1)}(\varphi,\mu),F^{(2)}(\varphi,\mu)\right)^{T}$ and $\varphi=\left(\varphi_{1},\varphi_{2}\right)^{T}\in\mathcal{C}$, and we write
\begin{eqnarray}
\frac{1}{j!}F_{j}(\varphi,\mu)=\sum_{j_{1}+j_{2}+j_{3}+j_{4}=j}\frac{1}{j_{1}!j_{2}!j_{3}!j_{4}!}f_{j_{1}j_{2}j_{3}j_{4}} \varphi_{1}^{j_{1}}(0)\varphi_{2}^{j_{2}}(0)\varphi_{1}^{j_{3}}(-1)\mu^{j_{4}},
\end{eqnarray}
where
\begin{eqnarray*}
f_{j_{1}j_{2}j_{3}j_{4}}=\left(f_{j_{1}j_{2}j_{3}j_{4}}^{(1)},f_{j_{1}j_{2}j_{3}j_{4}}^{(2)}\right)^{T}
\end{eqnarray*}
with
\begin{eqnarray*}
f_{j_{1}j_{2}j_{3}j_{4}}^{(k)}=\frac{\partial^{j_{1}+j_{2}+j_{3}+j_{4}}F^{(k)}(0,0,0,0)}{\partial\varphi_{1}^{j_{1}}(0) \partial\varphi_{2}^{j_{2}}(0)\partial\varphi_{1}^{j_{3}}(-1)\partial\mu^{j_{4}}},~k=1,2.
\end{eqnarray*}

Then from (3.23), we have
\begin{eqnarray}\begin{aligned}
F_{2}(\varphi,\mu)&=F_{2}(\varphi,0) \\
&=2\sum_{j_{1}+j_{2}+j_{3}+j_{4}=2}\frac{1}{j_{1}!j_{2}!j_{3}!j_{4}!}f_{j_{1}j_{2}j_{3}j_{4}} \varphi_{1}^{j_{1}}(0)\varphi_{2}^{j_{2}}(0)\varphi_{1}^{j_{3}}(-1)\mu^{j_{4}} \\
&=f_{0020}\varphi_{1}^{2}(-1)+2f_{0110}\varphi_{2}(0)\varphi_{1}(-1)+f_{0200}\varphi_{2}^{2}(0) \\
&+2f_{1010}\varphi_{1}(0)\varphi_{1}(-1)+2f_{1100}\varphi_{1}(0)\varphi_{2}(0)+f_{2000}\varphi_{1}^{2}(0)
\end{aligned}\end{eqnarray}
and
\begin{eqnarray}\begin{aligned}
F_{3}(\varphi,0)&=6\sum_{j_{1}+j_{2}+j_{3}+j_{4}=3}\frac{1}{j_{1}!j_{2}!j_{3}!j_{4}!}f_{j_{1}j_{2}j_{3}j_{4}} \varphi_{1}^{j_{1}}(0)\varphi_{2}^{j_{2}}(0)\varphi_{1}^{j_{3}}(-1)\mu^{j_{4}} \\
&=f_{0030}\varphi_{1}^{3}(-1)+3f_{0120}\varphi_{2}(0)\varphi_{1}^{2}(-1)+3f_{0210}\varphi_{2}^{2}(0)\varphi_{1}(-1) \\
&+f_{0300}\varphi_{2}^{3}(0)+3f_{1020}\varphi_{1}(0)\varphi_{1}^{2}(-1)+6f_{1110}\varphi_{1}(0)\varphi_{2}(0)\varphi_{1}(-1) \\
&+3f_{1200}\varphi_{1}(0)\varphi_{2}^{2}(0)+3f_{2010}\varphi_{1}^{2}(0)\varphi_{1}(-1)+3f_{2100}\varphi_{1}^{2}(0)\varphi_{2}(0) \\
&+f_{3000}\varphi_{1}^{3}(0).
\end{aligned}\end{eqnarray}

Notice that
\begin{eqnarray}\begin{aligned}
&\varphi(\theta)=\Phi(\theta)z_{x}=\phi(\theta)z_{1}(t)\gamma_{n_{c}}(x)+\overline{\phi}(\theta)z_{2}(t)\gamma_{n_{c}}(x) \\
&=\left(\begin{array}{c}
\phi_{1}(\theta)z_{1}(t)\gamma_{n_{c}}(x)+\overline{\phi}_{1}(\theta)z_{2}(t)\gamma_{n_{c}}(x) \\
\phi_{2}(\theta)z_{1}(t)\gamma_{n_{c}}(x)+\overline{\phi}_{2}(\theta)z_{2}(t)\gamma_{n_{c}}(x)
\end{array}\right) \\
&=\left(\begin{array}{c}
\varphi_{1}(\theta) \\
\varphi_{2}(\theta)
\end{array}\right),
\end{aligned}\end{eqnarray}
and similar to (2.41), we have
\begin{eqnarray}
F_{2}\left(\Phi(\theta)z_{x},0\right)=\sum_{q_{1}+q_{2}=2}A_{q_{1}q_{2}}\gamma_{n_{c}}^{q_{1}+q_{2}}(x)z_{1}^{q_{1}}z_{2}^{q_{2}},
\end{eqnarray}
then by combining with (3.24), (3.26) and (3.27), we have
\begin{eqnarray*}\begin{aligned}
A_{20}&=f_{0020}\phi_{1}^{2}(-1)+2f_{0110}\phi_{2}(0)\phi_{1}(-1)+f_{0200}\phi_{2}^{2}(0)+2f_{1010}\phi_{1}(0)\phi_{1}(-1) \\
&+2f_{1100}\phi_{1}(0)\phi_{2}(0)+f_{2000}\phi_{1}^{2}(0), \\
A_{02}&=f_{0020}\overline{\phi}_{1}^{2}(-1)+2f_{0110}\overline{\phi}_{2}(0)\overline{\phi}_{1}(-1)+f_{0200}\overline{\phi}_{2}^{2}(0)+2f_{1010}\overline{\phi}_{1}(0)\overline{\phi}_{1}(-1) \\
&+2f_{1100}\overline{\phi}_{1}(0)\overline{\phi}_{2}(0)+f_{2000}\overline{\phi}_{1}^{2}(0), \\
A_{11}&=2f_{0020}\phi_{1}(-1)\overline{\phi}_{1}(-1)+2f_{0110}\left(\phi_{2}(0)\overline{\phi}_{1}(-1)+\overline{\phi}_{2}(0)\phi_{1}(-1)\right)+2f_{0200}\phi_{2}(0)\overline{\phi}_{2}(0) \\
&+2f_{1010}\left(\phi_{1}(0)\overline{\phi}_{1}(-1)+\overline{\phi}_{1}(0)\phi_{1}(-1)\right)+2f_{1100}\left(\phi_{1}(0)\overline{\phi}_{2}(0)+\overline{\phi}_{1}(0)\phi_{2}(0)\right) \\
&+2f_{2000}\phi_{1}(0)\overline{\phi}_{1}(0).
\end{aligned}\end{eqnarray*}
Furthermore, from (2.41), (3.25) and (3.26), we have
\begin{eqnarray*}\begin{aligned}
A_{30}&=f_{0030}\phi_{1}^{3}(-1)+3f_{0120}\phi_{2}(0)\phi_{1}^{2}(-1)+3f_{0210}\phi_{2}^{2}(0)\phi_{1}(-1)+f_{0300}\phi_{2}^{3}(0) \\
&+3f_{1020}\phi_{1}(0)\phi_{1}^{2}(-1)+6f_{1110}\phi_{1}(0)\phi_{2}(0)\phi_{1}(-1) \\
&+3f_{1200}\phi_{1}(0)\phi_{2}^{2}(0)+3f_{2010}\phi_{1}^{2}(0)\phi_{1}(-1) \\
&+3f_{2100}\phi_{1}^{2}(0)\phi_{2}(0)+f_{3000}\phi_{1}^{3}(0), \\
A_{03}&=f_{0030}\overline{\phi}_{1}^{3}(-1)+3f_{0120}\overline{\phi}_{2}(0)\overline{\phi}_{1}^{2}(-1)+3f_{0210}\overline{\phi}_{2}^{2}(0)\overline{\phi}_{1}(-1)+f_{0300}\overline{\phi}_{2}^{3}(0) \\
&+3f_{1020}\overline{\phi}_{1}(0)\overline{\phi}_{1}^{2}(-1)+6f_{1110}\overline{\phi}_{1}(0)\overline{\phi}_{2}(0)\overline{\phi}_{1}(-1) \\
&+3f_{1200}\overline{\phi}_{1}(0)\overline{\phi}_{2}^{2}(0)+3f_{2010}\overline{\phi}_{1}^{2}(0)\overline{\phi}_{1}(-1) \\
&+3f_{2100}\overline{\phi}_{1}^{2}(0)\overline{\phi}_{2}(0)+f_{3000}\overline{\phi}_{1}^{3}(0)
\end{aligned}\end{eqnarray*}
and
\begin{eqnarray*}\begin{aligned}
A_{21}&=3f_{0030}\phi_{1}^{2}(-1)\overline{\phi}_{1}(-1)+3f_{0120}\left(\phi_{2}(0)2\phi_{1}(-1)\overline{\phi}_{1}(-1)+\overline{\phi}_{2}(0)\phi_{1}^{2}(-1)\right) \\
&+3f_{0210}\left(2\phi_{2}(0)\overline{\phi}_{2}(0)\overline{\phi}_{1}(-1)+\phi_{2}^{2}(0)\overline{\phi}_{1}(-1)\right)+3f_{0300}\phi_{2}^{2}(0)\overline{\phi}_{2}(0), \\
&+3f_{1020}\left(2\phi_{1}(0)\phi_{1}(-1)\overline{\phi}_{1}(-1)+\overline{\phi}_{1}(0)\phi_{1}^{2}(-1)\right) \\
&+6f_{1110}\left(\phi_{1}(0)\phi_{2}(0)\overline{\phi}_{1}(-1)+\phi_{1}(0)\overline{\phi}_{2}(0)\phi_{1}(-1)+\overline{\phi}_{1}(0)\phi_{2}(0)\phi_{1}(-1)\right) \\
&+3f_{1200}\left(2\phi_{1}(0)\phi_{2}(0)\overline{\phi}_{2}(0)+\overline{\phi}_{1}(0)\phi_{2}^{2}(0)\right)+3f_{2100}\left(\phi_{1}^{2}(0)\overline{\phi}_{2}(0)+2\phi_{1}(0)\overline{\phi}_{1}(0)\phi_{2}(0)\right) \\
&+3f_{2010}\left(2\phi_{1}(0)\overline{\phi}_{1}(0)\phi_{1}(-1)+\phi_{1}^{2}(0)\overline{\phi}_{1}(-1)\right) \\
&+3f_{3000}\phi_{1}^{2}(0)\overline{\phi}_{1}(0), \\
A_{12}&=3f_{0030}\phi_{1}(-1)\overline{\phi}_{1}^{2}(-1)+3f_{0120}\left(\phi_{2}(0)\overline{\phi}_{1}^{2}(-1)+2\overline{\phi}_{2}(0)\phi_{1}(-1)\overline{\phi}_{1}(-1)\right) \\
&+3f_{0210}\left(2\phi_{2}(0)\overline{\phi}_{2}(0)\overline{\phi}_{1}(-1)+\overline{\phi}_{2}^{2}(0)\phi_{1}(-1)\right)+3f_{0300}\phi_{2}(0)\overline{\phi}^{2}_{2}(0), \\
&+3f_{1020}\left(\phi_{1}(0)\overline{\phi}_{1}^{2}(-1)+2\overline{\phi}_{1}(0)\phi_{1}(-1)\overline{\phi}_{1}(-1)\right) \\
&+6f_{1110}\left(\phi_{1}(0)\overline{\phi}_{2}(0)\overline{\phi}_{1}(-1)+\overline{\phi}_{1}(0)\phi_{2}(0)\overline{\phi}_{1}(-1)+\overline{\phi}_{1}(0)\overline{\phi}_{2}(0)\phi_{1}(-1)\right) \\
&+3f_{1200}\left(\phi_{1}(0)\overline{\phi}_{2}^{2}(0)+2\overline{\phi}_{1}(0)\phi_{2}(0)\overline{\phi}_{2}(0)\right)+3f_{2100}\left(2\phi_{1}(0)\overline{\phi}_{1}(0)\overline{\phi}_{2}(0)+\overline{\phi}_{1}^{2}(0)\phi_{2}(0)\right) \\
&+3f_{2010}\left(\overline{\phi}_{1}^{2}(0)\phi_{1}(-1)+2\phi_{1}(0)\overline{\phi}_{1}(0)\overline{\phi}_{1}(-1)\right) \\
&+3f_{3000}\phi_{1}(0)\overline{\phi}^{2}_{1}(0).
\end{aligned}\end{eqnarray*}

Moreover, from (3.23), we have
\begin{eqnarray}\begin{aligned}
&F_{2}(\varphi(\theta)+w,\mu)=2 \sum_{j_{1}+j_{2}+j_{3}+j_{4}=2}\frac{1}{j_{1}!j_{2}!j_{3}!j_{4}!}f_{j_{1}j_{2}j_{3}j_{4}} (\varphi_{1}(0)+w_{1}(0))^{j_{1}}(\varphi_{2}(0)+w_{2}(0))^{j_{2}}(\varphi_{1}(-1)+w_{1}(-1))^{j_{3}}\mu^{j_{4}} \\
&=f_{0020}(\varphi_{1}(-1)+w_{1}(-1))^{2}+2f_{0110}(\varphi_{2}(0)+w_{2}(0))(\varphi_{1}(-1)+w_{1}(-1))+f_{0200}(\varphi_{2}(0)+w_{2}(0))^{2} \\
&+2f_{1010}(\varphi_{1}(0)+w_{1}(0))(\varphi_{1}(-1)+w_{1}(-1)) \\
&+2f_{1100}(\varphi_{1}(0)+w_{1}(0))(\varphi_{2}(0)+w_{2}(0))+f_{2000}(\varphi_{1}(0)+w_{1}(0))^{2}.
\end{aligned}\end{eqnarray}

Notice that
\begin{eqnarray}\begin{aligned}
\varphi(\theta)+w(\theta)&=\Phi(\theta)z_{x}+w(\theta)=\phi(\theta)z_{1}(t)\gamma_{n_{c}}(x)+\overline{\phi}(\theta)z_{2}(t)\gamma_{n_{c}}(x)+w(\theta) \\
&=\left(\begin{array}{c}
\phi_{1}(\theta)z_{1}(t)\gamma_{n_{c}}(x)+\overline{\phi}_{1}(\theta)z_{2}(t)\gamma_{n_{c}}(x)+w_{1}(\theta) \\
\phi_{2}(\theta)z_{1}(t)\gamma_{n_{c}}(x)+\overline{\phi}_{2}(\theta)z_{2}(t)\gamma_{n_{c}}(x)+w_{2}(\theta)
\end{array}\right) \\
&=\left(\begin{array}{c}
\varphi_{1}(\theta)+w_{1}(\theta) \\
\varphi_{2}(\theta)+w_{2}(\theta)
\end{array}\right)
\end{aligned}\end{eqnarray}
and
\begin{eqnarray}\begin{aligned}
F_{2}\left(\Phi(\theta)z_{x}+w,\mu\right)&=F_{2}\left(\Phi(\theta)z_{x}+w,0\right) \\
&=\sum_{q_{1}+q_{2}=2} A_{q_{1}q_{2}}\gamma_{n_{c}}^{q_{1}+q_{2}}(x)z_{1}^{q_{1}}z_{2}^{q_{2}}+S_{2}\left(\Phi(\theta)z_{x}, w\right)+O\left(|w|^{2}\right),
\end{aligned}\end{eqnarray}
then by combining with (3.28), (3.29) and (3.30), we have
\begin{eqnarray*}\begin{aligned}
&S_{2}\left(\Phi(\theta)z_{x}, w\right) \\
&=2f_{0020}\left(\phi_{1}(-1)z_{1}(t)\gamma_{n_{c}}(x)+\overline{\phi}_{1}(-1)z_{2}(t)\gamma_{n_{c}}(x)\right)w_{1}(-1) \\
&+2f_{0110}\left(\left(\phi_{2}(0)z_{1}(t)\gamma_{n_{c}}(x)+\overline{\phi}_{2}(0)z_{2}(t)\gamma_{n_{c}}(x)\right)w_{1}(-1)+\left(\phi_{1}(-1)z_{1}(t) \gamma_{n_{c}}(x)+\overline{\phi}_{1}(-1)z_{2}(t)\gamma_{n_{c}}(x)\right)w_{2}(0)\right) \\
&+2f_{0200}\left(\phi_{2}(0)z_{1}(t)\gamma_{n_{c}}(x)+\overline{\phi}_{2}(0)z_{2}(t)\gamma_{n_{c}}(x)\right)w_{2}(0) \\
&+2f_{1010}\left(\left(\phi_{1}(0)z_{1}(t)\gamma_{n_{c}}(x)+\overline{\phi}_{1}(0)z_{2}(t)\gamma_{n_{c}}(x)\right)w_{1}(-1)+\left(\phi_{1}(-1)z_{1}(t) \gamma_{n_{c}}(x)+\overline{\phi}_{1}(-1)z_{2}(t)\gamma_{n_{c}}(x)\right)w_{1}(0)\right) \\
&+2f_{1100}\left(\left(\phi_{1}(0)z_{1}(t)\gamma_{n_{c}}(x)+\overline{\phi}_{1}(0)z_{2}(t)\gamma_{n_{c}}(x)\right)w_{2}(0)+\left(\phi_{2}(0)z_{1}(t) \gamma_{n_{c}}(x)+\overline{\phi}_{2}(0)z_{2}(t)\gamma_{n_{c}}(x)\right)w_{1}(0)\right) \\
&+2f_{2000}\left(\phi_{1}(0)z_{1}(t)\gamma_{n_{c}}(x)+\overline{\phi}_{1}(0)z_{2}(t)\gamma_{n_{c}}(x)\right)w_{1}(0).
\end{aligned}\end{eqnarray*}

\section{Application to a predator-prey model with memory and gestation time delays}
\label{sec:4}

In this section, we consider the following diffusive predator-prey model with ratio-dependent Holling type-\uppercase\expandafter{\romannumeral3} functional response, which includes with memory and gestation delays
\begin{eqnarray}
\left\{\begin{aligned}
&\frac{\partial u(x,t)}{\partial t}=d_{11}\Delta u(x,t)+u(x,t)\left(1-u(x,t)\right)-\frac{\beta u^{2}(x,t)v(x,t)}{u^{2}(x,t)+mv^{2}(x,t)}, & x\in(0,\ell\pi),~t>0, \\
&\frac{\partial v(x,t)}{\partial t}=d_{22}\Delta v(x,t)-d_{21}\left(v(x,t)u_{x}(x,t-\tau)\right)_{x}+\gamma v(x,t)\left(1-\frac{v(x,t)}{u(x,t-\tau)}\right), & x\in(0,\ell\pi),~t>0, \\
&u_{x}(0,t)=u_{x}(\ell\pi,t)=v_{x}(0,t)=v_{x}(\ell\pi,t)=0, & t \geq 0, \\
&u(x,t)=u_{0}(x,t),~v(x,t)=v_{0}(x,t), & x\in(0,\ell\pi),~-\tau \leq t \leq 0,
\end{aligned}\right.
\end{eqnarray}
where $u(x,t)$ and $v(x,t)$ stand for the densities of the prey and predators at location $x$ and time $t$, respectively, $\beta>0$, $m>0$ and $\gamma>0$.

\subsection{The case of with memory delay and without gestation delay}

When system (4.1) includes memory delay and doesn't include gestation delay, that is to say, in the model (1.3), we let
\begin{eqnarray*}\begin{aligned}
f\left(u(x,t),v(x,t)\right)&=u(x,t)\left(1-u(x,t)\right)-\frac{\beta u^{2}(x,t)v(x,t)}{u^{2}(x,t)+mv^{2}(x,t)}, \\
g\left(u(x,t),v(x,t)\right)&=\gamma v(x,t)\left(1-\frac{v(x,t)}{u(x,t)}\right).
\end{aligned}\end{eqnarray*}

Then the model (1.3) can be written as
\begin{eqnarray}\left\{\begin{aligned}
&\frac{\partial u(x,t)}{\partial t}=d_{11}\Delta u(x,t)+u(x,t)\left(1-u(x,t)\right)-\frac{\beta u^{2}(x,t)v(x,t)}{u^{2}(x,t)+mv^{2}(x,t)}, & x\in(0,\ell\pi),~t>0, \\
&\frac{\partial v(x,t)}{\partial t}=d_{22}\Delta v(x,t)-d_{21}\left(v(x,t)u_{x}(x,t-\tau)\right)_{x}+\gamma v(x,t)\left(1-\frac{v(x,t)}{u(x,t)}\right), & x\in(0,\ell\pi),~t>0, \\
&u_{x}(0,t)=u_{x}(\ell\pi,t)=v_{x}(0,t)=v_{x}(\ell\pi,t)=0, & t \geq 0, \\
&u(x,t)=u_{0}(x,t), & x \in (0,\ell\pi),~-\tau \leq t \leq 0, \\
&v(x,t)=v_{0}(x), & x \in (0,\ell\pi).
\end{aligned}\right.\end{eqnarray}

Notice that for the system (4.2), when $d_{21}=0$, the global asymptotic stability of the positive constant steady state in this system has been investigated by Shi et al. in \cite{lv41}. Furthermore, the normal form for Hopf bifurcation can be calculated by using the developed algorithm in \cite{lv30}, and the detail calculation procedures are give in Appendix A. In the following, we first give the stability and Hopf bifurcation analysis for the model (4.2), then by employing the developed procedure in \cite{lv30} for calculating the normal form for Hopf bifurcation, the direction and stability of the Hopf bifurcation are determined.

\subsubsection{Stability and Hopf bifurcation analysis}

The system (4.2) has the positive constant steady state $E_{*}\left(u_{*},v_{*}\right)$, where
\begin{eqnarray}
u_{*}=v_{*}=1-\frac{\beta}{m+1}
\end{eqnarray}
with $0<\beta<m+1$. For $E_{*}\left(u_{*},v_{*}\right)$, form (2.4), when $m>1$, we have
\begin{eqnarray*}\begin{aligned}
&a_{11}=\frac{2\beta}{(m+1)^{2}}-1\left\{
\begin{array}{cc}
\leq 0, & 0<\beta \leq \frac{(m+1)^{2}}{2}, \\
>0, & \beta>\frac{(m+1)^{2}}{2}.
\end{array}\right.
\end{aligned}\end{eqnarray*}

Notice that when $m>1$, if $a_{11}>0$, then we have $\beta>\frac{(m+1)^{2}}{2}>m+1$, which is contradict to the condition $0<\beta<m+1$. Thus, when $m>1$, $a_{11}\leq 0$ under the condition $0<\beta<m+1$. When $0<m<1$, we have
\begin{eqnarray}\begin{aligned}
&a_{11}=\frac{2\beta}{(m+1)^{2}}-1 \left\{
\begin{array}{cc}
\leq 0, & 0 < \beta \leq \frac{(m+1)^{2}}{2}, \\
>0, & \frac{(m+1)^{2}}{2}<\beta<m+1.
\end{array}\right.
\end{aligned}\end{eqnarray}

Furthermore, we have
\begin{eqnarray}\begin{aligned}
&a_{12}=\frac{\beta(m-1)}{(m+1)^{2}}\left\{
\begin{array}{cc}
\leq 0, & 0 < m \leq 1, \\
>0, & m>1,
\end{array}\right. \\
&a_{21}=\gamma>0,~a_{22}=-\gamma<0, \\
&b_{11}=0,~b_{12}=0,~b_{21}=0,~b_{22}=0.
\end{aligned}\end{eqnarray}

Moreover, by combining with (4.4), (4.5),
\begin{eqnarray*}
D_{1}=\left(\begin{array}{cc}
d_{11} & 0 \\
0 & d_{22}
\end{array}\right),~D_{2}=\left(\begin{array}{cc}
0 & 0 \\
-d_{21} v_{*} & 0
\end{array}\right),~A_{1}=\left(\begin{array}{cc}
a_{11} & a_{12} \\
a_{21} & a_{22}
\end{array}\right),~A_{2}=\left(\begin{array}{cc}
b_{11} & b_{12} \\
b_{21} & b_{22}
\end{array}\right),
\end{eqnarray*}
and
\begin{eqnarray*}
M_{n}(\lambda)=\lambda I_{2}+\frac{n^{2}}{\ell^{2}}D_{1}+\frac{n^{2}}{\ell^{2}}e^{-\lambda\tau}D_{2}-A_{1}-A_{2}e^{-\lambda\tau},
\end{eqnarray*}
or according to (2.7), the characteristic equation of system (4.2) can be written as
\begin{eqnarray}
\Gamma_{n}(\lambda)=\det\left(M_{n}(\lambda)\right)=\lambda^{2}-T_{n}\lambda+\widetilde{J}_{n}(\tau)=0,
\end{eqnarray}
where
\begin{eqnarray}\begin{aligned}
T_{n}&=(a_{11}+a_{22})-(d_{11}+d_{22})\frac{n^{2}}{\ell^{2}}, \\
\widetilde{J}_{n}(\tau)&=d_{11}d_{22}\frac{n^{4}}{\ell^{4}}-\left(d_{11}a_{22}+d_{22}a_{11}+d_{21}a_{12}v_{*}e^{-\lambda \tau}\right)\frac{n^{2}}{\ell^{2}}+\operatorname{Det}(A_{1})
\end{aligned}\end{eqnarray}
with $\operatorname{Det}(A_{1})=a_{11}a_{22}-a_{12}a_{21}$.

When $d_{21}=0$, from the second mathematical expression in (4.7), we denote
\begin{eqnarray}
J_{n}:=d_{11} d_{22}\frac{n^{4}}{\ell^{4}}-\left(d_{11}a_{22}+d_{22}a_{11}\right)\frac{n^{2}}{\ell^{2}}+\operatorname{Det}(A_{1}),
\end{eqnarray}
then from (4.4), (4.5), (4.7) and (4.8), it is easy to verify that $T_{n}<0$ and $J_{n}>0$ provided that
\begin{eqnarray*}
(C_{0}):~0<\beta\leq\frac{(m+1)^{2}}{2},~0<m\leq1.
\end{eqnarray*}
This implies that when $d_{21}=0$ and the condition $(C_{0})$ holds, the positive constant steady state $E_{*}(u_{*},v_{*})$ is asymptotically stable for $d_{11} \geq 0$ and $d_{22} \geq 0$. In this subsection, we always assume that the condition $(C_{0})$ holds.

Since $J_{n}>0$ and $a_{12}<0$ under the condition $(C_{0})$, then according to (4.6), we have
\begin{eqnarray*}
\Gamma_{n}(0)=J_{n}-d_{21}a_{12}v_{*}\frac{n^{2}}{\ell^{2}}>0.
\end{eqnarray*}
This implies that $\lambda=0$ is not a root of (4.6). Let $\lambda=i \omega_{n}(\omega_{n}>0)$ be a root of (4.6). From (4.4), (4.5) and by substituting $\lambda=i \omega_{n}(\omega_{n}>0)$ into (4.6), and separating the real from the imaginary parts, we have
\begin{eqnarray}
\left\{\begin{aligned}
&J_{n}-\omega_{n}^{2}=\frac{n^{2}}{\ell^{2}}d_{21}a_{12}v_{*}\cos(\omega_{n}\tau), \\
&T_{n}\omega_{n}=\frac{n^{2}}{\ell^{2}}d_{21}a_{12}v_{*}\sin(\omega_{n}\tau),
\end{aligned}\right.
\end{eqnarray}
which yields
\begin{eqnarray}
\omega^{4}+P_{n}\omega^{2}+Q_{n}=0,
\end{eqnarray}
where
\begin{eqnarray}
P_{n}=T_{n}^{2}-2J_{n}=\left(d_{11}^{2}+d_{22}^{2}\right)\frac{n^{4}}{\ell^{4}}-2\left(d_{11}a_{11}+d_{22}a_{22}\right)\frac{n^{2}}{\ell^{2}}+a_{11}^{2}+a_{22}^{2}+2a_{12}a_{21},
\end{eqnarray}
and
\begin{eqnarray}
Q_{n}=\left(J_{n}+d_{21}a_{12}v_{*}\frac{n^{2}}{\ell^{2}}\right)\left(J_{n}-d_{21}a_{12}v_{*}\frac{n^{2}}{\ell^{2}}\right).
\end{eqnarray}
Here,
\begin{eqnarray}
a_{11}^{2}+a_{22}^{2}+2a_{12}a_{21}\left\{\begin{array}{cc}
\leq 0, & c_{*}\leq 0, \\
>0, & c_{*}>0
\end{array}\right.\end{eqnarray}
with
\begin{eqnarray}
c_{*}=\frac{4\beta^{2}-4\beta(m+1)^{2}+(m+1)^{4}+\gamma^{2}(m+1)^{4}+2\beta\gamma(m-1)(m+1)^{2}}{(m+1)^{4}}.
\end{eqnarray}

Notice that from (4.10), we can define
\begin{eqnarray*}
\omega_{n}^{\pm}:=\sqrt{\frac{-P_{n} \pm \sqrt{P_{n}^{2}-4Q_{n}}}{2}}.
\end{eqnarray*}
Moreover, by combining with (4.4), (4.5), (4.11) and (4.13), if we assume that $c_{*}>0$, then $P_{n}>0$ for any $n \in \mathbb{N}_{0}$. Furthermore, by defining
\begin{eqnarray}
d_{21}^{(n)}=-\frac{J_{n}}{a_{12}v_{*}(n/\ell)^{2}}=-\frac{1}{a_{12}v_{*}}\left(d_{11}d_{22}(n/\ell)^{2}+\frac{\operatorname{Det}(A_{1})}{(n/\ell)^{2}}-\left(d_{11}a_{22}+d_{22}a_{11}\right)\right)>0,
\end{eqnarray}
then for fixed $n$, by (4.12) we have
\begin{eqnarray}
Q_{n}\left\{\begin{aligned}
&>0, & 0<d_{21}<d_{21}^{(n)}, \\
&=0, & d_{21}=d_{21}^{(n)}, \\
&<0, & d_{21}>d_{21}^{(n)}.
\end{aligned}\right.
\end{eqnarray}

Thus, when $d_{21}>d_{21}^{(n)}$, (4.10) has one positive root $\omega_{n}$, where
\begin{eqnarray}
\omega_{n}=\sqrt{\frac{-P_{n}+\sqrt{P_{n}^{2}-4Q_{n}}}{2}}.
\end{eqnarray}

Notice that $T_{n}<0$ for any $n \in \mathbb{N}_{0}$ and $a_{12}<0$ under the condition $(C_{0})$, then from the second mathematical expression in (4.9), we have
\begin{eqnarray*}
\sin \left(\omega_{n}\tau\right)=\frac{T_{n}\omega_{n}}{(n/\ell)^{2}a_{12}d_{21}v_{*}}>0.
\end{eqnarray*}

Thus, from the first mathematical expression in (4.9), we can set
\begin{eqnarray}
\tau_{n,j}=\frac{1}{\omega_{n}}\left\{\arccos \left\{\frac{J_{n}-\omega_{n}^{2}}{d_{21}a_{12}v_{*}(n/\ell)^{2}}\right\}+2j\pi\right\},~n \in \mathbb{N},~j \in \mathbb{N}_{0}.
\end{eqnarray}
Furthermore, it is easy to verify that the transversality condition satisfies
\begin{eqnarray*}
\left.\frac{d\operatorname{Re}(\lambda(\tau))}{d\tau}\right|_{\tau=\tau_{n,j}}>0.
\end{eqnarray*}

Furthermore, if we let
\begin{eqnarray}
d_{21}^{*}=\min_{n \in \mathbb{N}}\left\{d_{21}^{(n)}\right\}>0,
\end{eqnarray}
then from (4.15), it is easy to verify that $d_{21}^{(n)}$ is decreasing for $n<\ell \sqrt[4]{\frac{\operatorname{Det}(A_{1})}{d_{11} d_{22}}}$, is increasing for $n>\ell \sqrt[4]{\frac{\operatorname{Det}(A_{1})}{d_{11}d_{22}}}$ and $d_{21}^{(n)}\rightarrow \infty$ as $n \rightarrow \infty$. This implies that $d_{21}^{*}$ exists. For fixed $d_{21}>d_{21}^{*}$, define an index set
\begin{eqnarray*}
U\left(d_{21}\right)=\left\{n \in \mathbb{N}:d_{21}^{(n)}<d_{21}\right\}.
\end{eqnarray*}

Moreover, according to the above analysis, we have the following results.
\begin{theorem}
If the condition $(C_{0})$ holds and $c_{*}>0$, then we have the following conclusions:

(a) when $0<d_{21} \leq d_{21}^{*}$, the positive constant steady state $E_{*}\left(u_{*},v_{*}\right)$ of system (4.2) is locally
asymptotically stable for any $\tau \geq 0$;

(b) when $d_{21}>d_{21}^{*}$, if denote
\begin{eqnarray*}
\tau_{*}(d_{21})=\min_{n \in U(d_{21})}\left\{\tau_{n,0}\right\},
\end{eqnarray*}
then the positive constant steady state $E_{*}\left(u_{*},v_{*}\right)$ of system (4.2) is asymptotically stable for $0 \leq \tau<\tau_{*}(d_{21})$ and unstable for $\tau>\tau_{*}(d_{21})$. Furthermore, system (4.2) undergoes Hopf bifurcations at $\tau=\tau_{n,0}$ for $n\in U(d_{21})$.
\end{theorem}

\subsubsection{Direction and stability of the Hopf bifurcation}

We now investigate the direction and stability of the Hopf bifurcation by some numerical simulations. In this section, we use the following initial conditions for the system (4.2)
\begin{eqnarray*}
u(x,t)=u_{0}(x),~v(x,t)=v_{0}(x),~t \in\left[-\tau,0\right],
\end{eqnarray*}
and we set the parameters as follows
\begin{eqnarray*}
d_{11}=0.6,~d_{22}=0.8,~m=0.5,~\gamma=0.5,~\beta=1,~\ell=2.
\end{eqnarray*}

Then according to (4.3), (4.4) and (4.5), we have $E_{*}\left(u_{*},v_{*}\right)=(0.3333,0.3333)$,
\begin{eqnarray*}
a_{11}=-0.1111,~a_{12}=-0.2222,~a_{21}=0.5,~a_{22}=-0.5.
\end{eqnarray*}

It follows from (4.11) and (4.15) that
\begin{eqnarray*}
P_{n}=0.0625n^{4}+0.2333n^{2}+0.0401>0
\end{eqnarray*}
and
\begin{eqnarray}
d_{21}^{(n)}=1.62n^{2}+\frac{9}{n^{2}}+5.25.
\end{eqnarray}

Notice that $P_{n}>0$ for any $n \in \mathbb{N}_{0}$, which together with (4.16), implies that for a fixed $n$, (4.10) has no positive root for $d_{21}<d_{21}^{(n)}$ and has only one positive root for $d_{21}\geq d_{21}^{(n)}$. From (4.20), it is easy to verify that $d_{21}^{(n)}<d_{21}^{(n+1)}$ for any $n \in \mathbb{N}_{0}$, and
\begin{eqnarray}
d_{21}^{(1)}=15.87,~d_{21}^{(2)}=13.98<d_{21}^{(3)}=20.83.
\end{eqnarray}

Therefore, by combining with (4.19) and (4.21), we have $d_{21}^{*}=d_{21}^{(2)}=13.98$. It follows from (4.14) that $c_{*}=0.0401>0$. By Theorem 4.1, we have the following Propositions 4.2 and 4.3.

\begin{proposition}
For system (4.2) with the parameters $d_{11}=0.6,~d_{22}=0.8,~m=0.5,~\gamma=0.5,~\beta=1,~\ell=2$, when $0 \leq d_{21}<d_{21}^{(2)}=13.98$, the positive constant steady state $E_{*}\left(u_{*},v_{*}\right)=(0.3333,0.3333)$ is locally asymptotically stable for any $\tau \geq 0$.
\end{proposition}

Figure 1 illustrates the stability region, and the Hopf bifurcation curves are plotted in the $d_{21}-\tau$ plane for $20 \leq d_{21} \leq 150$. The Hopf bifurcation curves $\tau=\tau_{2,0}$ and $\tau=\tau_{3,0}$ intersect at the point $P_{1}(42.87,0.817)$, which is the Hopf-Hopf bifurcation point. Furthermore, when $d_{21}=2$ and $\tau=4.2$, according to (4.21), we can see that the point $P_{2}(2,4.2)$ satisfies $0 \leq d_{21}<d_{21}^{(2)}=13.98$. According to Proposition 1, we know that under the above parameter settings, as long as $0 \leq d_{21}<d_{21}^{(2)}=13.98$, the positive constant steady state $E_{*}(u_{*},v_{*})$ of system (4.2) is locally asymptotically stable for any $\tau \geq 0$. Especially, by taking the point $P_{2}(2,4.2)$ which satisfies $0 \leq d_{21}<d_{21}^{(2)}=13.98$, we illustrate this result in Fig.2 with the initial values $u_{0}(x)=0.3333-0.1\cos(x),~v_{0}(x)=0.3333+0.1\cos(x)$.

\begin{figure}[!htbp]
\centering
\includegraphics[width=2.5in]{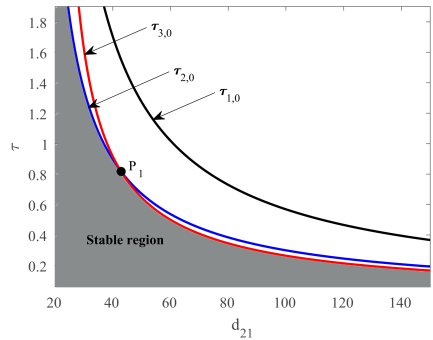}
\caption{Stable region and Hopf bifurcation curves $\tau=\tau_{n,0},~n=1,2,3$ in $d_{21}-\tau$ plane for the parameters $d_{11}=0.6,~d_{22}=0.8,~m=0.5,~\gamma=0.5,~\beta=1,~\ell=2$. Hopf bifurcation curves $\tau=\tau_{2,0}$ and $\tau=\tau_{3,0}$ intersect at the point $P_{1}(42.87,0.817)$.}
\label{fig:1}
\end{figure}

\begin{figure}[!htbp]
\centering
\includegraphics[width=2in]{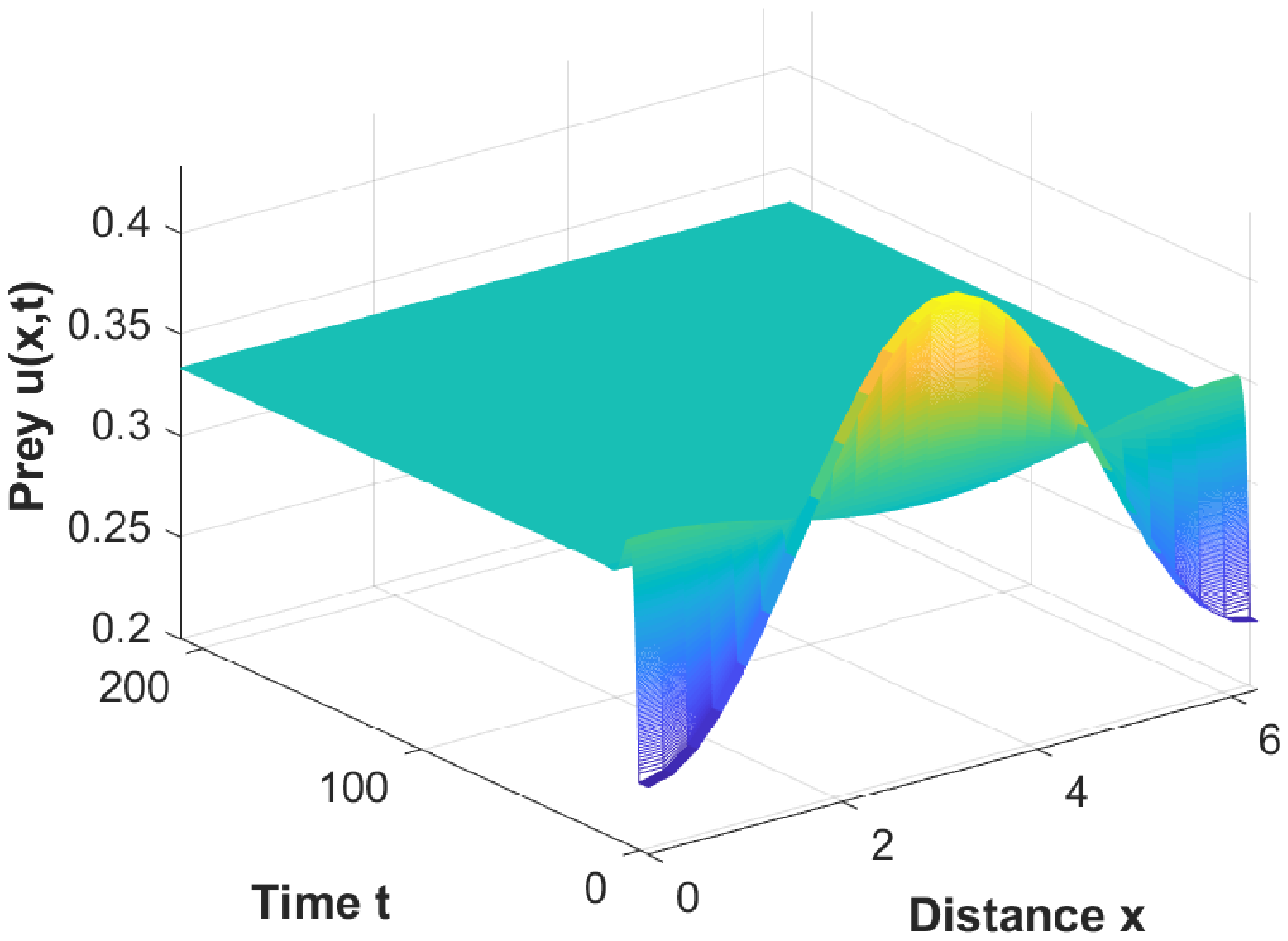}
\includegraphics[width=2in]{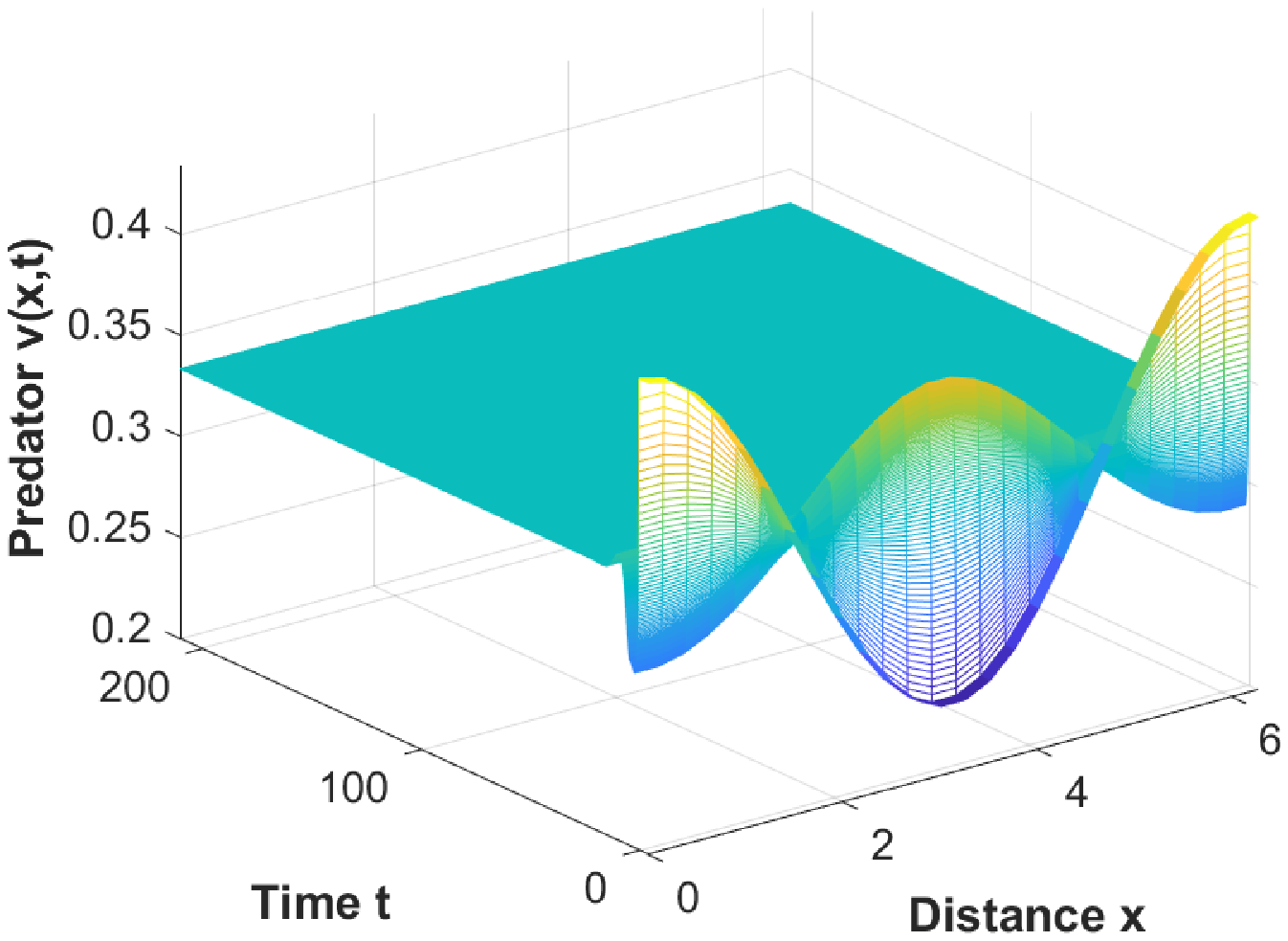} \\
\textbf{(a)} \hspace{4.5cm} \textbf{(b)} \\
\includegraphics[width=2in]{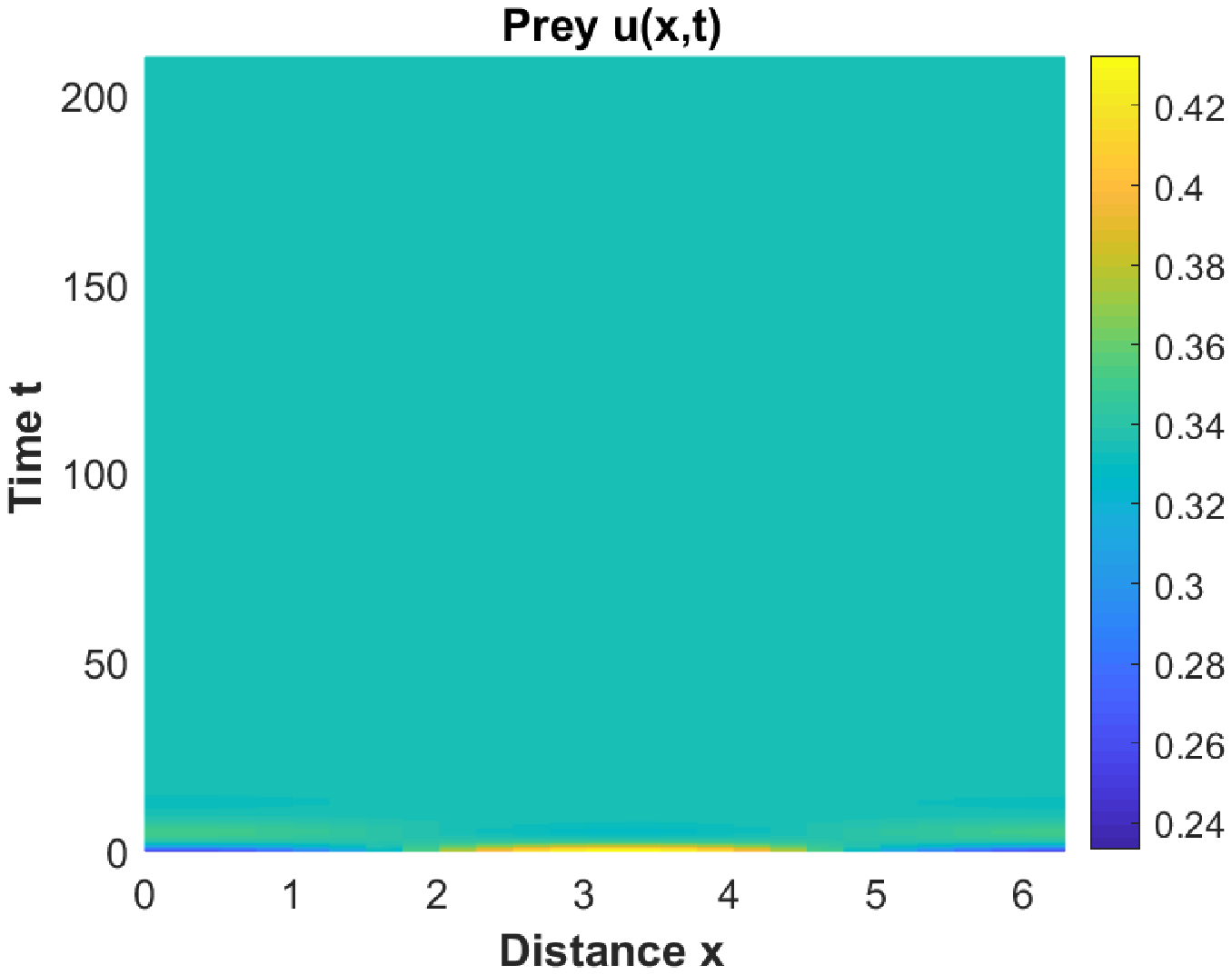}
\includegraphics[width=2in]{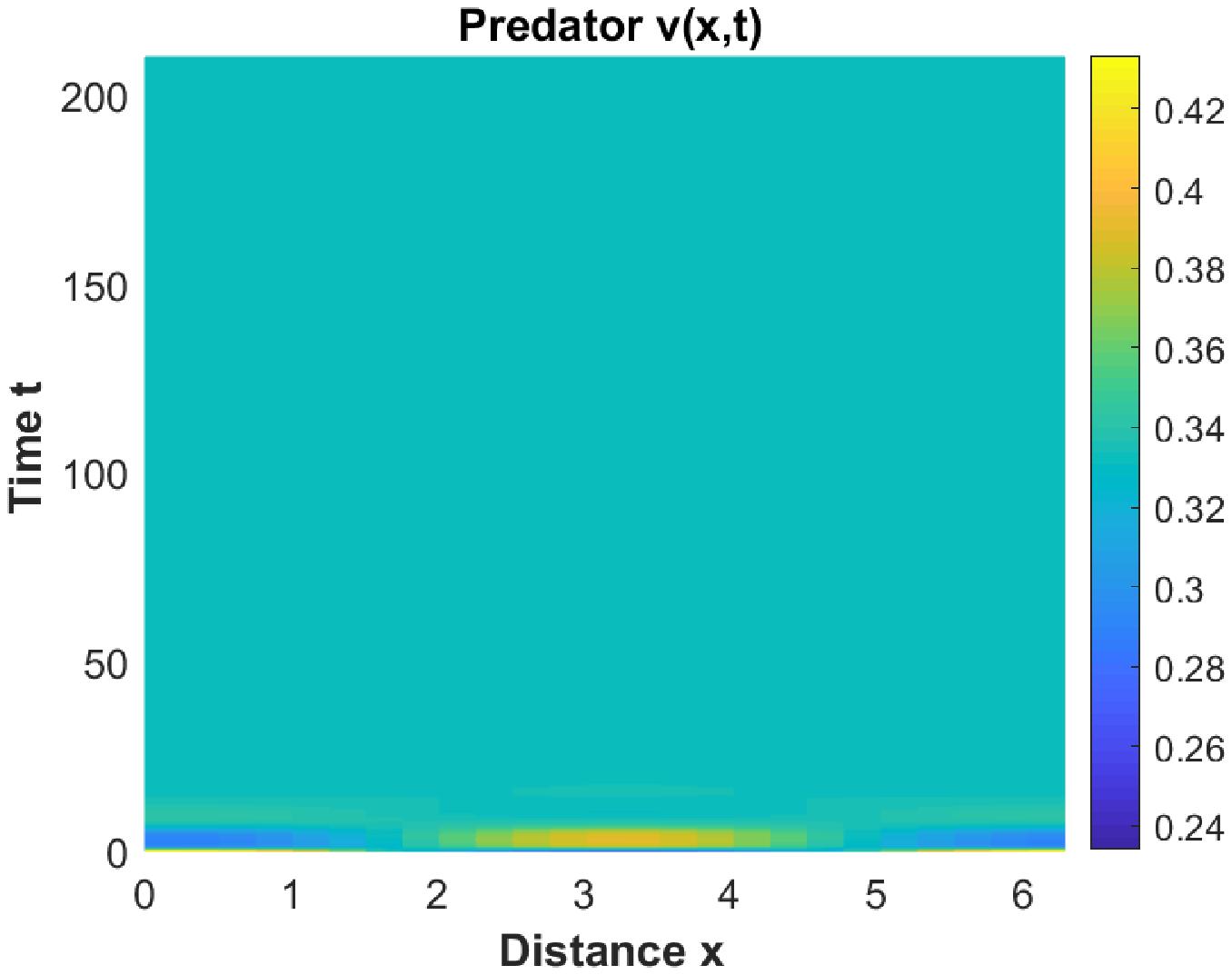} \\
\textbf{(c)} \hspace{4.5cm} \textbf{(d)} \\
\caption{For the parameters $d_{11}=0.6,~d_{22}=0.8,~m=0.5,~\gamma=0.5,~\beta=1,~\ell=2$, and $(d_{21},\tau)$ is chosen as the point $P_{2}(2,4.2)$ which satisfies $0<d_{21}<d_{21}^{(2)}=13.98$. The positive constant steady state $E_{*}\left(u_{*},v_{*}\right)=(0.3333,0.3333)$ is locally asymptotically stable for any $\tau \geq 0$. (a) and (b) are the evolution processes of the solutions $u(x,t)$ and $v(x,t)$ of system (4.2), respectively. (c) and (d) are spatio-temporal diagrams of the solutions $u(x,t)$ and $v(x,t)$ of system (4.2), respectively. The initial values are $u_{0}(x)=0.3333-0.1\cos(x),~v_{0}(x)=0.3333+0.1\cos(x)$.}
\label{fig:2}
\end{figure}

\begin{proposition}
For system (4.2) with the parameters $d_{11}=0.6,~d_{22}=0.8,~m=0.5,~\gamma=0.5,~\beta=1,~\ell=2$, and for fixed $d_{21}>d_{21}^{(2)}=13.98$, the positive constant steady state $E_{*}\left(u_{*},v_{*}\right)$ is asymptotically stable for $\tau<\tau_{*}\left(d_{21}\right)$ and unstable for $\tau>\tau_{*}\left(d_{21}\right)$.
\end{proposition}

From Fig.1, it is obvious to see that
\begin{eqnarray*}
\tau_{*}(d_{21})=\left\{\begin{aligned}
&\tau_{2,0}, & d_{21}^{(2)}<d_{21}<42.87, \\
&\tau_{3,0}, & 42.87<d_{21}<150,
\end{aligned}\right.
\end{eqnarray*}
and when $d_{21}=21$, it follows from (4.17) and (4.18) that
\begin{eqnarray}
\tau_{2,0}=2.5896<\tau_{3,0}=17.9261.
\end{eqnarray}

For $d_{21}=21$ which satisfies $d_{21}^{(2)}=13.98<d_{21}<42.87$, according to (4.22), we know that system (4.2) undergoes Hopf bifurcation at $\tau_{2,0}=2.5896$. Furthermore, the direction and stability of the Hopf bifurcation can be determined by calculating $K_{1}K_{2}$ and $K_{2}$ using the procedures listed in Appendix A. By a direct calculation, we obtain
\begin{eqnarray*}
K_{1}=0.1092>0,~K_{2}=103.5071>0,~K_{1}K_{2}=11.2997>0,
\end{eqnarray*}
which implies that the spatially inhomogeneous Hopf bifurcation at $\tau_{2,0}$ is subcritical and unstable. When $d_{21}=21$ and $\tau=1.5$, by combining with (4.21) and (4.22), we can see that the point $P_{3}(21,1.5)$ satisfies $d_{21}^{(2)}=13.98<d_{21}<42.87$ and $\tau=1.5<\tau_{2,0}=2.5896$. There exists an unstable spatially inhomogeneous periodic solution, and its amplitude is decreasing, see Fig.3 (a)-(d) for detail. The initial values are $u_{0}(x)=0.3333+0.02\cos(x),~v_{0}(x)=0.3333+0.02\cos(x)$.

\begin{figure}[!htbp]
\centering
\includegraphics[width=2in]{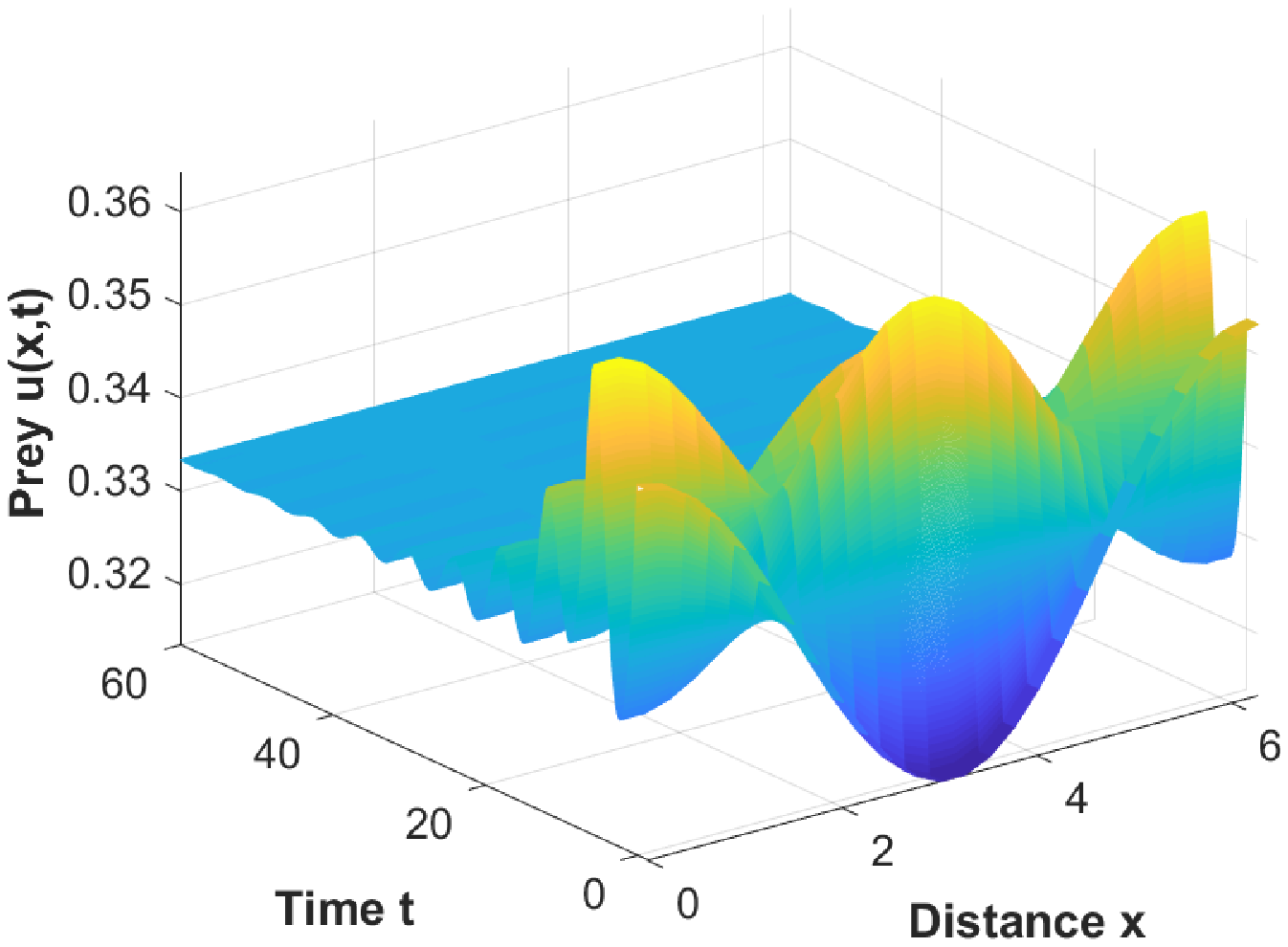}
\includegraphics[width=2in]{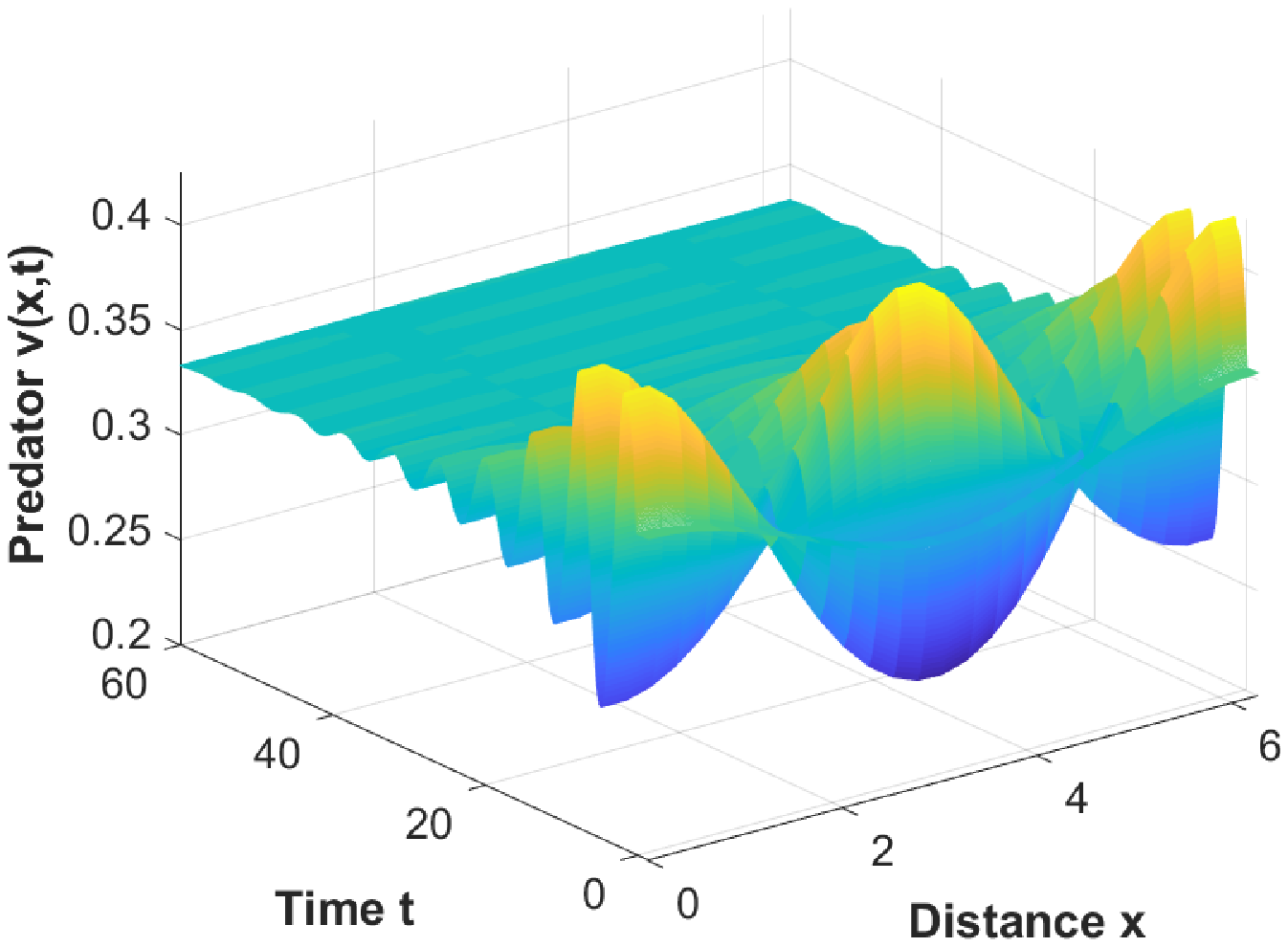} \\
\textbf{(a)} \hspace{4.5cm} \textbf{(b)} \\
\includegraphics[width=2in]{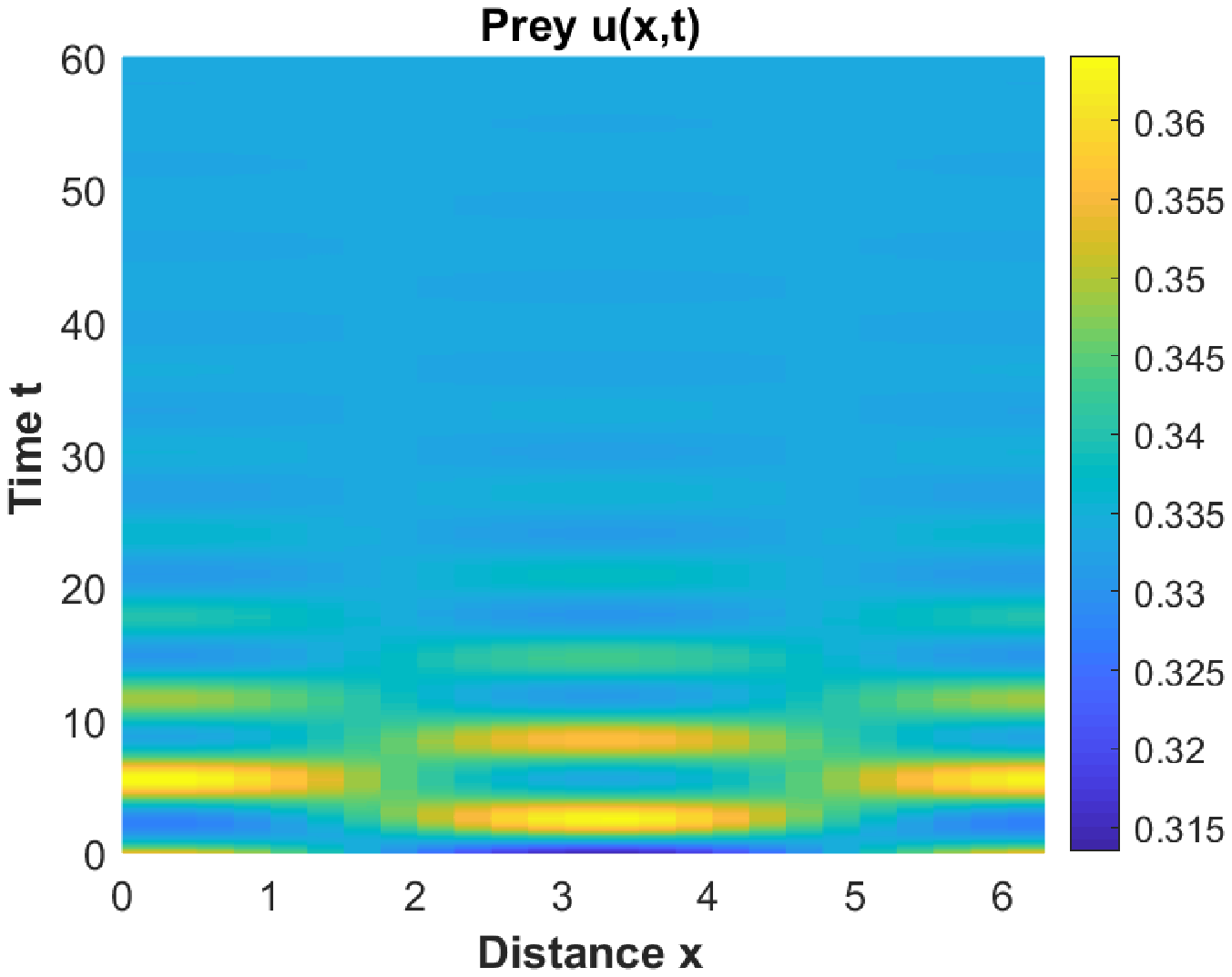}
\includegraphics[width=2in]{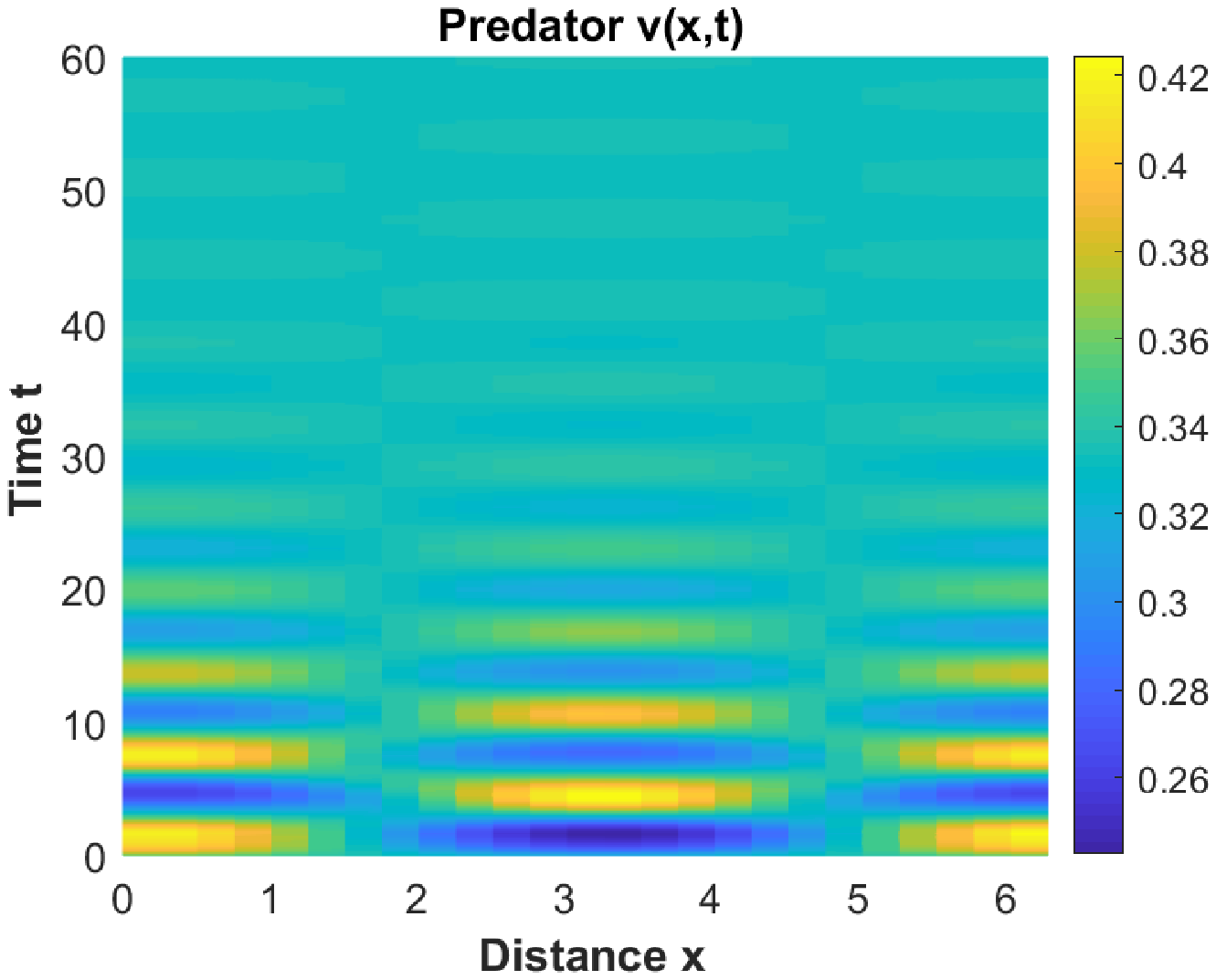} \\
\textbf{(c)} \hspace{4.5cm} \textbf{(d)} \\
\caption{For the parameters $d_{11}=0.6,~d_{22}=0.8,~m=0.5,~\gamma=0.5,~\beta=1,~\ell=2$, and $(d_{21},\tau)$ is chosen as the point $P_{3}(21,1.5)$ which satisfies $d_{21}^{(2)}<d_{21}<42.87$ and $\tau<\tau_{*}(d_{21}^{(2)})=2.5896$. There exists an unstable spatially inhomogeneous periodic solution, and its amplitude is decreasing. (a) and (b) are the evolution processes of the solutions $u(x,t)$ and $v(x,t)$ of system (4.2), respectively. (c) and (d) are spatio-temporal diagrams of the solutions $u(x,t)$ and $v(x,t)$ of system (4.2), respectively. The initial values are $u_{0}(x)=0.3333+0.02\cos(x),~v_{0}(x)=0.3333+0.02\cos(x)$.}
\label{fig:3}
\end{figure}

For $d_{21}=43$ which satisfies $42.87<d_{21}<150$, it follows from (4.17) and (4.18) that
\begin{eqnarray}
\tau_{3,0}=0.813<\tau_{2,0}=0.8138.
\end{eqnarray}

According to (4.23), we know that system (4.2) undergoes Hopf bifurcation at $\tau_{3,0}=0.813$. Furthermore, the direction and stability of the Hopf bifurcation can be determined by calculating $K_{1}K_{2}$ and $K_{2}$ using the procedures listed in Appendix A. By a direct calculation, we obtain
\begin{eqnarray*}
K_{1}=0.4024>0,~K_{2}=326.1951>0,~K_{1}K_{2}=131.2501>0,
\end{eqnarray*}
which implies that the spatially inhomogeneous Hopf bifurcation at $\tau_{3,0}$ is subcritical and unstable. When $d_{21}=43$ and $\tau=0.4$, by combining with (4.21) and (4.22), we can see that the point $P_{4}(43,0.4)$ satisfies $42.87<d_{21}<150$ and $\tau=0.4<\tau_{3,0}=0.813$. There exists an unstable spatially inhomogeneous periodic solution, and its amplitude is decreasing, see Fig.4 (a)-(d) for detail. The initial values are $u_{0}(x)=0.3333+0.02\cos(3x/2),~v_{0}(x)=0.3333+0.02\cos(3x/2)$.

\begin{figure}[!htbp]
\centering
\includegraphics[width=2in]{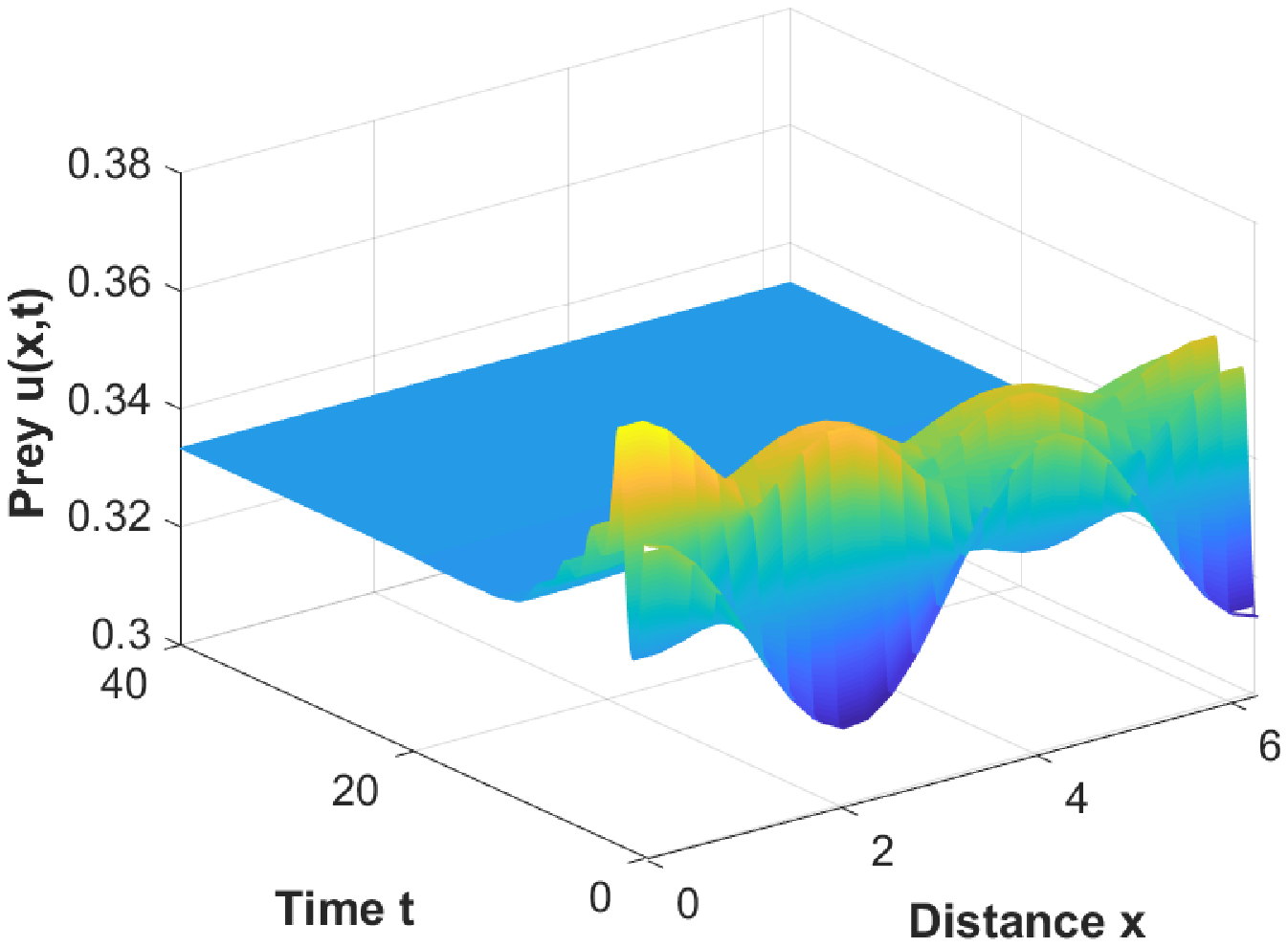}
\includegraphics[width=2in]{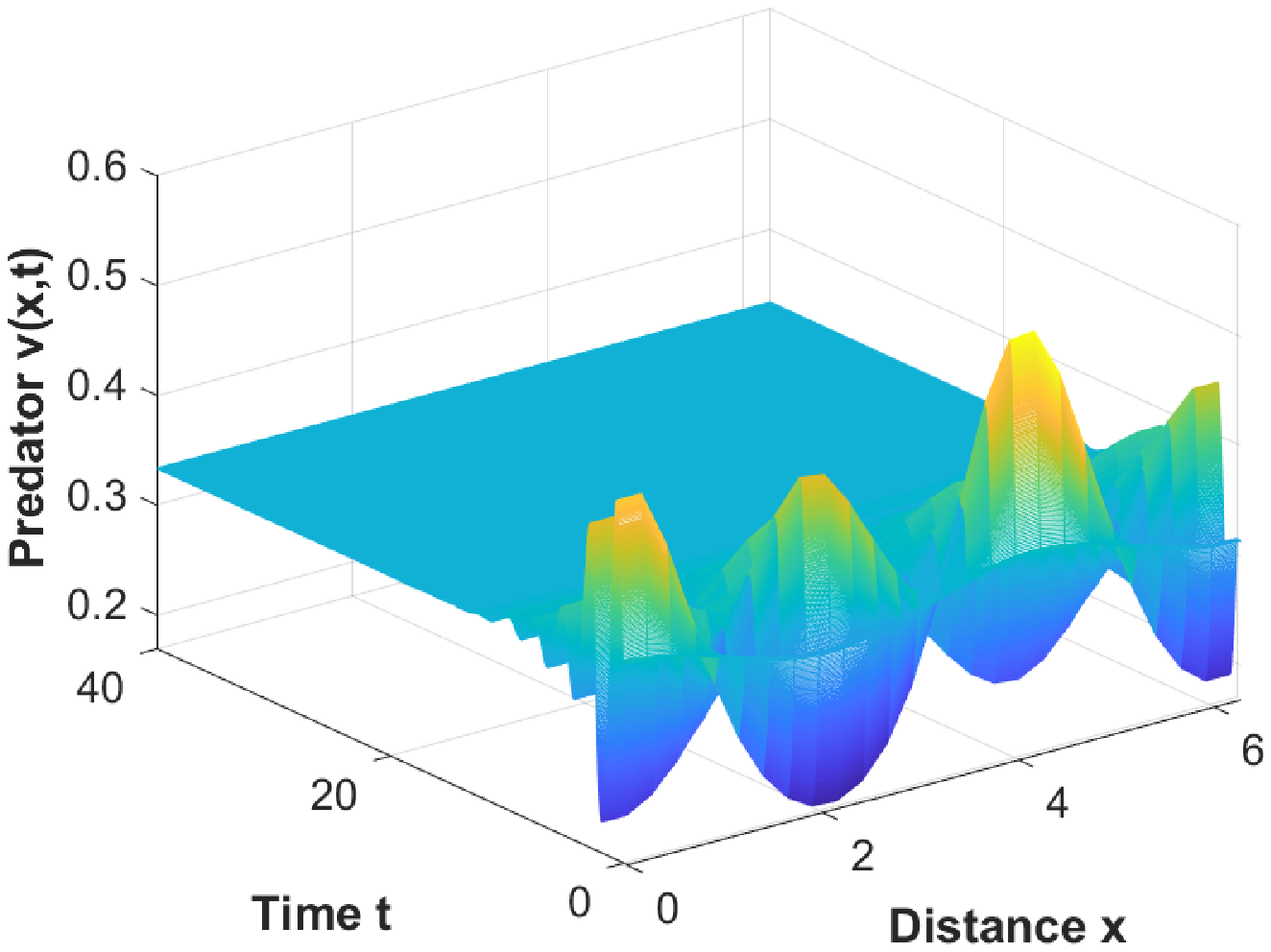} \\
\textbf{(a)} \hspace{4.5cm} \textbf{(b)} \\
\includegraphics[width=2in]{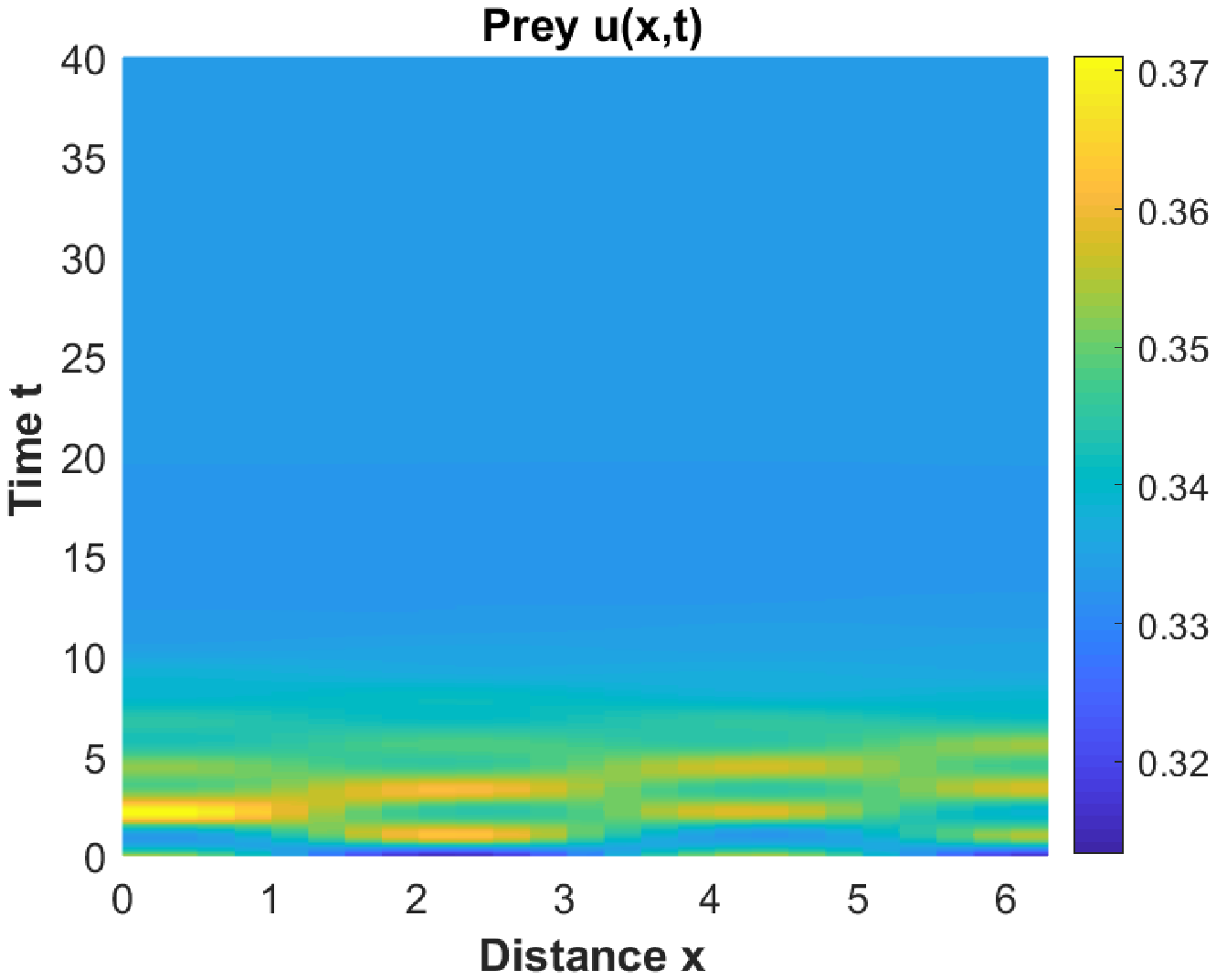}
\includegraphics[width=2in]{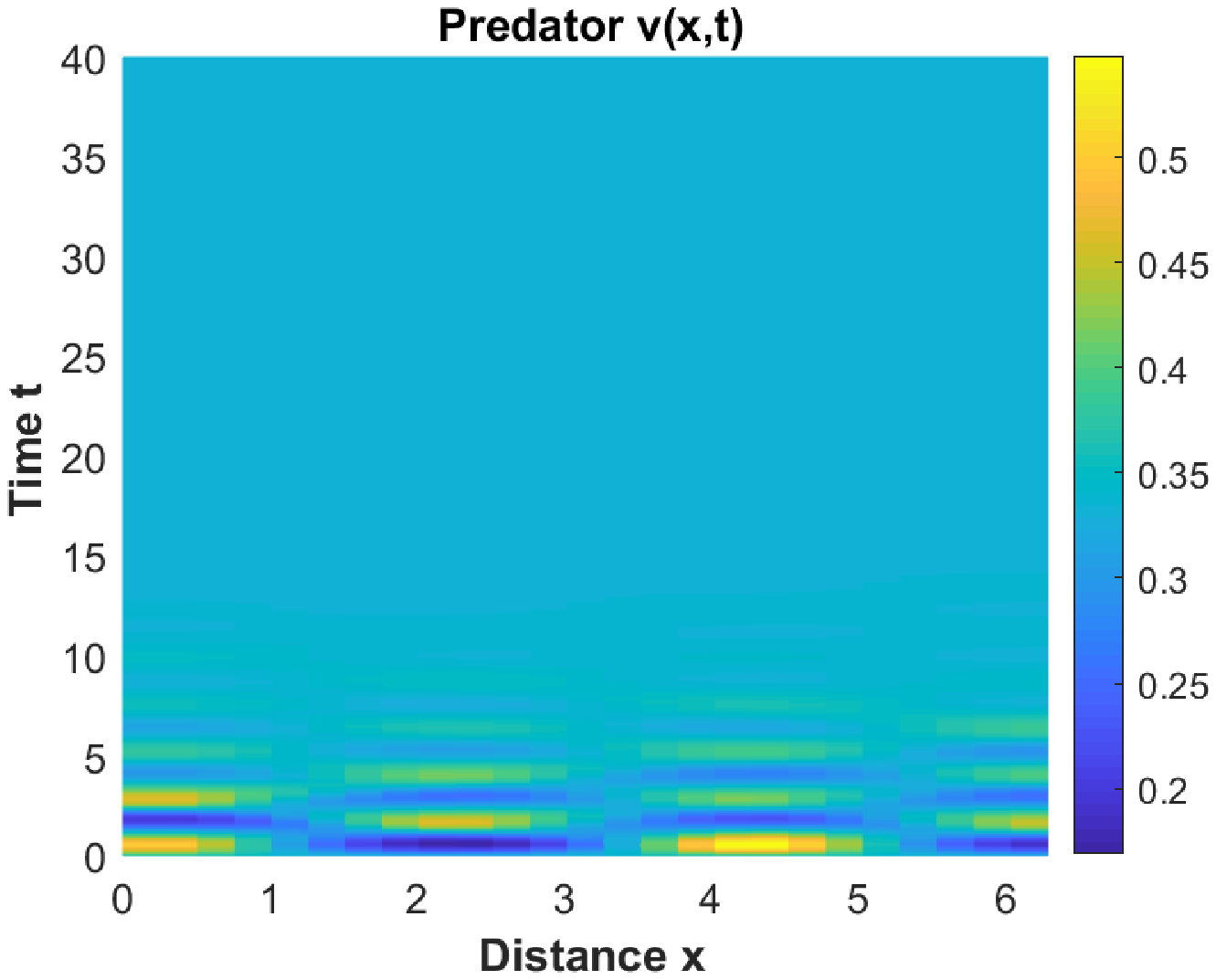} \\
\textbf{(c)} \hspace{4.5cm} \textbf{(d)} \\
\caption{For the parameters $d_{11}=0.6,~d_{22}=0.8,~m=0.5,~\gamma=0.5,~\beta=1,~\ell=2$, and $(d_{21},\tau)$ is chosen as the point $P_{4}(43,0.4)$ which satisfies $42.87<d_{21}<150$ and $\tau<\tau_{*}(d_{21}^{(3)})=0.813$. There exists an unstable spatially inhomogeneous periodic solution, and its amplitude is decreasing. (a) and (b) are the evolution processes of the solutions $u(x,t)$ and $v(x,t)$ of system (4.2), respectively. (c) and (d) are spatio-temporal diagrams of the solutions $u(x,t)$ and $v(x,t)$ of system (4.2), respectively. The initial values are $u_{0}(x)=0.3333+0.02\cos(3x/2),~v_{0}(x)=0.3333+0.02\cos(3x/2)$.}
\label{fig:4}
\end{figure}

\subsection{The case of with memory and gestation delays}

When system (4.1) includes memory and gestation delays, that is to say, in the model (1.3), we let
\begin{eqnarray*}\begin{aligned}
&f\left(u(x,t),v(x,t)\right)=u(x,t)\left(1-u(x,t)\right)-\frac{\beta u^{2}(x,t) v(x,t)}{u^{2}(x,t)+mv^{2}(x,t)}, \\
&g\left(u(x,t),v(x,t)\right)=\gamma v(x,t)\left(1-\frac{v(x,t)}{u(x,t-\tau)}\right).
\end{aligned}\end{eqnarray*}

Then the model (1.3) can be written as
\begin{eqnarray}
\left\{\begin{aligned}
&\frac{\partial u(x,t)}{\partial t}=d_{11}\Delta u(x,t)+u(x,t)\left(1-u(x,t)\right)-\frac{\beta u^{2}(x,t)v(x,t)}{u^{2}(x,t)+mv^{2}(x,t)}, & x\in(0,\ell\pi),~t>0, \\
&\frac{\partial v(x,t)}{\partial t}=d_{22}\Delta v(x,t)-d_{21}\left(v(x,t)u_{x}(x,t-\tau)\right)_{x}+\gamma v(x,t)\left(1-\frac{v(x,t)}{u(x,t-\tau)}\right), & x\in(0,\ell\pi),~t>0, \\
&u_{x}(0,t)=u_{x}(\ell\pi,t)=v_{x}(0,t)=v_{x}(\ell\pi,t)=0, & t \geq 0, \\
&u(x,t)=u_{0}(x,t),~v(x,t)=v_{0}(x,t), & x \in (0,\ell\pi),~-\tau \leq t \leq 0.
\end{aligned}\right.
\end{eqnarray}

Notice that for the system (4.24), the normal form for Hopf bifurcation can be calculated by using our developed algorithm in Section 2. In the following, we first give the stability and Hopf bifurcation analysis for the system (4.24), then by employing our developed procedure in Section 2 for calculating the normal form of Hopf bifurcation, the direction and stability of the Hopf bifurcation are determined.

\subsubsection{Stability and Hopf bifurcation analysis}

The system (4.24) has the positive constant steady state $E_{*}\left(u_{*},v_{*}\right)$, where
\begin{eqnarray}
u_{*}=v_{*}=1-\frac{\beta}{m+1}
\end{eqnarray}
with $0<\beta<m+1$. For $E_{*}\left(u_{*},v_{*}\right)$, form (2.4), when $m>1$, we have
\begin{eqnarray*}\begin{aligned}
&a_{11}=\frac{2\beta}{(m+1)^{2}}-1 \left\{
\begin{array}{cc}
\leq 0, & 0<\beta \leq \frac{(m+1)^{2}}{2}, \\
>0, & \beta>\frac{(m+1)^{2}}{2}.
\end{array}\right.
\end{aligned}\end{eqnarray*}
Notice that when $m>1$, if $a_{11}>0$, then we have $\beta>\frac{(m+1)^{2}}{2}>m+1$, which is contradict to the condition $0<\beta<m+1$. Thus, when $m>1$, $a_{11}\leq 0$ under the condition $0<\beta<m+1$. When $0<m<1$, we have
\begin{eqnarray}\begin{aligned}
&a_{11}=\frac{2\beta}{(m+1)^{2}}-1 \left\{
\begin{array}{cc}
\leq 0, & 0 < \beta \leq \frac{(m+1)^{2}}{2}, \\
>0, & \frac{(m+1)^{2}}{2}<\beta<m+1.
\end{array}\right.
\end{aligned}\end{eqnarray}
Figure 5 shows the curves $f_{1}=m+1$ and $f_{2}=(m+1)^{2}/2$ for $0\leq m\leq 3$, and they intersect at the point $P(1,2)$.

\begin{figure}[!htbp]
\centering
\includegraphics[width=2.5in]{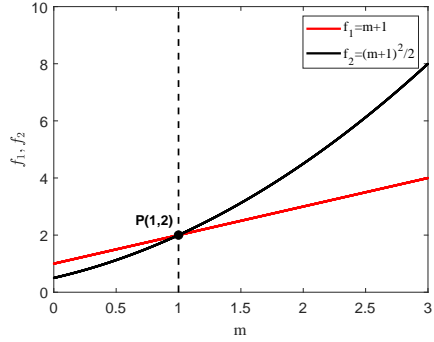} \\
\caption{The curves $f_{1}=m+1$ and $f_{2}=(m+1)^{2}/2$ for $0\leq m\leq 3$. The intersection point is $P(1,2)$.}
\label{fig:5}
\end{figure}

Furthermore, we have
\begin{eqnarray}\begin{aligned}
&a_{12}=\frac{\beta(m-1)}{(m+1)^{2}}\left\{
\begin{array}{cc}
\leq 0, & 0 < m \leq 1, \\
>0, & m>1,
\end{array}\right.~a_{21}=0,~a_{22}=-\gamma<0, \\
&b_{11}=0,~b_{12}=0,~b_{21}=\gamma>0,~b_{22}=0.
\end{aligned}\end{eqnarray}

Moreover, by combining with (4.26), (4.27),
\begin{eqnarray*}
D_{1}=\left(\begin{array}{cc}
d_{11} & 0 \\
0 & d_{22}
\end{array}\right),~D_{2}=\left(\begin{array}{cc}
0 & 0 \\
-d_{21} v_{*} & 0
\end{array}\right),~A_{1}=\left(\begin{array}{cc}
\frac{2 \beta}{(m+1)^{2}}-1 & \frac{\beta(m-1)}{(m+1)^{2}} \\
0 & -\gamma
\end{array}\right),~A_{2}=\left(\begin{array}{cc}
0 & 0 \\
\gamma & 0
\end{array}\right)
\end{eqnarray*}
and
\begin{eqnarray*}
M_{n}(\lambda)=\lambda I_{2}+\frac{n^{2}}{\ell^{2}}D_{1}+\frac{n^{2}}{\ell^{2}}e^{-\lambda \tau}D_{2}-A_{1}-A_{2}e^{-\lambda \tau},
\end{eqnarray*}
or according to (2.7), the characteristic equation of system (4.24) can be written as
\begin{eqnarray}
\Gamma_{n}(\lambda)=\det\left(M_{n}(\lambda)\right)=\lambda^{2}-T_{n}\lambda+\widetilde{J}_{n}(\tau)=0,
\end{eqnarray}
where
\begin{eqnarray}\begin{aligned}
T_{n}&=(a_{11}+a_{22})-(d_{11}+d_{22})\frac{n^{2}}{\ell^{2}}, \\
\widetilde{J}_{n}(\tau)&=d_{11}d_{22}\frac{n^{4}}{\ell^{4}}-\left(d_{11}a_{22}+d_{22}a_{11}+d_{21}a_{12}v_{*}e^{-\lambda \tau}\right)\frac{n^{2}}{\ell^{2}}+a_{11}a_{22}-a_{12}b_{21}e^{-\lambda\tau}.
\end{aligned}\end{eqnarray}

Notice that when $\tau=0$, the characteristic equation (4.28) becomes
\begin{eqnarray}
\lambda^{2}-T_{n} \lambda+\widetilde{J}_{n}(0)=0,
\end{eqnarray}
where $\widetilde{J}_{n}(0)$ is defined by
\begin{eqnarray*}
\widetilde{J}_{n}(0)=d_{11}d_{22}\frac{n^{4}}{\ell^{4}}-\left(d_{11}a_{22}+d_{22}a_{11}+d_{21}a_{12}v_{*}\right)\frac{n^{2}}{\ell^{2}}+a_{11}a_{22}-a_{12}b_{21}.
\end{eqnarray*}

A set of sufficient and necessary condition that all roots of (4.30) have a negative real part is $T_{n}<0,~\widetilde{J}_{n}(0)>0$, which is always holds provided that $a_{11}<0$ and $a_{12}<0$, i.e.,
\begin{eqnarray*}
(C_{0}):~0<\beta<\frac{(m+1)^{2}}{2},~0<m<1.
\end{eqnarray*}

This implies that when $\tau=0$ and the condition $(C_{0})$ holds, the positive steady state $E_{*}(u_{*},v_{*})$ is asymptotically stable for $d_{11} \geq 0$, $d_{21} \geq 0$ and $d_{22} \geq 0$. Meanwhile, if we let $d_{21}=0$, then we have
\begin{eqnarray*}
\widetilde{J}_{n}:=d_{11}d_{22}\frac{n^{4}}{\ell^{4}}-\left(d_{11}a_{22}+d_{22}a_{11}\right)\frac{n^{2}}{\ell^{2}}+a_{11}a_{22}-a_{12}b_{21}.
\end{eqnarray*}

It is easy to verify that $T_{n}<0$ and $\widetilde{J}_{n}>0$ provided that the condition $(C_{0})$ holds. This implies that when $d_{21}=0$, $\tau=0$ and the condition $(C_{0})$ holds, the positive steady state $E_{*}(u_{*},v_{*})$ is asymptotically stable for $d_{11} \geq 0$ and $d_{22} \geq 0$. Furthermore, since $\Gamma_{n}(0)=\widetilde{J}_{n}(0)>0$ under the condition $(C_{0})$, this implies that $\lambda=0$ is not a root of (4.28).

In the following, we let
\begin{eqnarray}\begin{aligned}
J_{n}=d_{11}d_{22}\frac{n^{4}}{\ell^{4}}-\left(d_{11}a_{22}+d_{22}a_{11}\right)\frac{n^{2}}{\ell^{2}}+a_{11}a_{22}.
\end{aligned}\end{eqnarray}
Furthermore, let $\lambda=i \omega_{n}(\omega_{n}>0)$ be a root of (4.28). By substituting it along with expressions in (4.26) and (4.27) into (4.28), and separating the real part from the imaginary part, we have
\begin{eqnarray}
\left\{\begin{aligned}
&J_{n}-\omega_{n}^{2}=\left(d_{21}a_{12}v_{*}\frac{n^{2}}{\ell^{2}}+a_{12}b_{21}\right)\cos(\omega_{n}\tau), \\
&T_{n}\omega_{n}=\left(d_{21}a_{12}v_{*}\frac{n^{2}}{\ell^{2}}+a_{12}b_{21}\right)\sin(\omega_{n}\tau),
\end{aligned}\right.
\end{eqnarray}
which yields
\begin{eqnarray}
\omega_{n}^{4}+P_{n}\omega_{n}^{2}+Q_{n}=0,
\end{eqnarray}
where
\begin{eqnarray}\begin{aligned}
P_{n}&=T_{n}^{2}-2J_{n} \\
&=\left(d_{11}^{2}+d_{22}^{2}\right)\frac{n^{4}}{\ell^{4}}-2\left(d_{11}a_{11}+d_{22}a_{22}\right)\frac{n^{2}}{\ell^{2}}+a_{11}^{2}+a_{22}^{2},
\end{aligned}\end{eqnarray}
and
\begin{eqnarray}
Q_{n}=\left(J_{n}+\left(d_{21}a_{12}v_{*}\frac{n^{2}}{\ell^{2}}+a_{12}b_{21}\right)\right)\left(J_{n}-\left(d_{21} a_{12}v_{*}\frac{n^{2}}{\ell^{2}}+a_{12}b_{21}\right)\right).
\end{eqnarray}

It is easy to verify that $P_{n}>0$ for any $n \in \mathbb{N}_{0}$. Thus, (4.33) has one positive root when $Q_{n}<0$. In the following, we will discuss several cases under the condition $(C_{0})$, which are used to guarantee $Q_{n}<0$.

When $\tau>0$, according to (4.31) and (4.35), we can define $Q_{n}=\Gamma_{n}(0)\widetilde{Q}_{n}$ with
\begin{eqnarray*}
\Gamma_{n}(0)=\widetilde{J}_{n}(0)=d_{11}d_{22}\frac{n^{4}}{\ell^{4}}-\left(d_{11}a_{22}+d_{22}a_{11}+d_{21}a_{12}v_{*}\right)\frac{n^{2}}{\ell^{2}}+a_{11}a_{22}-a_{12}b_{21}
\end{eqnarray*}
and
\begin{eqnarray}
\widetilde{Q}_{n}=d_{11}d_{22}\frac{n^{4}}{\ell^{4}}-\left(d_{11}a_{22}+d_{22}a_{11}-d_{21}a_{12}v_{*}\right)\frac{n^{2}}{\ell^{2}}+a_{11}a_{22}+a_{12}b_{21}=0,
\end{eqnarray}
and then by a simple analysis, we have $\Gamma_{n}(0)=\widetilde{J}_{n}(0)>0$ for any $n \in \mathbb{N}_{0}$. Therefore, the sign of $Q_{n}$ coincides with that of $\widetilde{Q}_{n}$, and in order to guaranteeing $Q_{n}<0$, we only need to study the case of $\widetilde{Q}_{n}<0$.

\begin{case}
It is easy to see that if the conditions $(C_{0})$ and
\begin{eqnarray*}
(C_{1}):~d_{11}a_{22}+d_{22}a_{11}-d_{21}a_{12}v_{*}<0,~a_{11}a_{22}+a_{12}b_{21}>0
\end{eqnarray*}
or
\begin{eqnarray*}
(C_{11}):~\left(d_{11}a_{22}+d_{22}a_{11}-d_{21}a_{12}v_{*}\right)^{2}-4d_{11}d_{22}\left(a_{11}a_{22}+a_{12}b_{21}\right)<0
\end{eqnarray*}
holds, then (4.36) has no positive roots. Hence, all roots of (4.28) have negative real parts when $\tau \in[0,+\infty)$ under the conditions $(C_{0})$ and $(C_{1})$ or $(C_{11})$.
\end{case}

\begin{case}
If the conditions $(C_{0})$ and
\begin{eqnarray*}
(C_{2}):~a_{11}a_{22}+a_{12}b_{21}<0
\end{eqnarray*}
or
\begin{eqnarray*}\begin{aligned}
&(C_{21}):~d_{11}a_{22}+d_{22}a_{11}-d_{21}a_{12}v_{*}>0, \\
&\left(d_{11}a_{22}+d_{22}a_{11}-d_{21}a_{12}v_{*}\right)^{2}-4d_{11}d_{22}\left(a_{11}a_{22}+a_{12}b_{21}\right)=0
\end{aligned}\end{eqnarray*}
hold, then (4.36) has a positive root. Notice that when $d_{11}a_{22}+d_{22}a_{11}-d_{21}a_{12}v_{*}>0$, $\widetilde{Q}_{n}\geq0$ for all $n\in \mathbb{N}_{0}$ and we only need to study the case of $\widetilde{Q}_{n}<0$, thus for the case 4.5, we only consider the condition $(C_{2})$. Moreover, if we let $\widetilde{x}=n^{2}/\ell^{2}$, then the mathematical expression of $\widetilde{Q}_{n}$ can be rewritten as
\begin{eqnarray}
\widetilde{f}(\widetilde{x})=d_{11}d_{22}\widetilde{x}^{2}-(d_{11}a_{22}+d_{22}a_{11}-d_{21}a_{12}v_{*})\widetilde{x}+a_{11}a_{22}+a_{12}b_{21},
\end{eqnarray}
and the unique positive root of this equation is
\begin{eqnarray}
\widetilde{x}_{*}=\frac{d_{11}a_{22}+d_{22}a_{11}-d_{21}v_{*}a_{12}+\sqrt{(d_{11}a_{22}+d_{22}a_{11}-d_{21}v_{*}a_{12})^{2}-4d_{11}d_{22}(a_{11}a_{22}+a_{12}b_{21})}}{2d_{11}d_{22}}
\end{eqnarray}
under the conditions $(C_{0})$ and $(C_{2})$ or $(C_{21})$. Since $\widetilde{x}_{*}=n^{2}/\ell^{2}$, then $n_{0}=\ell\sqrt{\widetilde{x}_{*}}$, and notice that $\widetilde{Q}_{n}$ is a quadratic polynomial with respect to $n^{2}/\ell^{2}$ and $\widetilde{Q}_{0}\leq 0$ under the condition $(C_{2})$. Thus, when the condition $(C_{2})$ holds, we can conclude that there exists $n_{0}>0$ such that $\widetilde{Q}_{n_{0}}=0$ and
\begin{eqnarray}
Q_{n}=\Gamma_{n}(0)\widetilde{Q}_{n}\left\{\begin{aligned}
&<0, & 0 \leq n \leq n_{*}, \\
&\geq 0, & n \geq n_{*}+1,
\end{aligned}
\right.\end{eqnarray}
where $n \in \mathbb{N}_{0}$, and $n_{*}$ is defined by
\begin{eqnarray}
n_{*}=\left\{\begin{aligned}
& n_{0}-1, & n_{0} \in \mathbb{N}, \\
& \left[n_{0}\right], & n_{0} \notin \mathbb{N}.
\end{aligned}\right.
\end{eqnarray}
Here, $\left[.\right]$ stands for the integer part function. Therefore, (4.33) has one positive root $\omega_{n}$ for $0 \leq n \leq n_{*}$ with $n\in \mathbb{N}_{0}$, where
\begin{eqnarray}
\omega_{n}=\sqrt{\frac{-P_{n}+\sqrt{P_{n}^{2}-4Q_{n}}}{2}}.
\end{eqnarray}
By combining with (4.32), and notice that $a_{12}<0$, $T_{n}<0$ under the condition $(C_{0})$, then we have
\begin{eqnarray*}
\sin(\omega_{n} \tau)=\frac{T_{n}\omega_{n}}{d_{21}a_{12}v_{*}(n^{2}/\ell^{2})+a_{12}b_{21}}>0.
\end{eqnarray*}
Thus, from the first mathematical expression in (4.32), we can set
\begin{eqnarray}
\tau_{n,j}=\frac{1}{\omega_{n}}\left\{\arccos \left\{\frac{J_{n}-\omega_{n}^{2}}{d_{21}a_{12}v_{*}(n^{2}/\ell^{2})+a_{12}b_{21}}\right\}+2j\pi\right\},~n \in \mathbb{N}_{0},~j \in \mathbb{N}_{0}.
\end{eqnarray}
\end{case}

\begin{case}
If the conditions $(C_{0})$ and
\begin{eqnarray*}\begin{aligned}
&(C_{3}):~d_{11}a_{22}+d_{22}a_{11}-d_{21}a_{12}v_{*}>0,~a_{11}a_{22}+a_{12}b_{21}>0, \\
&\left(d_{11}a_{22}+d_{22}a_{11}-d_{21}a_{12}v_{*}\right)^{2}-4d_{11}d_{22}\left(a_{11}a_{22}+a_{12}b_{21}\right)>0
\end{aligned}\end{eqnarray*}
hold, then the (4.36) has two positive roots. Without loss of generality, we assume that the two positive roots of (4.37) are $\widetilde{x}_{1}$ and $\widetilde{x}_{2}$, i.e.,
\begin{eqnarray*}
\widetilde{x}_{1,2}=\frac{d_{11}a_{22}+d_{22}a_{11}-d_{21}v_{*}a_{12}\mp\sqrt{(d_{11}a_{22}+d_{22}a_{11}-d_{21}v_{*}a_{12})^{2}-4d_{11}d_{22}(a_{11}a_{22}+a_{12}b_{21})}}{2d_{11}d_{22}}
\end{eqnarray*}
under the conditions $(C_{0})$ and $(C_{3})$. Since $\widetilde{x}_{1}=n_{1}^{2}/\ell^{2}$ and $\widetilde{x}_{2}=n_{2}^{2}/\ell^{2}$, then $n_{1}=\ell\sqrt{\widetilde{x}_{1}}$ and $n_{2}=\ell\sqrt{\widetilde{x}_{2}}$. By using a geometric argument, we can conclude that
\begin{eqnarray*}
Q_{n}=\Gamma_{n}(0)\widetilde{Q}_{n}\left\{\begin{array}{cc}
<0, & n_{1}<n<n_{2}, \\
\geq 0, & n\leq n_{1} \text{ or } n \geq n_{2},
\end{array}
\right.\end{eqnarray*}
where $n \in \mathbb{N}_{0}$. Therefore, (4.33) has one positive root $\omega_{n}^{+}$ for $n_{1}<n<n_{2}$ with $n\in \mathbb{N}_{0}$, where
\begin{eqnarray*}
\omega_{n}^{+}=\sqrt{\frac{-P_{n}+\sqrt{P_{n}^{2}-4Q_{n}}}{2}}.
\end{eqnarray*}
Furthermore, by combining with the second mathematical expression in (4.32), and notice that $a_{12}<0$, $T_{n}<0$ under the condition $(C_{0})$, then we have $\sin(\omega_{n}^{+}\tau)>0$. Thus, from the first mathematical expression in (4.32), we can set
\begin{eqnarray*}
\tau_{n,j}^{+}=\frac{1}{\omega_{n}^{+}}\left\{\arccos\left\{\frac{J_{n}-(\omega_{n}^{+})^{2}}{d_{21}a_{12}v_{*}(n^{2}/\ell^{2})+a_{12}b_{21}}\right\}+2j\pi\right\},~n \in \mathbb{N}_{0},~j \in \mathbb{N}_{0}.
\end{eqnarray*}
\end{case}

Next, we continue to verify the transversality conditions for the Cases 4.5 and 4.6.
\begin{lemma}
Suppose that the conditions $(C_{0})$ and $(C_{2})$ hold, and $0 \leq n \leq n_{*}$ with $n \in \mathbb{N}_{0}$, then we have
\begin{eqnarray*}
\left.\frac{d\operatorname{Re}(\lambda(\tau))}{d\tau}\right|_{\tau=\tau_{n,j}}>0,
\end{eqnarray*}
where $\operatorname{Re}(\lambda(\tau))$ represents the real part of $\lambda(\tau)$.
\end{lemma}

\begin{proof}
By differentiating the two sides of
\begin{eqnarray*}
\Gamma_{n}(\lambda)=\operatorname{det}\left(M_{n}(\lambda)\right)=\lambda^{2}-T_{n}\lambda+\widetilde{J}_{n}(\tau)=0
\end{eqnarray*}
with respect to $\tau$, where $T_{n}$ and $\widetilde{J}_{n}(\tau)$ are defined by (4.29), we have
\begin{eqnarray}
\left(\frac{d\lambda(\tau)}{d\tau}\right)^{-1}=\frac{(2\lambda-T_{n})e^{\lambda \tau}}{-\lambda\left( d_{21}a_{12}v_{*}(n^{2}/\ell^{2})+a_{12}b_{21}\right)}-\frac{\tau}{\lambda}.
\end{eqnarray}

Therefore, by (4.43), we have
\begin{eqnarray}\begin{aligned}
\operatorname{Re}\left(\left.\frac{d \lambda(\tau)}{d \tau}\right|_{\tau=\tau_{n,j}}\right)^{-1}&=\operatorname{Re}\left(\frac{(2i\omega_{n}-T_{n})e^{i\omega_{n} \tau_{n,j}}}{-i\omega_{n}\left( d_{21}a_{12}v_{*}(n^{2}/\ell^{2})+a_{12}b_{21}\right)}\right) \\
&=\operatorname{Re}\left(\frac{(2i\omega_{n}-T_{n})\left(\cos(\omega_{n}\tau_{n,j})+i\sin(\omega_{n}\tau_{n,j})\right) }{-i\omega_{n}\left(d_{21}a_{12}v_{*}(n^{2}/\ell^{2})+a_{12}b_{21}\right)}\right) \\
&=\operatorname{Re}\left(\frac{(2i\omega_{n}-T_{n})\cos(\omega_{n}\tau_{n,j})}{-i\omega_{n}\left(d_{21}a_{12}v_{*}(n^{2}/\ell^{2})+a_{12}b_{21}\right)}+\frac{i(2i\omega_{n}-T_{n})\sin(\omega_{n}\tau_{n,j})}{-i\omega_{n}\left(d_{21}a_{12}v_{*}(n^{2}/\ell^{2})+a_{12}b_{21}\right)} \right) \\
&=\frac{T_{n}\sin(\omega_{n}\tau_{n,j})}{\omega_{n}\left(d_{21}a_{12}v_{*}(n^{2}/\ell^{2})+a_{12}b_{21}\right)}-\frac{2\cos(\omega_{n}\tau_{n,j})}{\left(d_{21}a_{12}v_{*}(n^{2}/\ell^{2})+a_{12}b_{21}\right)}.
\end{aligned}\end{eqnarray}
Furthermore, according to (4.32), we have
\begin{eqnarray}
\sin(\omega_{n} \tau_{n,j})=\frac{T_{n}\omega_{n}}{\left(d_{21}a_{12}v_{*}(n^{2}/\ell^{2})+a_{12}b_{21}\right)},~\cos(\omega_{n} \tau_{n,j})=\frac{J_{n}-\omega_{n}^{2}}{\left(d_{21}a_{12}v_{*}(n^{2}/\ell^{2})+a_{12}b_{21}\right)}.
\end{eqnarray}
Moreover, by combining with (4.44), (4.45) and
\begin{eqnarray*}
\omega_{n}=\sqrt{\frac{-P_{n}+\sqrt{P_{n}^{2}-4Q_{n}}}{2}}>0,~a_{12}<0,~P_{n}=T_{n}^{2}-2J_{n}>0,~Q_{n}<0,
\end{eqnarray*}
we have
\begin{eqnarray*}\begin{aligned}
\operatorname{Re}\left(\left.\frac{d \lambda(\tau)}{d \tau}\right|_{\tau=\tau_{n,j}}\right)^{-1}&=\frac{T_{n}\sin(\omega_{n}\tau_{n,j})}{\omega_{n}\left(d_{21}a_{12}v_{*}(n^{2}/\ell^{2})+a_{12}b_{21}\right)}-\frac{2\cos(\omega_{n}\tau_{n,j})}{\left(d_{21}a_{12}v_{*}(n^{2}/\ell^{2})+a_{12}b_{21}\right)} \\
&=\frac{T_{n}^{2}-2(J_{n}-\omega_{n}^{2})}{\left(d_{21}a_{12}v_{*}(n^{2}/\ell^{2})+a_{12}b_{21}\right)^{2}}=\frac{\sqrt{P_{n}^{2}-4Q_{n}}}{\left(d_{21}a_{12}v_{*}(n^{2}/\ell^{2})+a_{12}b_{21}\right)^{2}}>0.
\end{aligned}\end{eqnarray*}

This, together with the fact that
\begin{eqnarray*}
\operatorname{sign}\left\{\left.\frac{d \operatorname{Re}(\lambda(\tau))}{d \tau}\right|_{\tau=\tau_{n,j}}\right\}=\operatorname{sign}\left\{\operatorname{Re}\left(\left.\frac{d\lambda(\tau)}{d \tau}\right|_{\tau=\tau_{n,j}}\right)^{-1}\right\}
\end{eqnarray*}
completes the proof, where $\operatorname{sign}(.)$ represents the sign function.
\end{proof}

\begin{remark}
Similarly, if we suppose that the conditions $(C_{0})$ and $(C_{3})$ hold, and $n_{1}<n<n_{2}$ with $n \in \mathbb{N}_{0}$, then we have
\begin{eqnarray*}
\left.\frac{d \operatorname{Re}(\lambda(\tau))}{d \tau}\right|_{\tau=\tau_{n,j}^{+}}>0.
\end{eqnarray*}
Notice that the transversality condition for $\tau=\tau_{n,j}^{+}$ can be verified by a similar argument in Lemma 4.7, we hence omit here.
\end{remark}

Moreover, according to the above analysis, we have the following results.
\begin{lemma}
If the condition $(C_{0})$ is satisfied, then we have the following conclusions:

(i) if the condition $(C_{1})$ or $(C_{11})$ holds, then the positive constant steady state $E_{*}\left(u_{*},v_{*}\right)$ of system (4.24) is asymptotically stable for all $\tau \geq 0$;

(ii) if the condition $(C_{2})$ holds, and denote $\tau_{*}=\min\left\{\tau_{n,0}: 0 \leq n \leq n_{*},~n \in \mathbb{N}_{0}\right\}$, then the positive constant steady state $E_{*}\left(u_{*},v_{*}\right)$ of system (4.24) is asymptotically stable for $0 \leq \tau<\tau_{*}$ and unstable for $\tau>\tau_{*}$. Furthermore, system (4.24) undergoes Hopf bifurcations at $\tau=\tau_{n,0}$ for $n \in \mathbb{N}_{0}$. If $n=0$, then the bifurcating periodic solutions are all spatially homogeneous, and when $n\geq 1$ and $n\in \mathbb{N}$, these bifurcating periodic solutions are spatially inhomogeneous;

(iii) if the condition $(C_{3})$ holds, and denote $\tau_{*}=\min\left\{\tau_{n,0}^{+}: n_{1}<n<n_{2},~n \in \mathbb{N}_{0}\right\}$, then the positive constant steady state $E_{*}\left(u_{*},v_{*}\right)$ of system (4.24) is asymptotically stable for $0 \leq \tau<\tau_{*}$ and unstable for $\tau>\tau_{*}$. Furthermore, system (4.24) undergoes Hopf bifurcations at $\tau=\tau_{n,0}^{+}$ for $n \in \mathbb{N}_{0}$. If $n=0$, then the bifurcating periodic solutions are all spatially homogeneous, and when $n\geq 1$ and $n\in \mathbb{N}$, these bifurcating periodic solutions are spatially inhomogeneous.
\end{lemma}

\subsubsection{Direction and stability of the Hopf bifurcation}

In this section, we verify the analytical results given in the previous sections by some numerical simulations and investigate the direction and stability of the Hopf bifurcation. We use the following initial conditions for the system (4.24)
\begin{eqnarray*}
u(x,t)=u_{0}(x),~v(x,t)=v_{0}(x),~t \in\left[-\tau,0\right],
\end{eqnarray*}
and we set the parameters as follows
\begin{eqnarray*}
d_{11}=0.6,~d_{21}=3.6,~d_{22}=0.8,~m=0.5,~\gamma=0.5,~\beta=1,~\ell=2,
\end{eqnarray*}
we can easily obtain that
\begin{eqnarray*}
0<\beta=1<m+1=1.5,~0<m=0.5<1,~a_{11}a_{22}+a_{12}b_{21}=-0.0556<0.
\end{eqnarray*}

Therefore, the conditions $(C_{0})$ and $(C_{2})$ are satisfied under the above parameters settings. In the following, we mainly verify the conclusion in Lemma 4.9 (ii). According to (4.25), (4.26) and (4.27), we have $E_{*}\left(u_{*}, v_{*}\right)=(0.3333,0.3333)$,
\begin{eqnarray*}\begin{aligned}
&a_{11}=-0.1111,~a_{12}=-0.2222,~a_{21}=0,~a_{22}=-0.5, \\
&b_{11}=0,~b_{12}=0,~b_{21}=0.5,~b_{22}=0.
\end{aligned}\end{eqnarray*}

It follows from (4.34) that
\begin{eqnarray*}
P_{n}=0.0625 n^{4}+0.2333 n^{2}+0.2623>0.
\end{eqnarray*}

Notice that $P_{n}>0$ for any $n \in \mathbb{N}_{0}$, which together with (4.39) and Lemma 4.9 (ii), implies that for a fixed $n$, (4.33) has only one positive root for $0\leq n\leq n_{*}$. Furthermore, by combining with (4.38), (4.39), (4.40), (4.41) and (4.42), we have $n_{*}=0$, $\omega_{c}=\omega_{0}=0.1775$ and $\tau_{c}=\tau_{0,0}=10.078$.

Moreover, by Lemma 4.9 (ii), we have the following proposition.

\begin{proposition}
For system (4.24) with the parameters $d_{11}=0.6,~d_{21}=3.6,~d_{22}=0.8,~m=0.5,~\gamma=0.5,~\beta=1,~\ell=2$, the positive constant steady state $E_{*}\left(u_{*},v_{*}\right)$ of system (4.24) is asymptotically stable for $0 \leq \tau<\tau_{0,0}=10.078$ and unstable for $\tau>\tau_{0,0}=10.078$. Furthermore, system (4.24) undergoes a Hopf bifurcation at the positive constant steady state $E_{*}\left(u_{*},v_{*}\right)$ when $\tau=\tau_{0,0}=10.078$.
\end{proposition}

For the parameters $d_{11}=0.6,~d_{21}=3.6,~d_{22}=0.8,~m=0.5,~\gamma=0.5,~\beta=1,~\ell=2$, according to Proposition 4.11, we know that system (4.24) undergoes Hopf bifurcation at $\tau_{0,0}=10.078$. Furthermore, the direction and stability of the Hopf bifurcation can be determined by calculating $K_{1}K_{2}$ and $K_{2}$ using the procedures developed in Section 2. After a direct calculation using MATLAB software, we obtain
\begin{eqnarray*}
K_{1}=0.0366>0,~K_{2}=-14.9167<0,~K_{1}K_{2}=-0.5454<0,
\end{eqnarray*}
which implies that the Hopf bifurcation at $\tau_{0,0}=10.078$ is supercritical and stable.

When $\tau=6<\tau_{0,0}=10.078$, Fig.6 (a)-(d) illustrate the evolution of the solution of system (4.24) starting from the initial values $u_{0}(x)=0.3333+0.01,~v_{0}(x)=0.3333+0.01$, finally converging to the positive constant steady state $E_{*}\left(u_{*},v_{*}\right)$. When $\tau=6<\tau_{0,0}=10.078$, Fig.7 shows the behavior and phase portrait of system (4.24). Furthermore, when $\tau=13>\tau_{0,0}=10.078$, Fig.8 (a)-(d) illustrate the existence of the spatially homogeneous periodic solution with the initial values $u_{0}(x)=0.3333-0.01,~v_{0}(x)=0.3333+0.01$. When $\tau=13>\tau_{0,0}=10.078$, Fig.9 shows the behavior and phase portrait of system (4.24).

\begin{figure}[!htbp]
\centering
\includegraphics[width=2in]{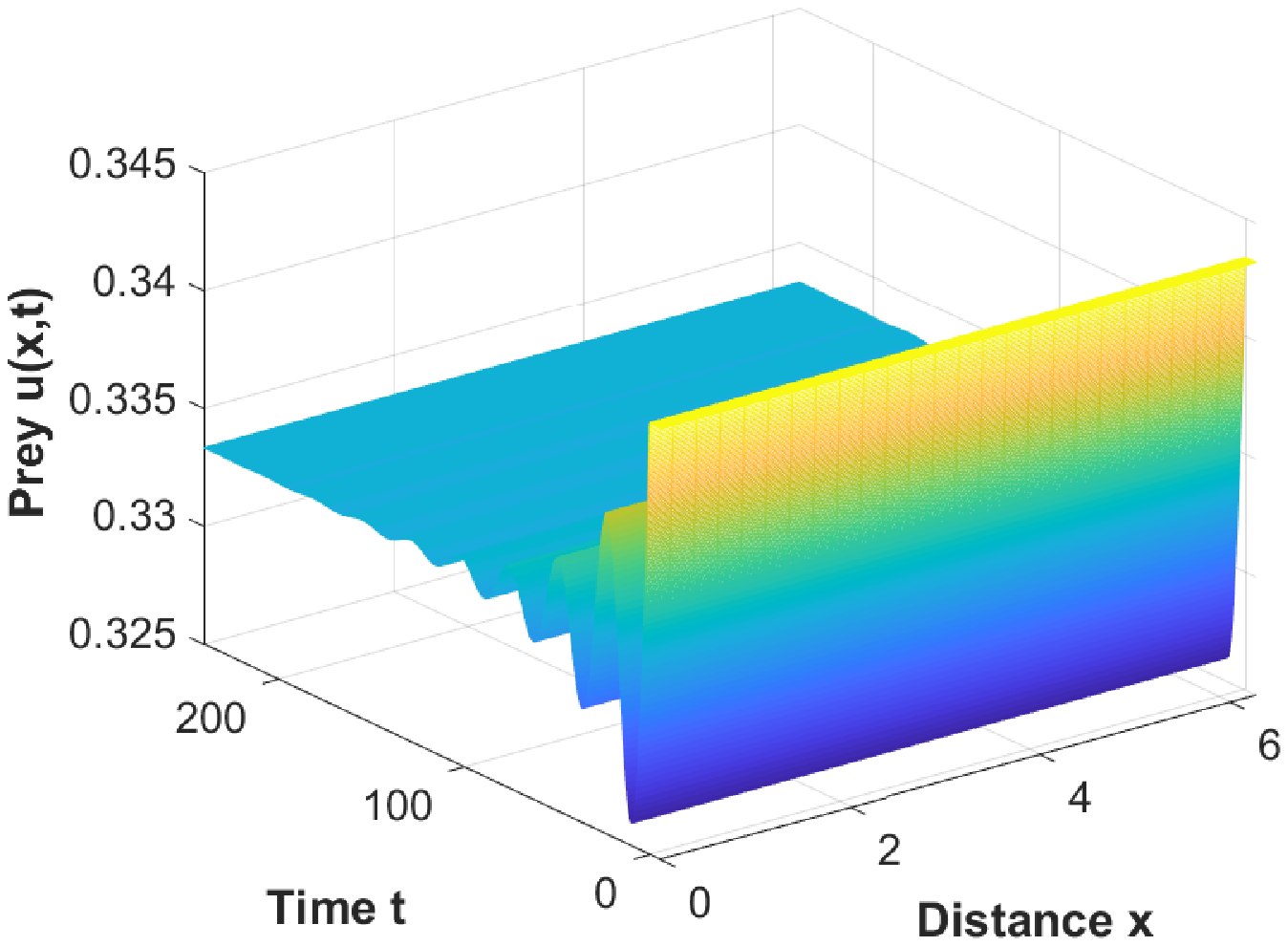}
\includegraphics[width=2in]{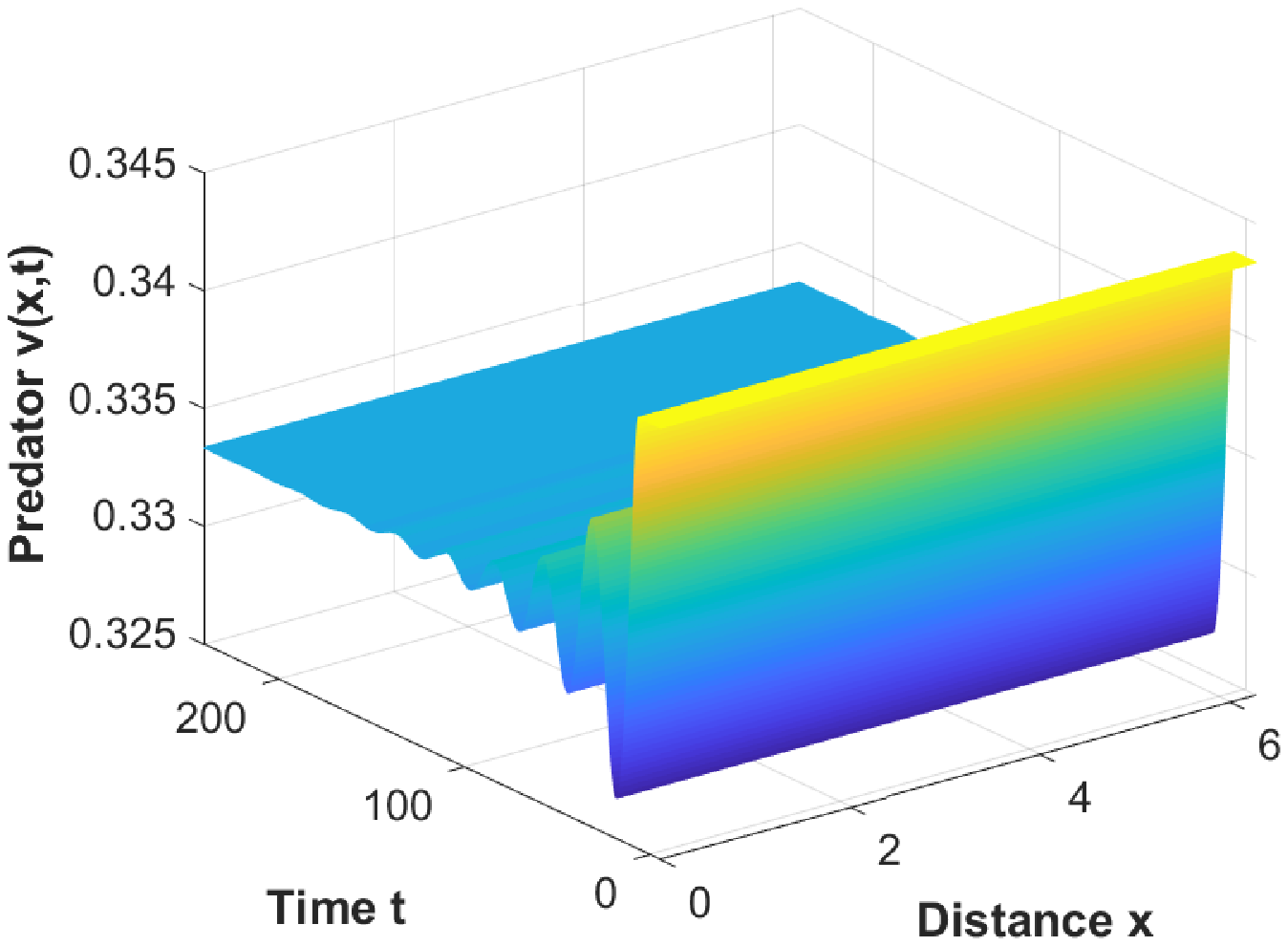} \\
\textbf{(a)} \hspace{4.5cm} \textbf{(b)} \\
\includegraphics[width=2in]{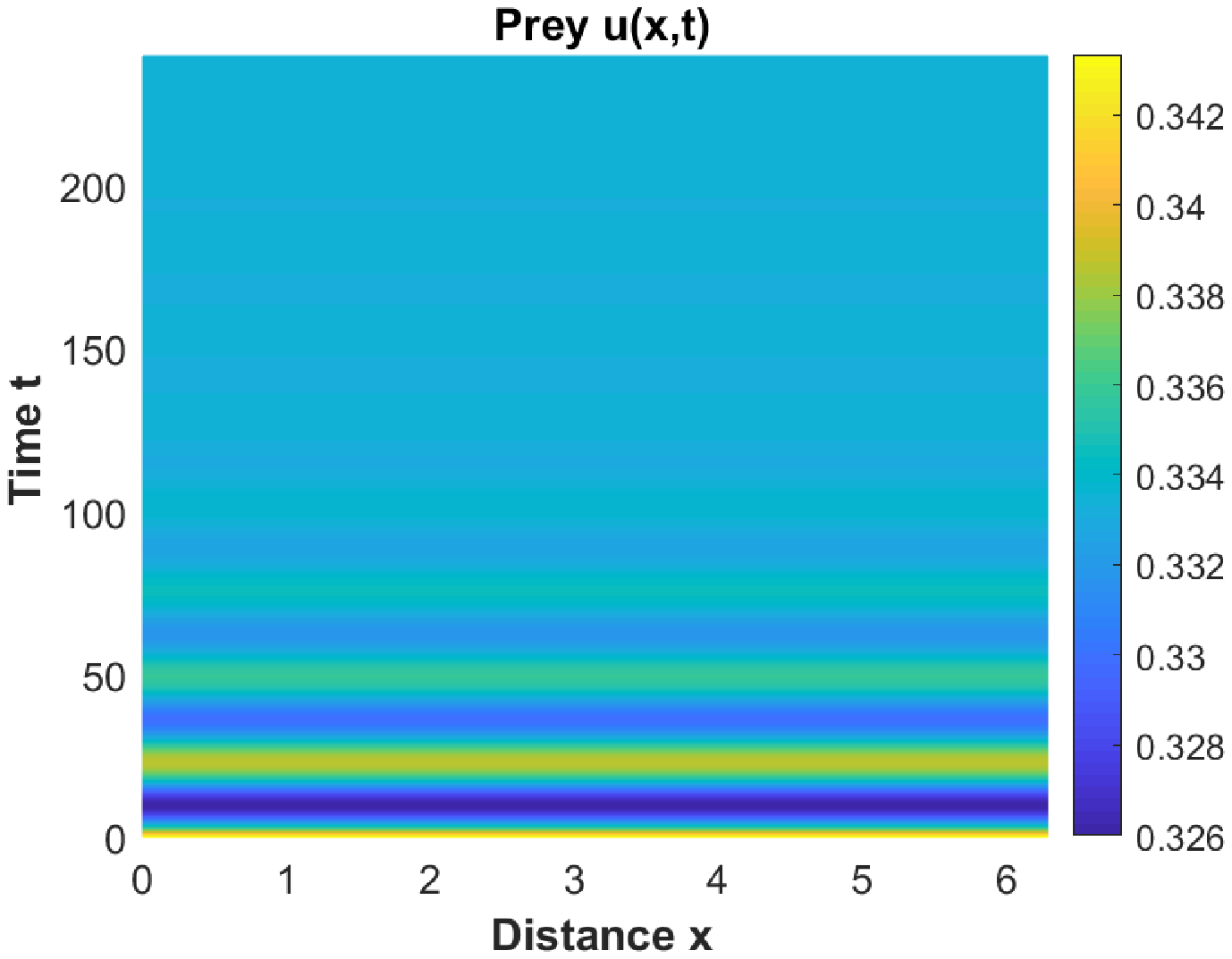}
\includegraphics[width=2in]{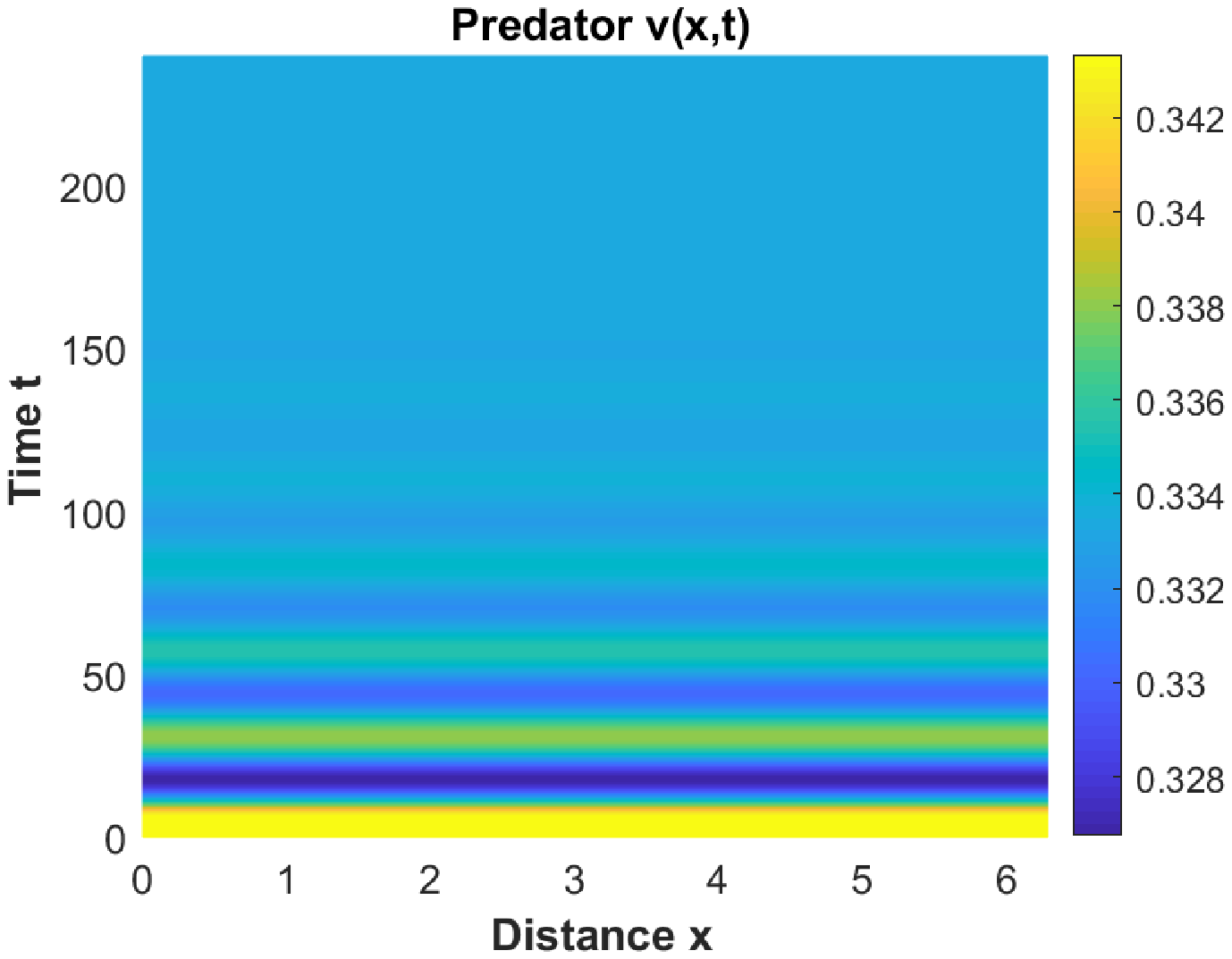} \\
\textbf{(c)} \hspace{4.5cm} \textbf{(d)} \\
\caption{For the parameters $d_{11}=0.6,~d_{22}=0.8,~m=0.5,~\gamma=0.5,~\beta=1,~\ell=2$ and $\tau=6<\tau_{0,0}=10.078$, the positive constant steady state $E_{*}\left(u_{*},v_{*}\right)=(0.3333,0.3333)$ is locally asymptotically stable. (a) and (b) are the evolution processes of the solutions $u(x,t)$ and $v(x,t)$ of system (4.24), respectively. (c) and (d) are spatio-temporal diagrams of the solutions $u(x,t)$ and $v(x,t)$ of system (4.24), respectively. The initial values are $u_{0}(x)=0.3333+0.01,~v_{0}(x)=0.3333+0.01$.}
\label{fig:6}
\end{figure}

\begin{figure}[!htbp]
\centering
\includegraphics[width=2in]{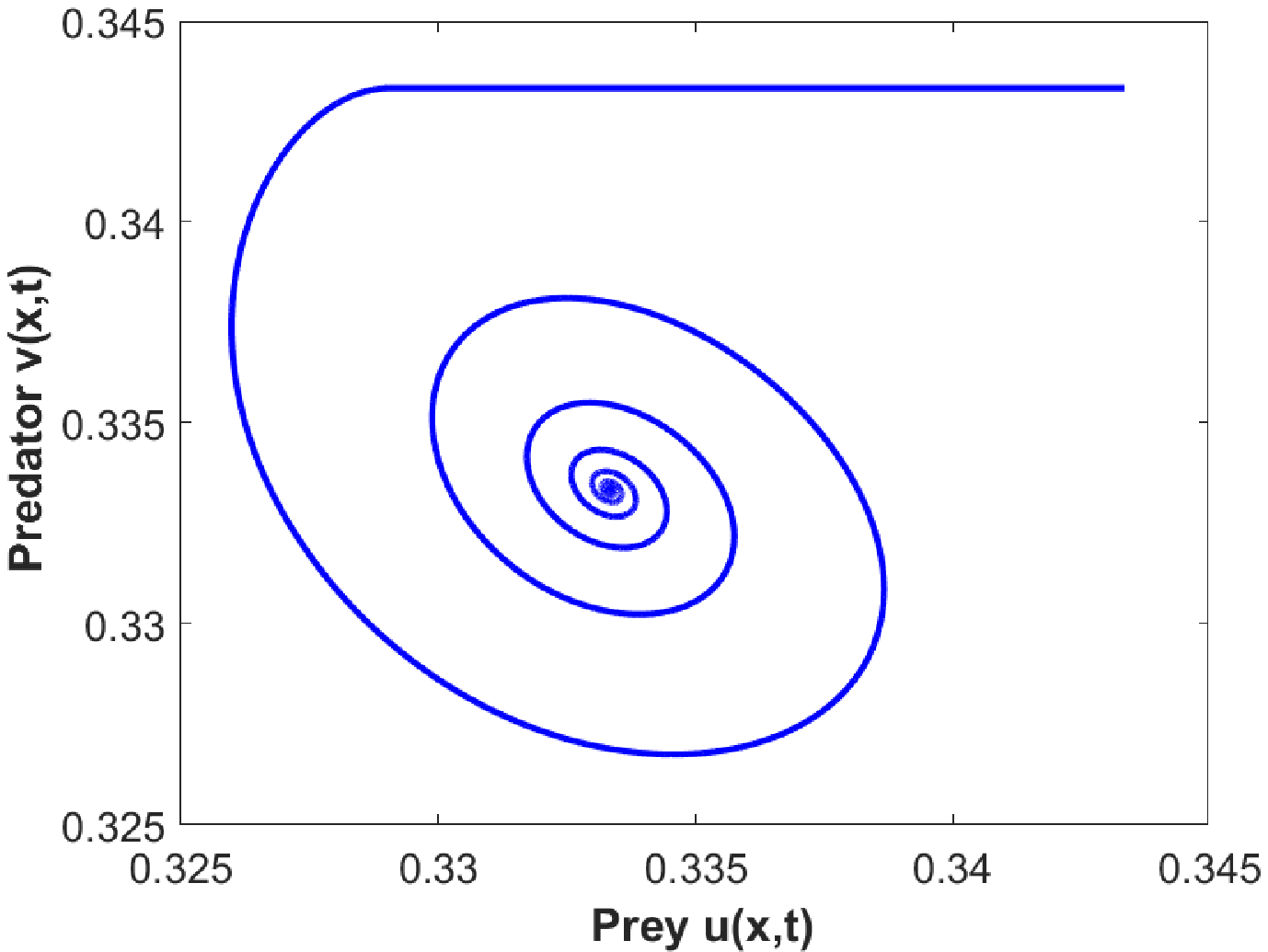} \\
\caption{For the parameters $d_{11}=0.6,~d_{22}=0.8,~m=0.5,~\gamma=0.5,~\beta=1,~\ell=2$ and $\tau=6<\tau_{0,0}=10.078$, the behavior and phase portrait of system (4.24) is shown. The initial values are $u_{0}(x)=0.3333+0.01,~v_{0}(x)=0.3333+0.01$.}
\label{fig:7}
\end{figure}

\begin{figure}[!htbp]
\centering
\includegraphics[width=2in]{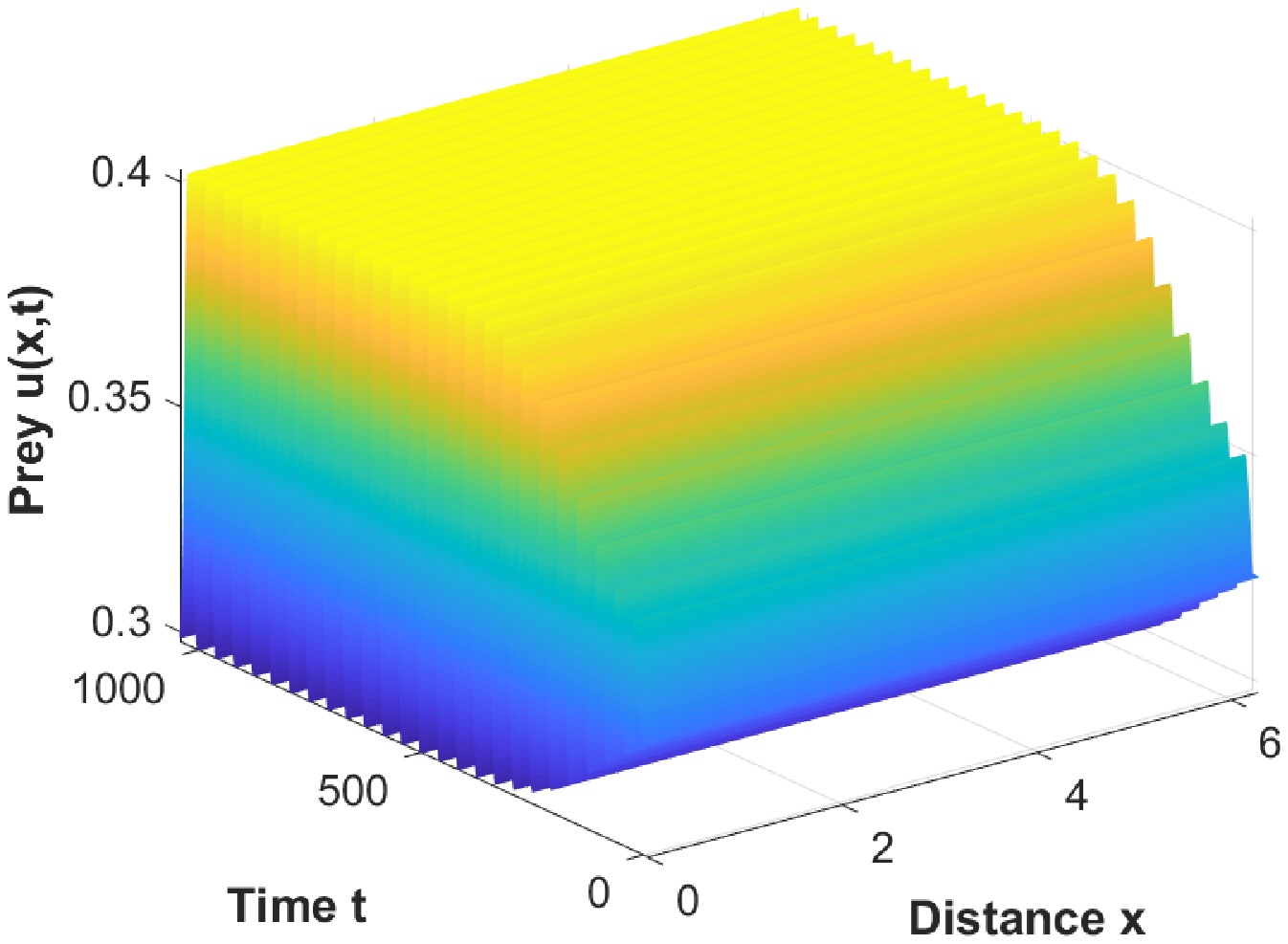}
\includegraphics[width=2in]{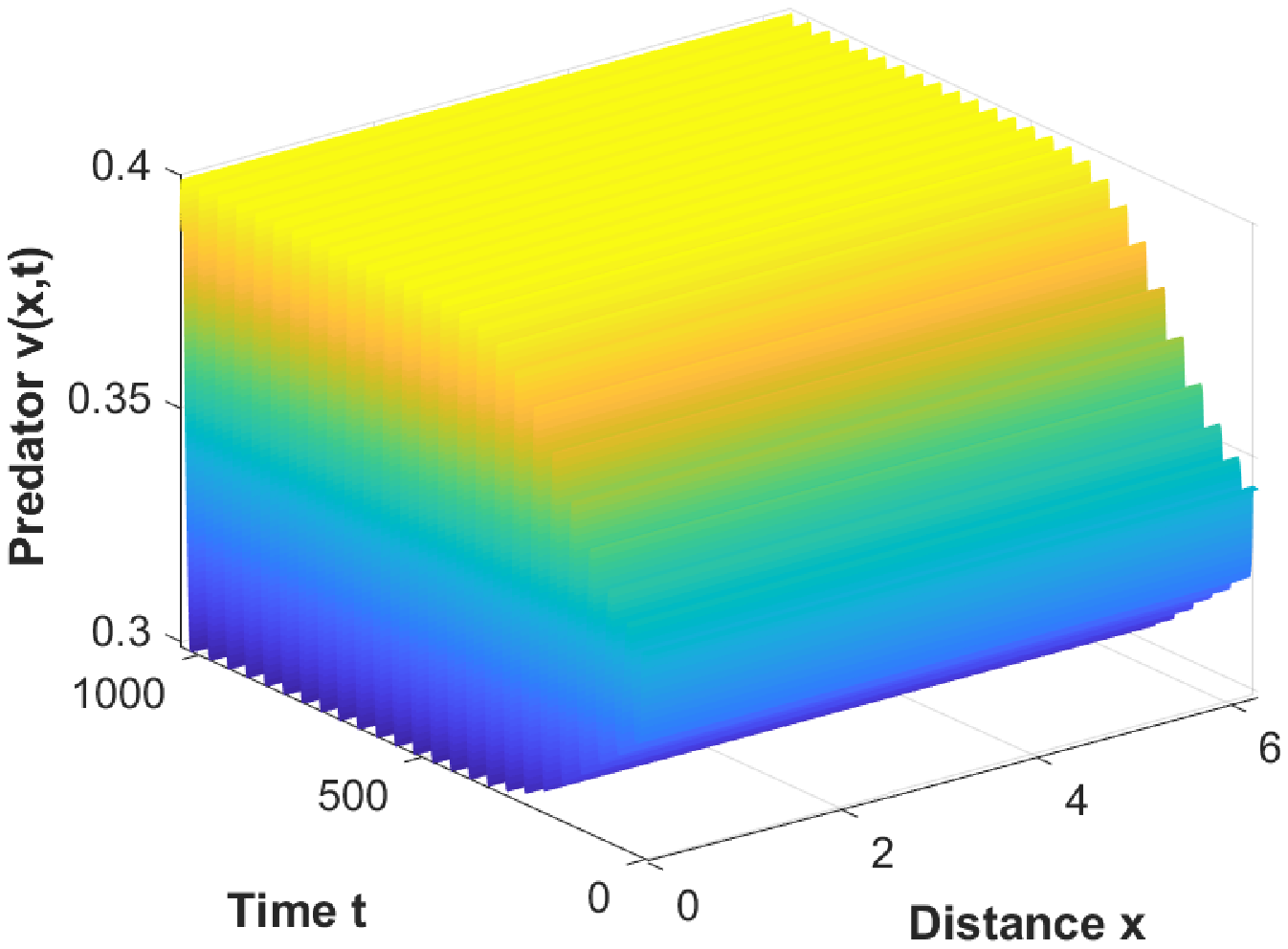} \\
\textbf{(a)} \hspace{4.5cm} \textbf{(b)} \\
\includegraphics[width=2in]{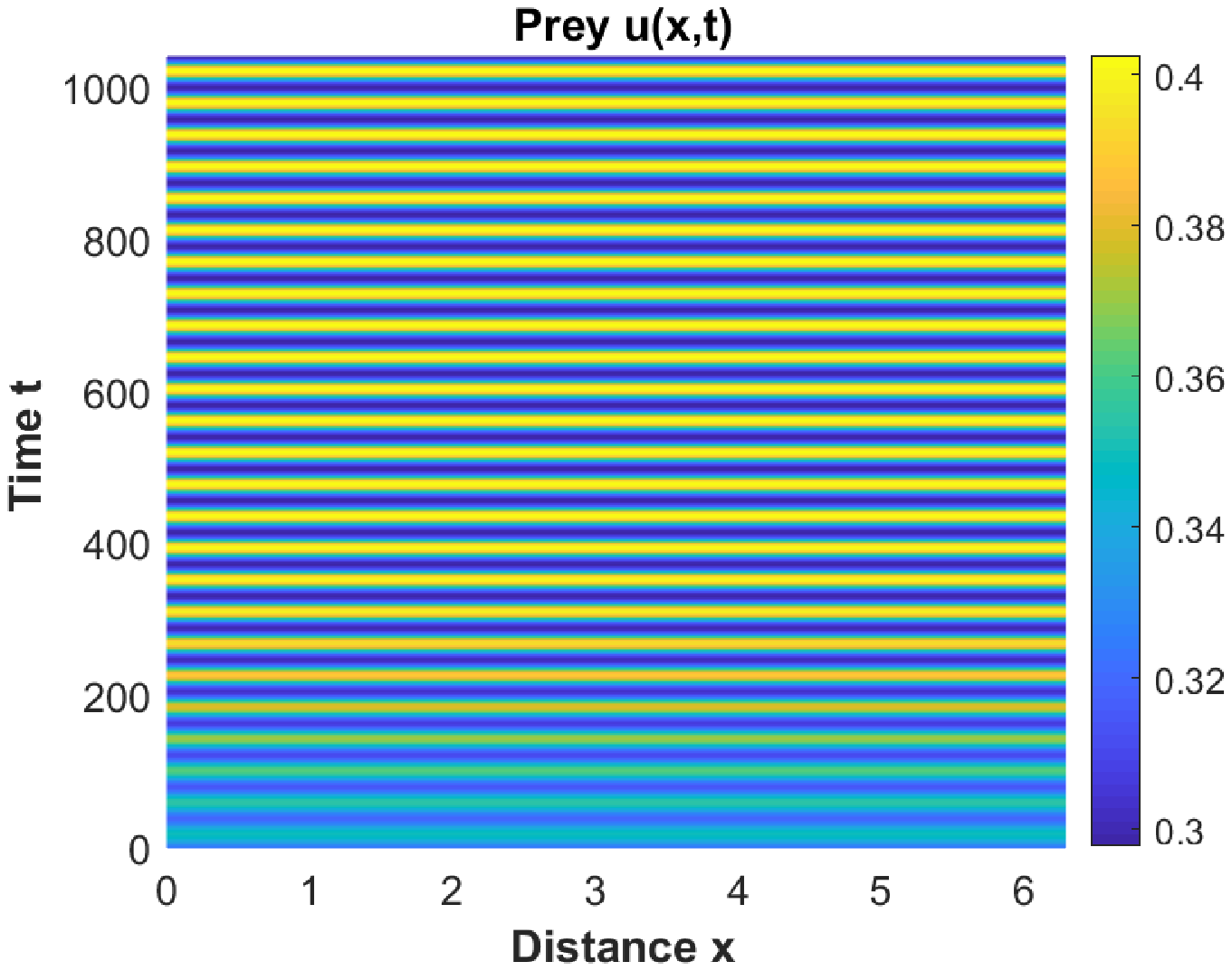}
\includegraphics[width=2in]{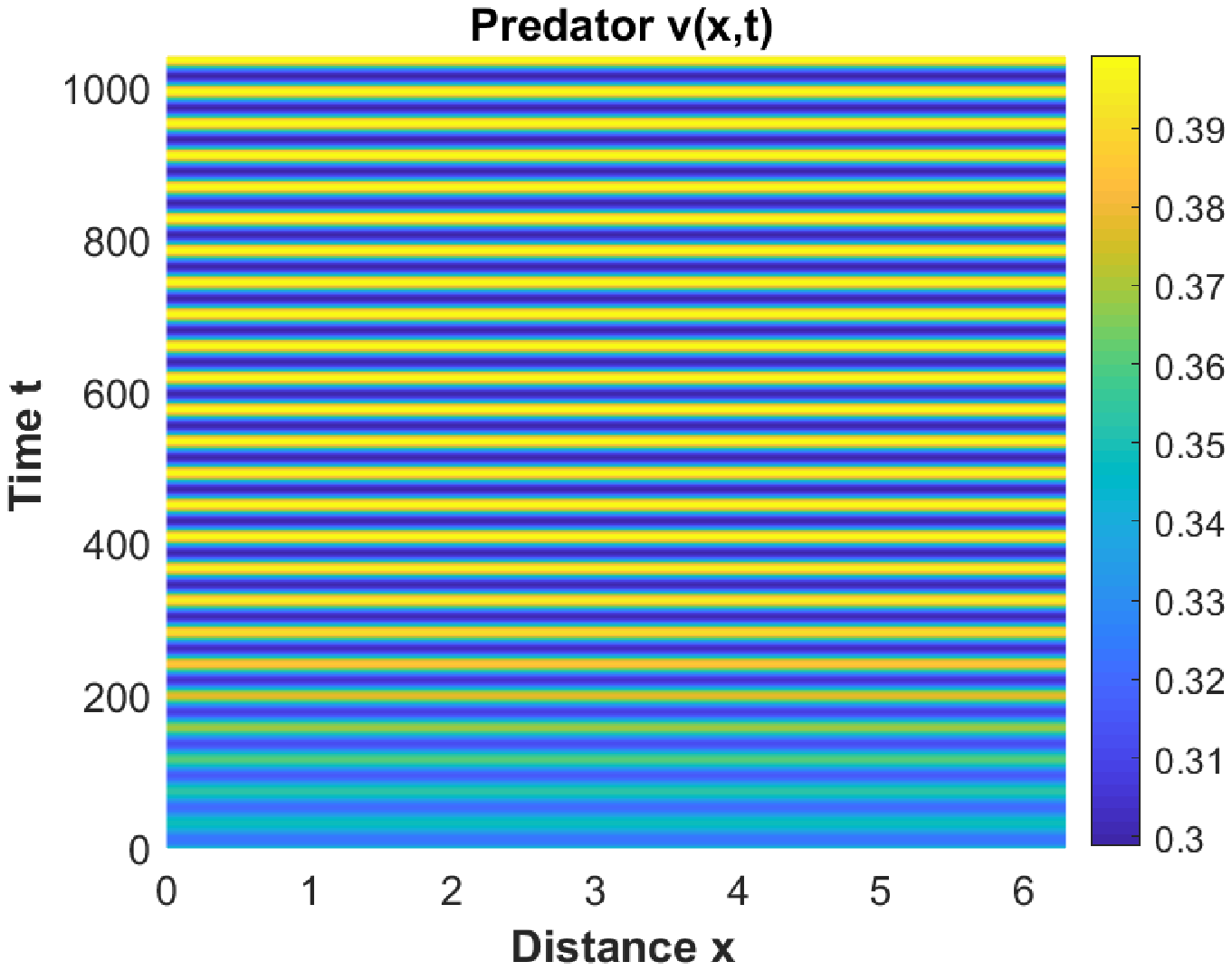} \\
\textbf{(c)} \hspace{4.5cm} \textbf{(d)} \\
\caption{For the parameters $d_{11}=0.6,~d_{22}=0.8,~m=0.5,~\gamma=0.5,~\beta=1,~\ell=2$ and $\tau=13>\tau_{0,0}=10.078$, there exists a stable spatially homogeneous periodic solution. (a) and (b) are the evolution processes of the solutions $u(x,t)$ and $v(x,t)$ of system (4.24), respectively. (c) and (d) are spatio-temporal diagrams of the solutions $u(x,t)$ and $v(x,t)$ of system (4.24), respectively. The initial values are $u_{0}(x)=0.3333-0.01,~v_{0}(x)=0.3333+0.01$.}
\label{fig:8}
\end{figure}

\begin{figure}[!htbp]
\centering
\includegraphics[width=2in]{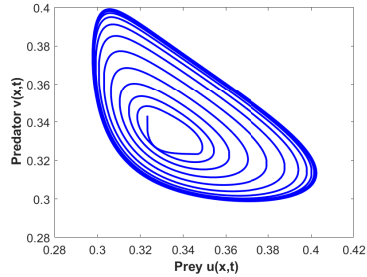} \\
\caption{For the parameters $d_{11}=0.6,~d_{22}=0.8,~m=0.5,~\gamma=0.5,~\beta=1,~\ell=2$ and $\tau=13>\tau_{0,0}=10.078$, the behavior and phase portrait of system (4.24) is shown. The initial values are $u_{0}(x)=0.3333-0.01,~v_{0}(x)=0.3333+0.01$.}
\label{fig:9}
\end{figure}

\section{Conclusion and discussion}
\label{sec:5}

In this paper, we have developed an algorithm for calculating the normal form of Hopf bifurcation in a diffusive system with memory and general delays. Since apart from the memory delay appears in the diffusion term, the general delay also occurs in the reaction term, the traditional algorithm for calculating the normal form of Hopf bifurcation in the memory-based system which without the general delays is not suitable for this system. To solve this problem, we derive an algorithm for calculating the normal form of Hopf bifurcation in a diffusive system with memory and general delays, which can be seen a generalization of the existing algorithm for the reaction-diffusion system where only the memory delay appears in the diffusion term. In order to show the effectiveness of our developed algorithm, we consider a diffusive predator-prey model with ratio-dependent Holling type-\uppercase\expandafter{\romannumeral3} functional response, which includes with memory and gestation delays. The memory and gestation delays-induced spatially homogeneous Hopf bifurcation is observed by theoretical analysis and numerical simulation.

In this paper, we assume that the memory delay and the general delay are the same. It is worth mentioning that when the memory delay and the general delay are different, i.e.,
\begin{eqnarray*}\left\{\begin{aligned}
\frac{\partial u(x,t)}{\partial t}&=d_{11} \Delta u(x,t)+f\left(u(x,t), v(x,t), u(x,t-\sigma),v(x,t-\sigma)\right), \\
\frac{\partial v(x,t)}{\partial t}&=d_{22} \Delta v(x,t)-d_{21}\left(v(x,t) u_{x}(x,t-\tau)\right)_{x}+g\left(u(x,t), v(x,t), u(x, t-\sigma),v(x, t-\sigma)\right),
\end{aligned}\right.\end{eqnarray*}
which needs further research, where $\sigma>0$ is the general delay, and $\tau=\sigma$ or $\tau \neq \sigma$.

\section*{Acknowledgments}

The author is grateful to the anonymous referees for their useful suggestions which improve the contents of this article.

\section*{Declarations}
\par\noindent This research did not involve human participants and animals.
\par\noindent Funding: This research did not receive any specific grant from funding agencies in the public, commercial, or not-for-profit sectors.
\par\noindent Conflicts of interest: The author declares that there is not conflict of interest, whether financial or non-financial.
\par\noindent Availability of data and material: This research didn't involve the private data, and the involving data and material are all available.
\par\noindent Code availability: The numerical simulations in this paper are carried by using the MATLAB software.
\par\noindent Authors' contributions: This manuscript is investigated and written by Yehu Lv.

\section*{Appendix A}
\setcounter{equation}{0}
\renewcommand\theequation{A.\arabic{equation}}

\begin{remark}
Assume that at $\tau=\tau_{c}$, (4.6) has a pair of purely imaginary roots $\pm i \omega_{n_{c}}$ with $\omega_{n_{c}}>0$ for $n=n_{c} \in \mathbb{N}$ and all other eigenvalues have negative real part. Let $\lambda(\tau)=\alpha_{1}(\tau) \pm i \alpha_{2}(\tau)$ be a pair of roots of (4.6) near $\tau=\tau_{c}$ satisfying $\alpha_{1}(\tau_{c})=0$ and $\alpha_{2}(\tau_{c})=\omega_{n_{c}}$. In addition, the corresponding transversality condition holds.
\end{remark}

The normal form of Hopf bifurcation for the system (4.2) can be calculated by using the developed algorithm in \cite{lv30}. Here, we give the detail calculation procedures of $B_{1}, B_{21}, B_{22}, B_{23}$ steps by steps.
\begin{enumerate}[{\bf{Step 1:}}]
\item
\begin{eqnarray*}
B_{1}=2\psi^{T}(0)\left(A_{1}\phi(0)-\frac{n_{c}^{2}}{\ell^{2}}\left(D_{1}\phi(0)+D_{2}\phi(-1)\right)\right)
\end{eqnarray*}
with
\begin{eqnarray*}
D_{1}=\left(\begin{array}{cc}
d_{11} & 0 \\
0 & d_{22}
\end{array}\right),~D_{2}=\left(\begin{array}{cc}
0 & 0 \\
-d_{21} v_{*} & 0
\end{array}\right),~A_{1}=\left(\begin{array}{cc}
\frac{2 \beta}{(m+1)^{2}}-1 & \frac{\beta(m-1)}{(m+1)^{2}} \\
\gamma & -\gamma
\end{array}\right).
\end{eqnarray*}
Here,
\begin{eqnarray*}
\phi=\left(\begin{array}{c}
1 \\
\frac{i\omega_{n_{c}}+\left(n_{c}/\ell\right)^{2}d_{11}-a_{11}}{a_{12}}
\end{array}\right),~\psi=\eta\left(\begin{array}{c}
1 \\
\frac{a_{12}}{i\omega_{n c}+\left(n_{c}/\ell\right)^{2}d_{22}-a_{22}}
\end{array}\right)
\end{eqnarray*}
with
\begin{eqnarray*}
\eta=\frac{i\omega_{n_{c}}+\left(n_{c}/\ell\right)^{2} d_{22}-a_{22}}{2i\omega_{n_{c}}+\left(n_{c}/\ell\right)^{2}d_{11}-a_{11}+\left(n_{c}/\ell\right)^{2}d_{22}-a_{22}+\tau_{c}a_{12}d_{21} v_{*}\left(n_{c}/\ell\right)^{2}e^{-i\omega_{c}}}.
\end{eqnarray*}
\end{enumerate}

\begin{enumerate}[{\bf{Step 2:}}]
\item
\begin{eqnarray*}
B_{21}=\frac{3}{2\ell\pi}\psi^{T}A_{21}
\end{eqnarray*}
with
\begin{eqnarray*}\begin{aligned}
A_{21}&=3f_{30}\phi_{1}^{2}(0)\overline{\phi_{1}}(0)+3f_{03}\phi_{2}^{2}(0)\overline{\phi_{2}}(0)+3f_{21}(\phi_{1}^{2}(0)\overline{\phi_{2}}(0)+2\phi_{1}(0)\overline{\phi_{1}}(0)\phi_{2}(0)) \\
&+3f_{12}(2\phi_{1}(0)\phi_{2}(0)\overline{\phi_{2}}(0)+\overline{\phi_{1}}(0)\phi_{2}^{2}(0)).
\end{aligned}\end{eqnarray*}
Here,
\begin{eqnarray*}\begin{aligned}
f^{(1)}_{03}&=6\tau_{c}\beta m u_{*}^{2}(u_{*}^{2}+mv_{*}^{2})^{-2}-48\tau_{c}\beta m^{2}u_{*}^{2}v_{*}^{2}(u_{*}^{2}+mv_{*}^{2})^{-3}+48\tau_{c}\beta m^{3}u_{*}^{2}v_{*}^{4}(u_{*}^{2}+mv_{*}^{2})^{-4}, \\
f^{(2)}_{03}&=0, \\
f^{(1)}_{12}&=12\tau_{c}\beta m u_{*}v_{*}(u_{*}^{2}+mv_{*}^{2})^{-2}-16\tau_{c}\beta m^{2}u_{*}v_{*}^{3}(u_{*}^{2}+mv_{*}^{2})^{-3},~f^{(2)}_{12}=2\tau_{c}\gamma u_{*}^{-2}, \\
f^{(1)}_{21}&=-2\tau_{c}\beta(u_{*}^{2}+mv_{*}^{2})^{-1}+4\tau_{c}\beta mv_{*}^{2}(u_{*}^{2}+mv_{*}^{2})^{-2}+10\tau_{c}\beta u_{*}^{2}(u_{*}^{2}+mv_{*}^{2})^{-2} \\
&-40\tau_{c}\beta m u_{*}^{2}v_{*}^{2}(u_{*}^{2}+mv_{*}^{2})^{-3}-8\tau_{c}\beta u_{*}^{4}(u_{*}^{2}+mv_{*}^{2})^{-3}+48\tau_{c}\beta m u_{*}^{4}v_{*}^{2}(u_{*}^{2}+mv_{*}^{2})^{-4}, \\
f^{(2)}_{21}&=-4\tau_{c}\gamma v_{*}u_{*}^{-3}, \\
f^{(1)}_{30}&=24\tau_{c}\beta u_{*}v_{*}(u_{*}^{2}+mv_{*}^{2})^{-2}-72\tau_{c}\beta u_{*}^{3}v_{*}(u_{*}^{2}+mv_{*}^{2})^{-3}+48\tau_{c}\beta u_{*}^{5}v_{*}(u_{*}^{2}+mv_{*}^{2})^{-4}, \\
f^{(2)}_{30}&=6\tau_{c}\gamma v_{*}^{2}u_{*}^{-4}.
\end{aligned}\end{eqnarray*}
\end{enumerate}

\begin{enumerate}[{\bf{Step 3:}}]
\item
\begin{eqnarray*}\begin{aligned}
B_{22}&=\frac{1}{\sqrt{\ell\pi}}\psi^{T}\left(S_{2}\left(\phi(\theta),h_{0,11}(\theta)\right)+S_{2}\left(\overline{\phi}(\theta), h_{0,20}(\theta)\right)\right) \\
&+\frac{1}{\sqrt{2\ell\pi}} \psi^{T}\left(S_{2}\left(\phi(\theta),h_{2n_{c},11}(\theta)\right)+S_{2}\left(\overline{\phi}(\theta),h_{2n_{c},20}(\theta)\right)\right)
\end{aligned}\end{eqnarray*}
with
\begin{eqnarray*}\begin{aligned}
S_{2}\left(\phi(\theta),h_{0,11}(\theta)\right)&=2f_{20}\phi_{1}(0)h^{(1)}_{0,11}(0)+2f_{02}\phi_{2}(0)h^{(2)}_{0,11}(0) \\
&+2f_{11}\left(\phi_{1}(0)h^{(2)}_{0,11}(0)+\phi_{2}(0)h^{(1)}_{0,11}(0)\right), \\
S_{2}\left(\overline{\phi}(\theta), h_{0,20}(\theta)\right)&=2f_{20}\overline{\phi}_{1}(0)h^{(1)}_{0,20}(0)+2f_{02}\overline{\phi}_{2}(0)h^{(2)}_{0,20}(0) \\
&+2f_{11}\left(\overline{\phi}_{1}(0)h^{(2)}_{0,20}(0)+\overline{\phi}_{2}(0)h^{(1)}_{0,20}(0)\right), \\
S_{2}\left(\phi(\theta), h_{2n_{c},11}(\theta)\right)&=2f_{20}\phi_{1}(0)h^{(1)}_{2n_{c},11}(0)+2f_{02}\phi_{2}(0)h^{(2)}_{2n_{c},11}(0) \\
&+2f_{11}\left(\phi_{1}(0)h^{(2)}_{2n_{c},11}(0)+\phi_{2}(0)h^{(1)}_{2n_{c},11}(0)\right), \\
S_{2}\left(\overline{\phi}(\theta), h_{2n_{c},20}(\theta)\right)&=2f_{20}\overline{\phi}_{1}(0)h^{(1)}_{2n_{c},20}(0)+2f_{02}\overline{\phi}_{2}(0)h^{(2)}_{2n_{c},20}(0) \\
&+2f_{11}\left(\overline{\phi}_{1}(0)h^{(2)}_{2n_{c},20}(0)+\overline{\phi}_{2}(0)h^{(1)}_{2n_{c},20}(0)\right).
\end{aligned}\end{eqnarray*}
Here,
\begin{eqnarray*}\begin{aligned}
f^{(1)}_{02}&=6\tau_{c}\beta m u_{*}^{2}v_{*}(u_{*}^{2}+mv_{*}^{2})^{-2}-8\tau_{c}\beta m^{2}u_{*}^{2}v_{*}^{3}(u_{*}^{2}+mv_{*}^{2})^{-3},~f^{(2)}_{02}=-2\tau_{c}\gamma u_{*}^{-1}, \\
f^{(1)}_{11}&=-2\tau_{c}\beta u_{*}(u_{*}^{2}+mv_{*}^{2})^{-1}+4\tau_{c}\beta m u_{*}v_{*}^{2}(u_{*}^{2}+mv_{*}^{2})^{-2},~f^{(2)}_{11}=2\tau_{c}\gamma u_{*}^{-2}v_{*}, \\
f^{(1)}_{20}&=-2\tau_{c}-2\tau_{c}\beta v_{*}(u_{*}^{2}+mv_{*}^{2})^{-1}+10\tau_{c}\beta u_{*}^{2}v_{*}(u_{*}^{2}+mv_{*}^{2})^{-2}-8\tau_{c}\beta u_{*}^{4}v_{*}(u_{*}^{2}+mv_{*}^{2})^{-3}, \\
f^{(2)}_{20}&=-2\tau_{c}\gamma u_{*}^{-3}v_{*}^{2}.
\end{aligned}\end{eqnarray*}
Furthermore, we have
\begin{eqnarray*}
\left\{\begin{aligned}
h_{0,20}(\theta)&=\frac{1}{\sqrt{\ell\pi}}\left(\widetilde{M}_{0}\left(2i\omega_{c}\right)\right)^{-1}A_{20}e^{2i\omega_{c}\theta}, \\
h_{0,11}(\theta)&=\frac{1}{\sqrt{\ell\pi}}\left(\widetilde{M}_{0}(0)\right)^{-1}A_{11}
\end{aligned}\right.
\end{eqnarray*}
and
\begin{eqnarray*}
\left\{\begin{aligned}
h_{2n_{c},20}(\theta)&=\frac{1}{\sqrt{2\ell\pi}}\left(\widetilde{M}_{2n_{c}}\left(2i\omega_{c}\right)\right)^{-1}\widetilde{A}_{20}e^{2i\omega_{c}\theta} \\
h_{2n_{c},11}(\theta)&=\frac{1}{\sqrt{2\ell\pi}}\left(\widetilde{M}_{2n_{c}}(0)\right)^{-1}\widetilde{A}_{11}
\end{aligned}\right.
\end{eqnarray*}
with
\begin{eqnarray*}
\widetilde{M}_{n}(\lambda)=\lambda I_{2}+\tau_{c}(n/\ell)^{2}D_{1}+\tau_{c}(n/\ell)^{2}e^{-\lambda}D_{2}-\tau_{c}A_{1}.
\end{eqnarray*}
Here,
\begin{eqnarray*}\begin{aligned}
A_{20}&=f_{20}\phi_{1}^{2}(0)+f_{02}\phi_{2}^{2}(0)+2f_{11}\phi_{1}(0)\phi_{2}(0), \\
A_{11}&=2f_{20}\phi_{1}(0)\overline{\phi_{1}}(0)+2f_{02}\phi_{2}(0)\overline{\phi_{2}}(0)+2f_{11}(\phi_{1}(0)\overline{\phi_{2}}(0)+\overline{\phi_{1}}(0)\phi_{2}(0))
\end{aligned}\end{eqnarray*}
and
\begin{eqnarray*}
\left\{\begin{aligned}
\widetilde{A}_{20}&=A_{20}-2\left(n_{c}/\ell\right)^{2}A_{20}^{d}, \\
\widetilde{A}_{11}&=A_{11}-2\left(n_{c}/\ell\right)^{2}A_{11}^{d}
\end{aligned}\right.
\end{eqnarray*}
with
\begin{eqnarray*}\begin{aligned}
&\left\{\begin{array}{l}
A_{20}^{d}=-2d_{21}\tau_{c}\left(\begin{array}{c}
0 \\
\phi_{1}(-1)\phi_{2}(0)
\end{array}\right)=\overline{A_{02}^{d}}, \\
A_{11}^{d}=-2 d_{21}\tau_{c}\left(\begin{array}{c}
0 \\
2 \operatorname{Re}\left\{\phi_{1}(-1)\overline{\phi_{2}}(0)\right\}
\end{array}\right).
\end{array}\right.
\end{aligned}\end{eqnarray*}
\end{enumerate}

\begin{enumerate}[{\bf{Step 4:}}]
\item
\begin{eqnarray*}\begin{aligned}
B_{23}=&-\frac{1}{\sqrt{\ell\pi}}\left(n_{c}/\ell\right)^{2} \psi^{T}\left(S_{2}^{(d,1)}\left(\phi(\theta),h_{0,11}(\theta)\right)+S_{2}^{(d,1)}\left(\overline{\phi}(\theta),h_{0,20}(\theta)\right)\right) \\
&+\frac{1}{\sqrt{2\ell\pi}}\psi^{T}\sum_{j=1,2,3} b_{2n_{c}}^{(j)}\left(S_{2}^{(d,j)}\left(\phi(\theta),h_{2n_{c},11}(\theta)\right)+S_{2}^{(d,j)}\left(\overline{\phi}(\theta),h_{2n_{c},20}(\theta)\right)\right)
\end{aligned}\end{eqnarray*}
with
\begin{eqnarray*}
b_{2n_{c}}^{(1)}=-\frac{n_{c}^{2}}{\ell^{2}},~b_{2n_{c}}^{(2)}=\frac{2n_{c}^{2}}{\ell^{2}},~b_{2n_{c}}^{(3)}=-\frac{(2n_{c})^{2}}{\ell^{2}}
\end{eqnarray*}
and
\begin{eqnarray*}
\left\{\begin{aligned}
&S_{2}^{(d,1)}\left(\phi(\theta),h_{0,11}(\theta)\right)=-2 d_{21}\tau_{c}\left(\begin{array}{c}
0 \\
\phi_{1}(-1)h^{(2)}_{0,11}(0)
\end{array}\right), \\
&S_{2}^{(d,1)}\left(\overline{\phi}(\theta),h_{0,20}(\theta)\right)=-2d_{21}\tau_{c}\left(\begin{array}{c}
0 \\
\overline{\phi}_{1}(-1)h^{(2)}_{0,20}(0)
\end{array}\right), \\
&S_{2}^{(d,1)}\left(\phi(\theta),h_{2n_{c},11}(\theta)\right)=-2d_{21}\tau_{c}\left(\begin{array}{c}
0 \\
\phi_{1}(-1)h^{(2)}_{2n_{c},11}(0)
\end{array}\right), \\
&S_{2}^{(d,2)}(\phi(\theta),h_{2n_{c},11}(\theta))=-2d_{21}\tau_{c}\left(\begin{array}{c}
0 \\
\phi_{1}(-1) h^{(2)}_{2n_{c},11}(0)+\phi_{2}(0)h^{(1)}_{2n_{c},11}(-1)
\end{array}\right), \\
&S_{2}^{(d,3)}(\phi(\theta),h_{2n_{c},11}(\theta))=-2d_{21}\tau_{c}\left(\begin{array}{c}
0 \\
\phi_{2}(0) h^{(1)}_{2n_{c},11}(-1)
\end{array}\right), \\
&S_{2}^{(d,1)}\left(\overline{\phi}(\theta),h_{2n_{c},20}(\theta)\right)=-2d_{21}\tau_{c}\left(\begin{array}{c}
0 \\
\overline{\phi}_{1}(-1) h^{(2)}_{2n_{c},20}(0)
\end{array}\right), \\
&S_{2}^{(d,2)}(\overline{\phi}(\theta),h_{2n_{c},20}(\theta))=-2d_{21}\tau_{c}\left(\begin{array}{c}
0 \\
\overline{\phi}_{1}(-1)h^{(2)}_{2n_{c},20}(0)+\overline{\phi}_{2}(0)h^{(1)}_{2n_{c},20}(-1)
\end{array}\right), \\
&S_{2}^{(d,3)}(\overline{\phi}(\theta),h_{2n_{c},20}(\theta))=-2d_{21}\tau_{c}\left(\begin{array}{c}
0 \\
\overline{\phi}_{2}(0)h^{(1)}_{2n_{c},20}(-1)
\end{array}\right).
\end{aligned}\right.
\end{eqnarray*}
\end{enumerate}

\end{document}